\documentclass[10pt]{amsart}
\usepackage{amsmath}
\usepackage{amsfonts}
\usepackage{amsthm}
\usepackage{amssymb}

\newtheorem{theorem}{Theorem}[section]
\newtheorem*{theorem*}{Theorem}
\newtheorem*{acknowledgement*}{Acknowledgement}
\newtheorem*{definition*}{Definition}

\newtheorem{claim}[theorem]{Claim}

\newtheorem{corollary}[theorem]{Corollary}

\newtheorem{definition}[theorem]{Definition}

\newtheorem{lemma}[theorem]{Lemma}

\newtheorem{proposition}[theorem]{Proposition}
\newtheorem{remark}[theorem]{Remark}

\numberwithin{equation}{section}

\newcommand{\RR}[0]{\mathbb{R}}

\newcommand{\pd}[2]{\frac{\partial #1}{\partial#2}}

\newcommand{\pdt}[0]{\frac{\partial}{\partial t}}

\newcommand{\mc}[1]{{\mathcal{#1}}}
\newcommand{\gt}[0]{\tilde{g}}

\newcommand{\del}[0]{\nabla}

\newcommand{\delt}[0]{\widetilde{\nabla}}

\newcommand{\Gamt}[0]{\widetilde{\Gamma}}

\newcommand{\Rc}[0]{\operatorname{Rc}}
\newcommand{\Rm}[0]{\operatorname{Rm}}

\newcommand{\Rmt}[0]{\widetilde{\operatorname{Rm}}}
\newcommand{\dfn}[0]{\doteqdot}
\newcommand{\Xb}[0]{\mathbf{X}}
\newcommand{\Yb}[0]{\mathbf{Y}}
\newcommand{\Xc}[0]{\mathcal{X}}
\newcommand{\Yc}[0]{\mathcal{Y}}

\newcommand{\pdtau}[0]{\pd{}{\tau}}
\newcommand{\srad}[0]{\mathcal{S}}

\newcommand{\abs}[1]{\left\vert#1\right\vert}

\newcommand{\lam}{\lambda}

\newcommand{\ra}{\rangle}
\newcommand{\la}{\langle}

\newcommand{\rad}{\mathcal R}

\newcommand{\Ec}{\mathcal E}
\newcommand{\mcA}[0]{\mathcal{A}}
\newcommand{\mcS}[0]{\mathcal{S}}
\newcommand{\Id}[0]{\operatorname{Id}}

\newcommand{\length}{\operatorname{length}}
\newcommand{\dist}{d}

\title[Rigidity of asymptotically conical shrinking solitons]{Rigidity of asymptotically conical shrinking gradient Ricci solitons}

\author{Brett Kotschwar}
\address{Arizona State University, Tempe, AZ, 85287}
\email{kotschwar@asu.edu}

\author{Lu Wang}
\address{Johns Hopkins University, Baltimore, MD, 21218}
\email{lwang58@jhu.edu}

\thanks{The first author was supported in part by NSF grant DMS-1160613.  The second author was supported in part by an AMS-Simons Travel Grant.}

\date{March 2013}

\begin{document}
 \begin{abstract}
We show that if two gradient Ricci solitons
are asymptotic along some end of each to the same regular cone $((0, \infty)\times \Sigma, dr^2 + r^2g_{\Sigma})$, 
then the soliton metrics must be isometric on some neighborhoods of infinity of these ends. 
Our theorem imposes no restrictions on the behavior of the metrics off of the ends in question
and in particular does not require their geodesic completeness.  
As an application, we prove that the only complete connected gradient shrinking Ricci soliton asymptotic to 
a rotationally symmetric cone is the Gaussian soliton on $\RR^n$. 
\end{abstract}
\maketitle

\section{Introduction}
In this paper, by a \emph{shrinking (gradient) Ricci soliton structure}, we will mean a triple $(M, g, f)$ 
consisting of a smooth manifold $M$, a Riemannian metric $g$, and a smooth function $f$ satisfying the equations 
\begin{equation}\label{eq:shrinker}
  \Rc(g) + \nabla\nabla f= \frac{1}{2} g \quad\mbox{and}\quad   R + |\nabla f|^2 = f
\end{equation}
on $M$.  Since $\nabla (R + |\nabla f|^2 - f) \equiv 0$ whenever $g$ and $f$ satisfy the first equation, the second equation
 is merely a convenient normalization   
and can be achieved by adding an appropriate constant to $f$ on every connected component of $M$.  When the potential
is well-known or can be determined from context, we often will refer simply to the metric $g$ as the soliton (or the \emph{shrinker}) on $M$.

Beyond their intrinsic interest as generalizations of positive Einstein metrics,
shrinking solitons occupy a prominent place in the analysis of singularities of the Ricci flow
\begin{equation}\label{eq:rf}
 \pdt g(t) = -2\Rc(g(t)),
\end{equation}
where they correspond to shrinking \emph{self-similar solutions} -- the fixed points of the equation 
modulo the actions of
$\operatorname{Diff}(M)$ and $\RR_{+}$ on the space of metrics on $M$. 
They are the critical cases in Perelman's entropy monotonicity formula and an important class of ancient solutions, 
arising frequently in applications as limits of rescalings of solutions to \eqref{eq:rf} 
about developing singularities.

It is a fundamental problem to extend the classification of shrinking solitons, which, at present, is
only fully complete in dimensions two and three. Hamilton \cite{HamiltonSurfaces} proved
that the only complete nonflat two-dimensional shrinking solitons are the standard round metrics on $S^2$ and $\RR P^2$, and, with a combination of results from his later paper \cite{HamiltonSingularities}
and the work of Ivey \cite{Ivey3DSolitons}, Perelman \cite{Perelman1}, Ni-Wallach \cite{NiWallach}, and Cao-Chen-Zhu \cite{CaoChenZhu},
it follows that the only nonflat complete three-dimensional examples 
are quotients of either  
the standard round or standard cylindrical metrics on $S^3$ and $\RR \times S^2$, respectively. 

In higher dimensions, there are a number of partial classifications for solitons satisfying certain auxiliary (and typically pointwise)
conditions on the curvature tensor. For example, Naber \cite{Naber4D} has shown that a four-dimensional complete noncompact nonflat shrinker 
of bounded nonnegative curvature
operator must be a quotient of the standard solitons on $\RR \times S^3$ or $\RR^2 \times S^2$, and the theorem
of B\"ohm-Wilking \cite{BohmWilking} implies that compact shrinkers with two-positive curvature
operator must be spherical space forms. Also, from \cite{CaoWangZhang}, \cite{EminentiLaNaveMantegazza},
\cite{FernandezLopezGarciaRio}, \cite{MunteanuSesum}, \cite{NiWallach}, \cite{PetersenWylie}, and \cite{ZhangWeyl}, it follows that the only complete
nonflat shrinking solitons of vanishing Weyl tensor (even harmonic Weyl tensor) in dimensions $n \geq 4$ 
are finite quotients of the standard metrics on $S^n$ or $S^{n-1}\times \RR$. 
A further classification, under the still weaker condition of vanishing Bach tensor, can be found in \cite{CaoChen}. We refer the reader to the two surveys \cite{CaoSurvey1}, \cite{CaoSurvey2}
of Cao for a detailed picture of the current state of the art. 

Our specific interest is in complete noncompact shrinking Ricci solitons.  Here, one might optimistically interpret
the sharp estimates now known to hold on the growth of the potential $f$ \cite{CaoZhou} 
and the volume of metric balls (see, e.g., \cite{CaoZhou}, \cite{CarrilloNi},  \cite{MunteanuWang}) as indicators
of an enforcement of some broader principle of asymptotic rigidity, however, the catalog of nontrivial examples is still exceedingly slim.
  Excluding products and otherwise locally reducible metrics, 
to the authors' knowledge, the only complete noncompact examples in the literature belong either to the family 
of K\"ahler-Ricci solitons on complex line bundles
constructed by Feldman-Ilmanen-Knopf \cite{FeldmanIlmanenKnopf} or to those of their generalizations 
in Dancer-Wang \cite{DancerWang} (see also \cite{Yang}). Both of these families possess conical
structures at infinity, and it is their example which motivates the investigation of the
rigidity of such asymptotic structures in this paper.
We approach this as a question of uniqueness: \emph{if two
gradient shrinking solitons are asymptotic to the same cone along some end of each, must they be isometric on some neighborhoods
of infinity on those ends?} 

\subsection{Asymptotically conical shrinking Ricci solitons}

We now make precise the sense in which we will understand a soliton to be asymptotic to a cone. 
First let us make a preliminary definition and fix some notation. By an \emph{end} of $M$, we will mean a connected unbounded component $V$ of $M\setminus \Omega$ for some
compact $\Omega \subset M$.  We will denote by $((0,\infty)\times\Sigma, g_c)$ a regular (i.e., Euclidean) cone, where 
$g_c=dr^2+r^2g_\Sigma$ and 
$(\Sigma, g_{\Sigma})$ is a closed $(n-1)$-dimensional Riemannian manifold, and write $E_{\rad} \dfn (\rad,\infty)\times\Sigma$
for $\rad \geq 0$. Finally, for $\lambda > 0$, we
define the \emph{dilation by $\lam$} to be the map $\rho_\lam:E_0\to E_0$ given by $\rho_\lam(r,\sigma)\dfn (\lam r, \sigma)$.

\begin{definition}
\label{def:asymcone} Let $V$ be an end of $M$.
We say that $(M, g)$ is \emph{asymptotic to the regular cone $(E_0,g_c)$ along $V$}
if, for some $\rad > 0$, there is a diffeomorphism $\Phi:E_\rad\to V$ such that $\lam^{-2}\rho_\lam^*\Phi^\ast g\to g_c$
as $\lam \to \infty$
in $C^2_{\emph{loc}}(E_0, g_c)$.
We will say that the soliton $(M, g, f)$ is asymptotic to $(E_0, g_c)$ along $V$ if $(M, g)$ is. 
\end{definition}
Our main result is the following theorem. Note that neither $(\bar{M}, \bar{g})$ nor $(\hat{M}, \hat{g})$
is assumed to be complete, and no restriction is made on the topology or geometry of $(\bar{M}, \bar{g})$ and 
$(\hat{M}, \hat{g})$ off of the ends in question.
\begin{theorem}
\label{thm:unique}
Suppose that $(\bar{M}, \bar{g},\bar{f})$ and $(\hat{M},\hat{g},\hat{f})$ 
are shrinking gradient Ricci solitons that are asymptotic 
to the regular cone $(E_0, g_c)$ along the ends $\bar{V}\subset \bar{M}$ and $\hat{V}\subset\hat{M}$,
respectively. Then there exist ends $\bar{W}\subset \bar{V}$ and $\hat{W}\subset \hat{V}$
and a diffeomorphism $\Psi:\bar{W}\to\hat{W}$ such that 
$\Psi^\ast\hat{g}=\bar{g}$.
\end{theorem}

Together with the local analyticity of Ricci solitons 
\cite{IveyLocalSoliton} and a standard monodromy argument (see, e.g., 
Theorem 3 of \cite{Myers} or Corollary 6.4 of \cite{KobayashiNomizu}), 
Theorem \ref{thm:unique} implies the following global statement. 
\begin{corollary}
\label{cor:rigidity}
Suppose $(\bar{M},\bar{g},\bar{f})$ and $(\hat{M},\hat{g},\hat{f})$
are complete gradient shrinking Ricci solitons, and $\bar{g}_0$ and $\hat{g}_0$
are the metrics induced by $\bar{g}$ and $\hat{g}$ on the universal covers $\bar{M}_0$ and $\hat{M}_0$
of $\bar{M}$ and $\hat{M}$, respectively. Then, if $(\bar{M},\bar{g},\bar{f})$ and $(\hat{M},\hat{g},\hat{f})$ 
are asymptotic to the same regular cone along some end of each, $(\bar{M}_0, \bar{g}_0)$ and $(\hat{M}_0, \hat{g}_0)$
must be isometric.
\end{corollary}

Theorem \ref{thm:unique} can also be used to rule out the possibility of nontrivial
complete shrinking solitons asymptotic to a rotationally symmetric cone. As we prove in Appendix 
\ref{app:rotsymend},  for each $\alpha \in (0, \infty)$,  there exists a rotationally symmetric 
shrinking gradient Ricci soliton $((0, \infty)\times S^{n-1}, g_{\alpha}, f_{\alpha})$ 
asymptotic to the rotationally symmetric cone $((0, \infty)\times S^{n-1} , dr^2 + \alpha r^2g_{S^{n-1}})$.
By Theorem \ref{thm:unique}, if $(M, g, f)$ is any complete shrinking gradient Ricci soliton asymptotic to the same
cone on some end $V\subset M$, there exists an isometry $\varphi: (V^{\prime}, g) \to (E^{\prime}, g_{\alpha})$
between some ends $V^{\prime} \subset V$ and $E^{\prime}\subset (0, \infty)\times S^{n-1}$.
But $g$ is then rotationally symmetric (and so also locally conformally flat) on $V^{\prime}$.  Appealing to analyticity,
we may then argue in dimensions $n\geq 4$ that the Weyl curvature tensor vanishes identically on $M$. From 
the aforementioned classification theorems in dimensions two and three and the locally conformally flat case, it follows that $(M, g)$ must be flat.
\begin{corollary}
\label{cor:nonexistencerotsym}
A complete connected shrinking gradient Ricci soliton $(M, g, f)$ is asymptotic to a rotationally symmetric cone
$((0, \infty)\times S^{n-1}, dr^2 + \alpha r^2 g_{S^{n-1}})$ along some end $V\subset M$ if and only if $M\approx \RR^n$
and $g$ is flat.
\end{corollary}

Corollary \ref{cor:nonexistencerotsym} has some precedent in the category of steady and expanding gradient Ricci solitons.
Brendle \cite{BrendleBryant3D} has proven that 
any three-dimensional nonflat $\kappa$-noncollapsed steady gradient Ricci soliton
must be rotationally symmetric and hence, up to homothety, identical to Bryant's soliton \cite{BryantSoliton}. 
This was also asserted by Perelman \cite{Perelman1},
who further conjectured that Bryant's soliton is the unique complete, 
noncompact, three-dimensional $\kappa$-noncollapsed \emph{ancient} solution to the Ricci flow of bounded positive sectional curvature.  Brendle's approach in \cite{BrendleBryant3D} combines the construction of 
``approximate Killing vector fields'' with a careful blow-down analysis and a Liouville-type theorem for solutions to the 
Lichnerowicz PDE. 
The essential dimension-specific aspects of his argument are, first, that the sectional curvature
of a complete steady three-dimensional soliton is 
(by the Hamilton-Ivey estimate in its local \cite{ChenStrongUniqueness} and global \cite{HamiltonSingularities} forms) necessarily nonnegative and, 
second, that the asymptotic shrinking soliton obtained by parabolic blow-down from  a positively curved $\kappa$-noncollapsed steady soliton 
is known to be a cylinder.  
In a later paper, following the same general outline,  Brendle \cite{BrendleBryantND} extended his theorem
to higher-dimensional steady solitons of positive curvature operator which blow-down similarly to a cylinder.  

Using a modification of this ``approximate Killing vector'' technique, Chodosh \cite{ChodoshExpanders} has proven 
that if a complete expanding gradient Ricci soliton with 
nonnegative sectional curvature is asymptotic  
to a rotationally symmetric cone $((0,\infty)\times S^{n-1}, dr^2 + \alpha r^2 S^{n-1})$ for $\alpha \in (0, 1]$, 
then the soliton must itself be rotationally symmetric.  Where the parabolic blow-down procedure in \cite{BrendleBryant3D},
\cite{BrendleBryantND} is inapplicable in the expanding setting, Chodosh substitutes an argument
based on the elliptic maximum principle and a judicious choice of barrier functions constructed from the potential $f$. 
Arguing along these lines, in their very recent paper \cite{ChodoshFong}, Chodosh-Fong further prove that any K\"ahler-Ricci 
expanding soliton of 
positive holomorphic bisectional curvature asymptotic
to a $\operatorname{U}(n)$-invariant cone must be itself $\operatorname{U}(n)$-invariant and so identical to one of the family
of expanding solitons constructed by Cao \cite{CaoExpander}.

\subsection{Overview of the proof of Theorem \ref{thm:unique}}
Brendle's technique, however, \emph{does not} seem to extend in the same straightforward way to the case of shrinking Ricci solitons.
According to \cite{CarrilloNi}, a complete Ricci shrinker with nonnegative Ricci curvature must have vanishing asymptotic volume ratio
and so cannot be asymptotically conical.  An 
assumption of positive curvature of any kind is therefore undesirable for our purposes, yet, in its absence, it is unclear how  to develop the Liouville-type theorem
needed to pass from approximate to exact Killing vector fields (cf. the concluding comment in \cite{ChodoshFong}).
The positive coefficient of the metric in \eqref{eq:shrinker} also generates a zeroth-order term of uncooperative sign
in the associated Lichnerowicz PDE.
 
We pursue instead a completely different strategy and convert Theorem \ref{thm:unique} -- on its face,
an assertion of unique continuation at infinity for the weakly elliptic system \eqref{eq:shrinker}
-- into an assertion of backwards uniqueness for the weakly parabolic system \eqref{eq:rf}. 
By the same general strategy, the second
author in \cite{Wang} recently obtained an analogous uniqueness result for asymptotically conical self-shrinking solutions to the mean curvature flow.
The key idea can be summarized very succinctly: after appropriate normalizations on the ends $\bar{V}$ and $\hat{V}$, 
\emph{the self-similar solutions to the Ricci flow
associated to the solitons in Theorem \ref{thm:unique} can be made to coincide in finite time with the conical metric $g_c$}.
Thus the problem in Theorem \ref{thm:unique} becomes a clean (if analytically somewhat subtle) problem of backwards uniqueness. We describe this conversion
in greater detail below.

\subsubsection{Self-similar solutions to the Ricci flow.}
Recall that a family $g(t)$, $t\in I$, of metrics on $M$
is said to be a \emph{shrinking self-similar solution} to \eqref{eq:rf} if there is a smooth family of diffeomorphisms $\Psi_t:M\to M$
and a positive decreasing function $c(t)$ defined for $t\in I$ such that 
\begin{equation}\label{eq:rfselfsim}
  g(t) = c(t)\Psi^*_t(g(t_0))
\end{equation} 
for some $t_0\in I$. As is well-known (see, e.g., Lemma 2.4 in \cite{ChowKnopf}),
one can construct a local shrinking self-similar solution from a shrinking gradient Ricci soliton structure $(M, g, f)$ 
in an essentially canonical fashion.  Moreover, when $\nabla f$ is complete as a vector field,
(e.g., as happens when $g$ is complete, according to \cite{ZhangCompleteness}), this construction
produces a globally defined ancient solution to the Ricci flow.  

In the setting of Theorem \ref{thm:unique}, on our (typically incomplete) ends 
$\bar{V}$ and $\hat{V}$, we will obtain solutions $g(t) = -t\Psi_t^*\bar{g}$ and  
$\tilde{g}(t) = -t\tilde{\Psi}_t^*\hat{g}$ defined for $t\in [-1, 0)$  
that satisfy $g(-1) = \bar{g}$ and $\tilde{g}(-1) = \hat{g}$, have uniform quadratic curvature decay, and (as can be seen) converge smoothly
as $t\nearrow 0$ to limit metrics $g(0)$ and $\tilde{g}(0)$.  On one hand, the self-similarity of the solutions for $t\in [-1, 0)$
forces these limit metrics to be conical, on the other (as we will verify, but is at least intuitively plausible), they must also be asymptotic 
to $g_c$ in the sense of Definition 1.1. It follows, then,
that $g(0)$ and $\tilde{g}(0)$ must actually be isometric to the cone $g_c$ on some sufficiently restricted end.
Adjusting $g(t)$ and $\tilde{g}(t)$ by appropriate diffeomorphisms, we can thus arrange that they inhabit the same
end $W$ and agree identically at $t=0$. To conclude that $\bar{g}$ and $\hat{g}$ are isometric on some end, it is then
enough to show that $g(t) = -t\Psi_t^*\bar{g}$ and $\tilde{g}(t) = -t\widetilde{\Psi}_t^*\hat{g}$ agree identically
on $W^{\prime}\subset W$ for $t\in (-\epsilon, 0]$, and this is the backwards uniqueness problem we seek to solve.

\subsubsection{The model Euclidean problem}  A distinctive feature of Theorem \ref{thm:unique} 
(and of the corresponding result, Theorem 1.1, in \cite{Wang})  is that its conclusion is valid
without any restrictions on the soliton structures off of the particular ends $\bar{V}$ and $\hat{V}$. 
The analytic artifact of this flexibility is that we have no control on $g(t)$ and $\tilde{g}(t)$ at the spatial boundary of the end,
and the backwards uniqueness problem described above is considerably more delicate than, e.g., the
global problem considered in \cite{KotschwarBackwardsUniqueness} for complete solutions to \eqref{eq:rf}.

For a model of an attack on this problem, as in \cite{Wang}, 
we can look to the paper of Escauriaza-Seregin-\v{S}ver\'ak \cite{EscauriazaSereginSverak}.
There it is proven that any smooth
function $u$ on $(\RR^n\setminus B_R(0))\times [0, T]$ which satisfies
\begin{equation*}
\left|\partial_t u + \Delta u\right| \leq N\left(|u|+ |\nabla u|\right), \quad
u(x, 0) = 0, \quad\mbox{and}\quad |u(x, t)|\leq Ne^{N|x|^2}, 
\end{equation*}
must vanish identically.  The significance of their result is that it makes no
restriction on the behavior of $u$ on the parabolic boundary of 
$(\RR^n\setminus B_R(0))\times[0, T]$; it was previously known that this particular formulation would settle
a longstanding open question
in the regularity of solutions to the Navier-Stokes equations.  

Since \eqref{eq:rf} is only weakly-parabolic,
there is no direct generalization
of this result which we may apply to our backwards uniqueness problem, nor is there, as there is for the mean curvature flow, 
a convenient means of breaking the gauge-invariance of the equation to reduce the problem
to one for a corresponding strictly parabolic equation. (See, e.g., the first section of \cite{KotschwarBackwardsUniqueness}
for an explanation of the inapplicability of DeTurck's method to backwards-time uniqueness problems.) 
Nevertheless, as in \cite{KotschwarBackwardsUniqueness}, we can embed the problem 
into one for a prolonged ``PDE-ODE'' system of mixed
differential inequalities for which an analog of the above theorem can be shown to hold.
It is worth remarking that the elliptic unique continuation problem implied by Theorem \ref{thm:unique} is itself somewhat nonstandard,
 even neglecting the complications arising from the gauge-degeneracy which the system \eqref{eq:shrinker} shares with \eqref{eq:rf} -- see 
Section 3 of \cite{Wang} for some discussion of the features of the corresponding equation
in the related case of self-shrinking solutions of the mean curvature flow.

\subsubsection{Structure of the paper}

In Section \ref{sec:reduction}, we construct from $\bar{g}$ and $\hat{g}$
the self-similar solutions to the Ricci flow described above and carry out the reduction of
 Theorem \ref{thm:unique} to a specific problem of backwards uniqueness (Theorem \ref{thm:bu}). 
In Section \ref{sec:pdeode}, we convert this backwards uniqueness problem into one for a larger coupled system of mixed differential inequalities
(a ``PDE-ODE'' system). The technical heart of the paper is contained in Sections \ref{sec:carleman1} and \ref{sec:carleman2} where
 we develop two pairs of Carleman inequalities for time-dependent sections of vector bundles on a self-similar Ricci flow background.
We then combine these estimates in Section 
\ref{sec:backwardsuniqueness} to prove Theorem \ref{thm:bu}. We conclude the paper with two technical appendices. 
In Appendix \ref{app:asymcone}, we record some
some elementary consequences of Definition 1.1 and give a proof of a normalization lemma for shrinking solitons
with quadratic curvature decay.  In Appendix \ref{app:rotsymend},  we construct a rotationally symmetric gradient 
shrinking soliton asymptotic to each rotationally symmetric cone $((0, \infty)\times S^{n-1}, dr^2 + \alpha r^2 g_{S^{n-1}})$.
These examples furnish the rotationally symmetric ``competitor'' solitons we need to deduce Corollary \ref{cor:nonexistencerotsym}
from Theorem \ref{thm:unique}.


\section{Reduction to a problem of backwards uniqueness}\label{sec:reduction}

Going forward, as in the statement of Theorem \ref{thm:unique}, $(\Sigma, g_{\Sigma})$ will denote a closed Riemannian $(n-1)$-manifold 
and $g_c = dr^2+ r^2g_{\Sigma}$ a regular conical metric on $E_0 = (0, \infty) \times \Sigma$. 
We will use $r_c: E_0 \to \RR$ to denote the radial distance from the vertex relative to the conical metric $g_c$
(so in coordinates $(r, \sigma)$ on $E_0$, we have $r_c(r, \sigma) = r$) and will use the shorthand
\[ 
E_{\rad} = \{\,x\in E_0\,|\,r_c(x) > \rad\,\}, \quad\mbox{and}\quad \mathcal{E}_{\rad}^T \dfn E_{\rad}\times[0, T].
\]
Our aim in this section is to take the soliton structures $(\bar{M}, \bar{g}, \bar{f})$ and $(\hat{M}, \hat{g}, \hat{f})$
from Theorem \ref{thm:unique} and construct from them self-similar solutions to the backwards Ricci flow
 on $E_{\rad}\times (0, 1]$,
for some sufficiently large $\rad$, which flow smoothly from the cone $g_c$ at the singular time $\tau=0$ to
 isometric copies of (restrictions of) $\bar{g}$ and $\hat{g}$.  This construction converts Theorem \ref{thm:unique} 
into the assertion of parabolic backwards uniqueness
stated in Theorem \ref{thm:bu} below.  
  
\subsection{An asymptotically conical self-similar solution to the Ricci flow}

\begin{proposition}
\label{prop:reduction}
Suppose $(\bar{M}, \bar{g}, \bar{f})$ is a shrinking Ricci 
soliton asymptotic to the regular cone $(E_0, g_c)$ along the end $\bar{V}\subset \bar{M}$. 
Then there exist $K_0$, $N_0$, and $\rad_0> 0$, and a smooth family of maps $\bar{\Psi}_\tau: E_{\rad_0}\to \bar{V}$
defined for $\tau\in (0,1]$ satisfying:
\begin{enumerate}
 \item[(1)] For each $\tau\in (0, 1]$, $\bar{\Psi}_{\tau}$ is a diffeomorphism onto its image and
$\bar{\Psi}_{\tau}(E_{\rad_0})$ is an end of $\bar{V}$.
\item[(2)] The family of metrics $g(x, \tau) \dfn\tau\bar{\Psi}_\tau^*\bar{g}(x)$ 
is a solution to the backwards Ricci flow
\begin{equation}\label{eq:brf}
 \pd{g}{\tau} = 2\Rc(g)
\end{equation}
for $\tau \in (0, 1]$, and extends smoothly as $\tau\searrow 0$ to $g(x, 0) \equiv g_c(x)$ on $E_{\rad_0}$.

\item[(3)] For all $m = 0, 1, 2, \ldots $,
\begin{align}
\label{eq:curvdecay}
\sup_{E_{\rad_0}\times [0,1]} \left(r_c^{m+2}+1\right)\abs{\nabla^{(m)}\Rm(g)} & \le K_0.
\end{align}
Here $|\cdot| = |\cdot|_{g(\tau)}$ and $\nabla = \nabla_{g(\tau)}$ denote the norm and the Levi-Civita connection associated to the metric $g = g(\tau)$.

\item[(4)] If $f$ is the function on $E_{\rad_0}\times (0,1]$ defined by 
$f(\tau)=\bar{\Psi}_\tau^\ast \bar{f}$, then $\tau f$ extends
to a smooth function on all of $\Ec_{\rad_0}^1$ and there $g$ and $\tau f$ together satisfy
\begin{align}
\label{eq:fid0}
&  \lim_{\tau\searrow 0} 4\tau f(x, \tau) =r_c^2(x),\quad r_c^2-\frac{N_0}{r_c^{2}} \le 4 \tau f \le r_c^2 + \frac{N_0}{r_c^{2}}, \quad\mbox{and}\\
\label{eq:fid1}
& \pdtau (\tau f) =\tau R, \quad \  \tau^2|\nabla f|^2 - \tau f = -\tau^2 R, \quad \tau \Rc(g) + \tau \nabla\nabla f= \frac{g}{2}.
\end{align}
\end{enumerate}
\end{proposition}

Therefore, Theorem \ref{thm:unique} reduces to the following assertion of backwards uniqueness.
\begin{theorem}\label{thm:bu} Suppose that $g$ and $\tilde{g}$ are self-similar solutions to \eqref{eq:brf}
 on $E_{\rad_0}\times (0, 1]$ for some $\rad_0\geq 1$ that extend smoothly to $g_c$ on $E_{\rad_0}\times \{0\}$
and with their potentials $f$ and $\hat{f}$ satisfy
\eqref{eq:curvdecay} -- \eqref{eq:fid1} for some constants $K_0$ and $N_0$.
 Then there exists $\rad \geq \rad_0$ and $\tau^{\prime}\in (0, 1)$ such that $g\equiv \tilde{g}$
 on $\Ec_{\rad}^{\tau^{\prime}}$.
\end{theorem}
For the application to Theorem \ref{thm:unique}, note that $g(1)$ and $\tilde{g}(1)$ are isometric to $\tau^{-1}g(\tau)$ and $\tau^{-1}\tilde{g}(\tau)$, respectively, for any $\tau \in (0, 1]$.
We will postpone the proof of this theorem to Section \ref{sec:backwardsuniqueness}, until after we have developed the necessary ingredients
in Sections \ref{sec:pdeode} - \ref{sec:carleman2}. 

\subsection{Proof of Proposition \ref{prop:reduction}}

There are three main steps.  First, we show that
if a shrinking soliton $(\bar{M}, \bar{g}, \bar{f})$
is asymptotically conical along an end $\bar{V}$,
it has quadratic curvature decay, and so, on some end $\bar{V}^{\prime} \subset \bar{V}$, admits a reparametrization
that is compatible in a certain sense with the level sets of $\bar{f}$. Second, we show that a shrinking soliton with quadratic curvature decay gives rise
to a self-similar solution to the backwards Ricci flow that extends smoothly to a conical metric
on sufficiently distant regions at the singular time. (In particular, this shows that a soliton on a cylinder of the form
 $(a, \infty) \times \Sigma$ for some compact $\Sigma$ with quadratic curvature decay must be asymptotically conical.)
Finally, we argue that this conical limit metric and
the original asymptotic cone $g_c$ are isometric.  With a further adjustment by a diffeomorphism, we can then arrange that our
self-similar solution interpolates between a soliton asymptotic to $g_c$ and the cone $g_c$ itself.  

\subsubsection{Initial technical simplifications} To eliminate some notational baggage that 
we do not wish to carry with us through the entire proof, we make a couple of up-front reductions.
First, if $\Phi: E_{\rad}\to \bar{V}$ is the map from Definition \ref{def:asymcone}, then, 
replacing $\bar{g}$ and $\bar{f}$ by $\Phi^*\bar{g}$ and $\Phi^*\bar{f}$, we may as well assume
that $\Phi = \operatorname{Id}$ and $\bar{V} = E_{\rad}$. Second, Lemma \ref{lem:asymcone1} (b)-(c)
and Lemma \ref{lem:fnormalization}, after pulling-back by an additional diffeomorphism (and relabeling $\rad$) we may as well also assume that $\bar{f}$ and 
$\bar{g}$
are defined on $\bar{E}_{\rad/2} = (\rad/2, \infty) \times \bar{\Sigma}$ for some smooth closed $(n-1)$-manifold $\bar{\Sigma}$,
and that, writing $\bar{r}(x) \dfn d_{\bar{g}}(x, \partial \bar{E}_{\rad})$,
there are constants $K$ and $N$ such that the conditions
\begin{equation}\label{eq:techsimp1}
\bar{f}(r, \bar{\sigma}) = \frac{r^2}{4}, \quad(r^2 + 1)|\Rm(\bar{g})|_{\bar{g}}(r, \bar{\sigma}) \leq K, 
\quad
  N(r - 1) \leq \bar{r}(x) \leq N (r + 1),
\end{equation}
are satisfied for all $x = (r, \bar{\sigma}) \in \bar{E}_{\rad}$.  As we have only modified our soliton structure
by diffeomorphisms, our ``normalized'' $(\bar{E}_{\rad/2}, \bar{g}, \bar{f})$ will still be asymptotic to $(E_0, g_c)$
along an end of the closure of $\bar{E}_{\rad}$ 
in the sense of Definition \ref{def:asymcone}
(that we can adjust the domain of the diffeomorphism required by this definition to have the form $E_{\srad}$ for some $\srad$, follows from Lemma \ref{lem:asymcone1}(b)
and \eqref{eq:frbarcomp}).  We do not assume here that
$\Sigma$ and $\bar{\Sigma}$ are diffeomorphic.

\subsubsection{Distance estimates on the trajectories of $\bar{\nabla}\bar{f}$}
We now examine the relationship 
between the integral curves of the vector field $\bar{\nabla}\bar{f}$
and the radial trajectories.  In what follows, we will use $r$ to denote both the global coordinate on the factor $(0, \infty)$
and the function on $\bar{E}_0$ given by $r(r, \bar{\sigma}) = r$.
\begin{claim}
\label{clm:trajectory}
There exists $\rad^{\prime}>\rad$ depending only on $\rad$, $K$, and $N$, and
a one-parameter family of local diffeomorphisms $\Psi_s :\bar{E}_{\rad^{\prime}}\to \bar{E}_{\rad^{\prime}}$,
defined for $s\geq 0$, which satisfy
\begin{equation}
\label{eq:Psidef}
\pd{\Psi_s}{s} =\bar\nabla\bar{f}\circ \Psi_s
\quad\mbox{and}\quad
\Psi_{0}= \operatorname{Id}_{\bar{E}_{\rad}}.
\end{equation}
Moreover, for all $(r, \bar\sigma)\in \hat{E}_{\rad^{\prime}}$, $r_s\dfn r \circ\Psi_s$ satisfies 
\begin{equation}
\label{eq:ssigma}
(r-1)e^{s/2}+1\le r_s(r, \bar\sigma) \le (r+1)e^{s/2}-1.
\end{equation}
\end{claim}

\begin{proof}
First, by the local existence and uniqueness theory for ODE, for each initial point $x\in \bar{E}_{\rad/2}$, 
the trajectory $\Psi_s(x)$ of $\bar{\nabla}\bar{f}$ with $\Psi_0(x) = x$ exists for small $s$. Moreover, since $\bar{f} = r^2/4$, 
and
\begin{equation}
\label{eq:basicode}
\pd{}{s}(\bar{f}\circ\Psi_s) =|\overline{\nabla}\bar{f}|_{\bar{g}}^2\circ {\Psi_{s}} >0,
\end{equation}
it follows that $r_s(x)> r(x)$ for all $x$ and all $s\geq 0$ for which the trajectory is defined. In particular,
$\Psi_s(\bar{E}_{\rad}) \subset \bar{E}_{\rad}$, i.e., trajectories which begin in $\bar{E}_{\rad}$ stay in 
$\bar{E}_{\rad}$. 

Using the second equation in \eqref{eq:shrinker} and the boundedness of the scalar curvature $\bar{R} = \operatorname{scal}(\bar{g})$,  
we can obtain even better control on the distance, namely, 
\[
		    \pd{r_s}{s} = (\bar{f}^{-\frac12}|\overline{\nabla}\bar{f}|_{\bar{g}}^2)\circ\Psi_s 
= \frac{r_s}{2}\left(1- \frac{4\bar{R}\circ\Psi_s}{r_s^2}\right).
\]
So, if $\rad^{\prime}\geq \rad$ is sufficiently
large (depending only on $n$, $K$, and $\rad$), then 
\[
	   \frac{1}{2}\left(r_s - 1\right) \leq \pd{r_s}{s} \leq 
	\frac{1}{2}\left(r_s + 1\right)
\]
on $\bar{E}_{\rad^{\prime}}$.
Integrating this last equation with respect to $s$ yields \eqref{eq:ssigma}
and also proves the existence of the local diffeomorphisms $\Psi_s:\bar{E}_{\rad^{\prime}}\to \bar{E}_{\rad^{\prime}}$ for all $s\geq 0$. 
\end{proof}

\subsubsection{Derivative estimates}

Now, continuing from the statement of Claim \ref{clm:trajectory}, 
we set $s(t)\dfn-\log(-t)$ for $t < 0$ and define the family of metrics
\begin{equation}
\label{eq:gdef}
g(t)=-t\Psi_{s(t)}^*\bar{g},
\end{equation}
on $\bar{E}_{\rad^{\prime}}\times [-1,0)$. Then, as in Section 2.1 of \cite{ChowKnopf}, $g(t)$ solves \eqref{eq:rf}
with initial condition $g(-1) = \bar{g}$.  

Using the self-similarity
of $g(t)$, we can parlay the quadratic decay of $\Rm(\bar{g})$ into decay estimates
for the higher derivatives of $\Rm(g(t))$.
First, by the quadratic curvature decay and (\ref{eq:ssigma}), it follows that 
\begin{align*}
\begin{split}
&\sup_{\bar{E}_{\rad^{\prime}}\times [-1,0)}(r^2(x)+1)|\Rm(g)|_g(x, t) = \sup_{\bar{E}_{\rad^{\prime}}\times[-1, 0)}\frac{(r^2(x)+1)}{-t}\left|\Rm(\bar{g})\right|_{\bar{g}}\circ\Psi_{s(t)}(x)\\
&\quad\le \sup_{\bar{E}_{\rad^{\prime}}\times [0, \infty)} 8(r_s^2(x)+1)\left|\Rm(\bar{g})\right|_{\bar{g}}\circ\Psi_{s}(x) 
\le 8K.\\
\end{split}
\end{align*}
Then, according to the estimates of Shi \cite{Shi}, we have
$|\nabla^{(m)}\Rm(g)|_g(x, t) \leq K_m$ for all $m\geq 0$ and 
$(x, t) \in \{\,2\rad^{\prime}\leq r(x) \leq 4\rad^{\prime}\,\}\times [-1/2,0)$. (Here and below, $K_m$ denotes a constant that 
changes from inequality to inequality
but depends only on $m$, $n$, and $K$.) 
Thus, by the definition of the metrics $g(t)$, the distance estimate \eqref{eq:ssigma},
and the estimate on $|\Rm(g)|_g$ above, we have
\begin{equation}
\label{eq:solitoncurv}
\sup_{\bar{E}_{\rad^{\prime}}}\left(r^{m+2}(x)+1\right)\left|\overline{\nabla}^{(m)}\Rm(\bar{g})\right|_{\bar{g}}(x)\le K_m.
\end{equation}
Consequently, from a scaling argument akin to the one above, we have the following estimate on the higher derivatives of $\Rm(g(t))$.
\begin{claim}\label{clm:derivdecay}
For all $m\geq 0$, there exists a constant $K_m= K_m(n, K)$ such that
the curvature tensor of the solution $g(t) = -t\Psi_{s(t)}^*\bar{g}$ satisfies
\begin{equation}
\label{eq:derivdecay}
\sup_{(x, t) \in \bar{E}_{\rad^{\prime}}\times [-1,0)} \left(r^{m+2}(x)+1\right)\left|\nabla^{(m)}\Rm(g)\right|_g(x, t)\le K_m.
\end{equation}
\end{claim}

Since $\bar{\Sigma}$ is compact, we may find a finite atlas for $\bar{E}_{\rad^{\prime}}$ 
for which we have uniform estimates on the derivatives of the charts,
and argue as in the proofs of Theorem 6.45 and Proposition 6.48 of \cite{ChowKnopf} to see that $g(t)$ converges smoothly to a smooth metric 
$g_0 = g(0)$ on $\bar{E}_{\rad^{\prime}}\times \{0\}$.

\subsubsection{The potential function $f$ and limit metric $g(0)$}
Now define $f$ on $\bar{E}_{\rad^{\prime}}\times [-1,0)$ by $f=\Psi_{s(t)}^\ast\bar{f}$.  Then $f$ and $g$ 
together form a shrinking soliton structure
on $\bar{E}_{\rad^{\prime}}$ for each $t\in [-1, 0)$, albeit one with the constant $-1/(2t)$ in place of $1/2$
on the right side of \eqref{eq:shrinker}. The following identities are standard (see, e.g., Section 4.1 of \cite{ChowLuNi})
and follow easily from the definition of $f$, $g$, and equation \eqref{eq:shrinker}.
\begin{claim}
\label{lem:fidentity}
On $\bar{E}_{\rad^{\prime}}\times [-1, 0)$, $f$ satisfies
\begin{align*}
\frac{\partial f}{\partial t} =|\nabla f|^2, \quad |\nabla f|^2 + \frac{f}{t} = -R, \quad\mbox{and}\quad
\nabla\nabla f = - \Rc(g) - \frac{g}{2t}.
\end{align*}
\end{claim}
Given the estimates \eqref{eq:derivdecay} on the derivatives of curvature, it follows from these identities
that $-tf$ converges locally smoothly as $t\nearrow 0$ to a smooth limit function $q$ on $\bar{E}_{\rad^{\prime}}$. 
Moreover, there exists $N>0$, depending only on $K$, such that
\begin{align}\label{eq:qprop}
-\frac{N}{r^{2}}\leq q - \frac{r^2}{4}\leq \frac{N}{r^{2}},\quad
|\nabla q|_{g_0}^2=q,\quad\mbox{and}\quad  \nabla\nabla_{g_0} q=\frac{1}{2}g_0.
\end{align}
The first inequality implies that $q$ is proper (on the closure of $\bar{E}_{\rad^{\prime}}$) and positive
on sufficiently distant regions, which, with the second identity, implies that the level sets of $q$ corresponding to sufficiently
large values are smooth and diffeomorphic to a common closed $(n-1)$-manifold $\hat\Sigma$. Moreover, the second inequality implies
the integral curves of $2\sqrt{q}$ are geodesic.
As in Section 1 of \cite{CheegerColding}, this and the third identity in \eqref{eq:qprop} implies that $g_0$ is conical, i.e., 
there exists $\hat{\rad}> 0$, and a diffeomorphism $\hat{\Phi}$ from
$\hat{E}_{\hat\rad}\dfn (\hat{\rad}, \infty) \times \hat{\Sigma}$ to an end of the closure of $\bar{E}_{\rad^{\prime}}$
satisfying
\[
(q \circ \hat\Phi)(\hat{r}, \hat\sigma) = \frac{\hat{r}^2}{4}, \quad\mbox{and}\quad
\hat{\Phi}^*(g_0) = \hat{g}_c\dfn d\hat{r}^2 + \hat{r}^2g_{\hat{\Sigma}}. 
\]
\subsection{A final reparametrization} \label{ssec:proofofredprop}
Now consider the family of metrics $\hat{\Phi}^*(g(t))$ on $\hat{E}_{\hat{\rad}}$ for $t\in [-1, 0]$. Each member of
this family is uniformly equivalent to $\hat{g}_c = \hat\Phi^*(g_0)$ in view of the boundedness of $\Rc(g(t))$, and
from this equivalence, the identity
\[
    \hat{\Phi}^*(\bar{g}) = \hat{g}_c + \int_{-1}^0 2\Rc(\hat{\Phi}^*(g(s)))\,ds,
\]
the second and third inequalities in \eqref{eq:techsimp1}, and equation \eqref{eq:qprop},
it follows that there is a constant $N$ such that
\[
    |\hat{\Phi}^*(\bar{g}) - \hat{g}_c|_{\hat{g}_c}(\hat{r}, \hat{\sigma}) \leq \frac{N}{\hat{r}^2 +1}
\]
on $\hat{E}_{\hat\rad}$. Writing $\hat{\rho}_{\lambda}$ for the dilation map on $\hat{E}_0$, it follows immediately that the family of 
metrics
$\lambda^{-2}\hat\rho_{\lambda}^*\hat\Phi^*(\bar{g})$ converges to $\hat{g}_c$ in $C^0_{\mathrm{loc}}(\hat{E}_{\hat{\rad}}, \hat{g}_c)$
as $\lambda\to\infty$.

On the other hand, we assume that $\bar{g}$ is asymptotic to $(E_0, g_c)$ along an end of $\bar{E}_{\rad}$. Since $\bar{E}_{\rad} \setminus \bar{E}_{\rad+1}$
is bounded relative to $\bar{g}$ (implying in particular, that $(\bar{E}_{\rad}, \bar{g})$ has at most one end relative to any compact set),
  Lemma \ref{lem:uniquenessofasymcones}
implies that $(E_0, g_c)$ and $(\hat{E}_0, \hat{g}_c)$ are isometric.  Call the isometry between them $F$.  Replacing $g(t)$ and $f(t)$ with
 their pull-backs by  $\hat{\Phi}\circ F$, on a sufficiently distant end we then achieve $g(0) = g_c$ exactly and that $-tf$ converges smoothly to $r_c^2/4$ 
as $t\nearrow 0$. The estimates \eqref{eq:curvdecay}
and \eqref{eq:fid0}, which hold in terms of the radial parameter $r$ for $f$ and $g$ prior to their replacement by their pull-backs
will then also hold (for possibly larger constants) in terms of $r_c$.
  Setting $\tau = -t$ completes the proof.

\section{A PDE-ODE System}
\label{sec:pdeode}
Next, as in Section 2 of \cite{KotschwarBackwardsUniqueness}, we convert the backwards uniqueness problem
in Theorem \ref{thm:bu}
into one for solutions to a PDE-ODE system of inequalities that are amenable to the 
development of parabolic-type Carleman inequalities.
The idea is to try to build a closed system out of sufficiently many components of the form
 $\nabla^{(k)}(\Rm - \Rmt)$ (the ``PDE part'') and $\nabla^{(k)}(g- \gt)$ (the ``ODE part'').  (Here and in what follows
we will simply write $\Rm$ and $\Rmt$ for the $(3, 1)$ curvature tensors of the solutions $g$ and $\gt$.)
The curvature tensors $\Rm$ and $\Rmt$ will
independently satisfy strictly parabolic equations, and their differences will satisfy parabolic
equations up to lower order differences of these derivatives and error terms involving the tensors
$\nabla^{(k)}(g- \gt)$.  In turn, the norms of these latter tensors can be controlled
by the norms of the tensors $\nabla^{(k)}\Rm - \widetilde{\nabla}^{(k)}\Rmt$ via their evolution equations by a simple ODE comparison.

We will only summarize the construction of this system below and refer the reader to \cite{KotschwarBackwardsUniqueness} for more details.
The use of such PDE-ODE systems  originated in the work of Alexakis \cite{Alexakis} on 
the problem of unique continuation for the vacuum Einstein equations. (See also \cite{WongYu}.)

\subsection{Elements of the prolonged system} \label{ssec:xydef}

The PDE part of our system will be composed of the tensors
\begin{equation}\label{eq:pdedef}
  S\dfn\Rm-\Rmt \quad \mbox{and} \quad T \dfn \del\Rm -\delt\Rmt,
\end{equation}
and the ODE part of the tensors 
\begin{equation}\label{eq:odedef}
U \dfn g - \gt,\quad V \dfn \del - \delt,\quad\mbox{and}\quad W \dfn \nabla V.
\end{equation}
Here $V$ is a $(2, 1)$-tensor, given in local coordinates by $V_{ij}^k = \Gamma_{ij}^k - \Gamt_{ij}^k$.
Using $T_k^l(E_{\rad_0})$ to denote the bundle of $(k,l)$-tensors  over $E_{\rad_0}$, we define
\[
  \Xc \dfn T_{3}^1(E_{\rad_0})\oplus T_{4}^1(E_{\rad_0}), \quad \Yc \dfn T_2(E_{\rad_0})\oplus T_{2}^1(E_{\rad_0})\oplus T_{3}^1(E_{\rad_0}), 
\]
and smooth families of sections $\Xb(\tau) \in C^{\infty}(\Xc)$, $\Yb(\tau) \in C^{\infty}(\Yc)$ for $\tau \in [0, 1]$ by
\[
  \Xb\dfn S\oplus T,\quad \mbox{and}\quad \Yb\dfn U\oplus V\oplus W.
\]

We will use $g = g(\tau)$ and its Levi-Civita connection, $\nabla = \nabla_{g(\tau)}$,
as a reference metric and connection in our calculations
(and will use the same symbols to denote the metrics and connections they induce on $\Xc$ and $\Yc$ and the other tensor bundles we consider).
 We will also write  $\Delta = g^{ab}\nabla_a\nabla_b$ for the induced Laplacian on $\Xc$,
and use $|\cdot| \dfn |\cdot|_{g(\tau)}$ for induced family of norms on each fiber of $\Xc$ and $\Yc$. 

\subsection{Evolution equations}

We now import from \cite{KotschwarBackwardsUniqueness} the following evolution equations
for the components of $\Xb$ and $\Yb$, correcting some typographical errors in that reference.  Here the notation
$A\ast B$ represents a linear combination of contractions of tensors $A$ and $B$ with the metric $g$. 

\begin{lemma}[Lemma 2.4, \cite{KotschwarBackwardsUniqueness}] On $\Ec_{\rad_0}^1$, we have the evolution equations
\label{lem:evolution}
\begin{align}
\begin{split}
 \label{eq:sev}
 	\left(\pdtau + \Delta\right) S &=   \gt^{-1}\ast\delt\delt\Rmt\ast U + \delt\Rmt\ast V + \Rmt\ast W + \Rmt\ast V\ast V\\  
&\phantom{=} 					+  \gt^{-1}\ast\Rmt\ast\Rmt\ast U + \Rmt\ast S + S\ast S,
 \end{split}\\
 \begin{split}
 \label{eq:tev}
 	\left(\pdtau + \Delta\right) T &= \gt^{-1}\ast\delt^{(3)}\Rmt\ast U + \delt\delt\Rmt\ast V +\delt\Rmt\ast V\ast V\\
    &\phantom{=}+\delt\Rmt\ast W+ \gt^{-1}\ast\Rmt\ast\delt\Rmt\ast U + \Rmt\ast T\\
    &\phantom{=}+ \delt\Rmt\ast S + S\ast T,
\end{split}
 \end{align}
and
\begin{align}
 \label{eq:uev} \pdtau U_{ij} &= 2S_{lij}^l, \\
\begin{split}
\label{eq:vev} 
  \pdtau V_{ij}^k &= g^{mk}\left(T_{ipjm}^p + T_{jpim}^p 
	 - T_{mpij}^p\right)\\
  &\phantom{=}- \gt^{am}g^{bk}\left(\delt_{i}\tilde{R}_{jm} 
 	 + \delt_{j}\tilde{R}_{im} - \delt_m\tilde{R}_{ij}\right)U_{ab},
\end{split}\\ 
\begin{split}
\label{eq:wev}
  	\pdtau W
  &= \nabla T  + \delt\Rmt \ast V + \gt^{-1}\ast\delt\delt\Rmt\ast U\\
  &\phantom{=}
      +T\ast V + \gt^{-1}\ast\delt\Rmt\ast U\ast V.
\end{split}
\end{align}
\end{lemma}

\subsection{A coupled system of inequalities}

The key feature of equations \eqref{eq:sev} - \eqref{eq:wev} is that each term on the right-hand side contains
at least one factor of a (possibly contracted) component of either $\Xb$, $\nabla\Xb$, or $\Yb$. Our
assumptions guarantee that the other factors in each term will be at least be uniformly bounded on $\Ec_{\rad_0}^{1}$.
Thus using the Cauchy-Schwarz inequality, we can organize the evolution equations for $\Xb$ and $\Yb$
into the following closed system of inequalities.
  
\begin{proposition}
\label{prop:pdeode}
There exists $N>0$, depending only on $n$, $K$ and $\widetilde{K}$, such that $S$, $T$, $U$, $V$, and $W$
satisfy
\begin{equation}\label{eq:stuvwbounds}
  \sup_{\Ec_{\rad_0}^1}\left\{|S| + |T| + |\nabla S| + |\nabla T| + |U| + |V| + |W|\right\} \leq Nr_c^{-2}
\end{equation}
and the coupled system of inequalities
\begin{align}
\begin{split}\label{eq:pdedetail}
  \left|\pd{S}{\tau} + \Delta S\right| &\leq Nr_c^{-2}\left(|S| + |U| +|V| + |W|\right),\\
  \left|\pd{T}{\tau} + \Delta T\right| &\leq Nr_c^{-2}\left(|S| + |T| + |U| + |V| + |W|\right),
\end{split}
\end{align}
and
\begin{align}
\begin{split}\label{eq:odedetail}
  \left|\pd{U}{\tau}\right| &\leq N|S|, \quad
  \left|\pd{V}{\tau}\right| \leq N|T| + Nr_c^{-2}|U|,\\ 
\left|\pd{W}{\tau}\right| &\leq N|\nabla T| + Nr_c^{-2}\left(|U| + |V|\right).
\end{split}
\end{align}
In particular $\Xb = S\oplus T$ and $\Yb = U\oplus V\oplus W$ satisfy
\begin{align}
\begin{split}
\label{eq:pdeode}
\left|\pd{\Xb}{\tau} + \Delta \Xb\right| &\leq Nr_c^{-2}\left(|\Xb|  + |\Yb|\right)\\
\left|\pd{\Yb}{\tau}\right| &\leq N\left(|\Xb| + |\nabla\Xb|\right) + Nr_c^{-2}|\Yb|
\end{split}
\end{align}
for $N$ sufficiently large on $\Ec_{\rad_0}^1$.
\end{proposition}
\begin{proof}
The argument goes essentially as in Proposition 2.1 of \cite{KotschwarBackwardsUniqueness}.
By \eqref{eq:curvdecay},
we know the derivatives of the curvature tensors of both solutions have at least quadratic
decay relative to $r_c$, so the coefficients of the form $\delt^{(m)}\Rmt$ and the extra-linear
factors of $S$ and $T$ in equations
\eqref{eq:sev} -- \eqref{eq:wev} have at least quadratic decay. (Note that the curvature bounds, together with the
fact that $g$ and $\gt$ agree identically on $E_{\rad_0}\times\{0\}$ in particular imply that the metrics
are uniformly equivalent, so that the bounds on $\widetilde{\nabla}^{(k)}\Rmt$ are also valid
in the $g(\tau)$-norm.) The same goes
for the extra-linear factors of  $U$, $V$, $W$, since, just as in \cite{KotschwarBackwardsUniqueness},
they can be estimated at each fixed $x$ by simply integrating their $\tau$-derivatives, and these
are controlled in turn by the pointwise values of $|\nabla^{(m)}\Rm|$ and $|\widetilde{\nabla}^{(m)}\Rmt|$. On the other hand,
not every term in the evolution equations for $U$, $V$, and $W$ has a coefficient with quadratic decay,
and so in \eqref{eq:pdeode} the coefficient of $|\Xb|$ and $|\nabla\Xb|$ in the second equation
is merely constant.
\end{proof}
 

\section{Carleman estimates to imply backwards uniqueness}
\label{sec:carleman1}
The key technical components which we will need to prove Theorem \ref{thm:bu} are two pairs of Carleman estimates.
In this section, we establish the first of these, the pair
which ultimately will imply the vanishing of $\Xb$ and $\Yb$. 
A model for the sort of thing we are after is estimate (1.4) of \cite{EscauriazaSereginSverak} (cf. Proposition 3.5 in \cite{Wang}),
which states that, for all $R > 0$, there is a constant $\alpha^* = \alpha^*(R, n)$ such
that 
\begin{align*}
  &\|e^{\alpha(T-\tau)(|x|-R)+|x|^2}u\|_{L^2(Q_{R, T})} +  \|e^{\alpha(T-\tau)(|x|-R)+|x|^2}\nabla u\|_{L^2(Q_{R, T})}\\
  &\qquad\leq \|e^{\alpha(T-\tau)(|x|-R)+|x|^2}(\partial_{\tau} + \Delta) u\|_{L^2(Q_{R, T})}
   +\|e^{|x|^2}\nabla u(\cdot, T)\|_{L^2(\RR^n\setminus B_R(0))}
\end{align*}
for all $\alpha \geq \alpha^*$ and $u \in C_c^{\infty}(Q_{\rad, T})$ satisfying $u(\cdot, 0) \equiv 0$. Here $Q_{R, T} \dfn
(\RR^n\setminus B_R(0))\times [0, T]$.
We devote most of this section to proving a generalization of this result
applicable to the components of the PDE portion of \eqref{eq:pdeode}, and then prove a compatible ``Carleman-type'' estimate 
for the ODE portion; these estimates are contained in Proposition
\ref{prop:systemcarlemanineq} below.

\subsection{Notation and standing assumptions} \label{ssec:standingassumptions}
It will be convenient to perform our calculations relative to the metric $g = g(\tau)$ from Theorem \ref{thm:bu} and its Levi-Civita connection. Thus,
\emph{in this section and the next we will operate under the standing assumption
that $\rad_0\geq 1$ and $0< \tau_0 \leq 1$ are given, and $g$ and its potential $f = f(\tau)$
satisfy  \eqref{eq:brf} -- \eqref{eq:fid1} for some constant $K_0$, relative to the regular cone $(E_0, g_c)$.} 
In most places, we will suppress the dependency of the norms and connections on $g$, and simply
write $|\cdot| = |\cdot|_{g(\tau)}$, $\nabla = \nabla_{g(\tau)}$,
and $d\mu = d\mu_{g(\tau)}$. We will continue to use $r_c(x)$ to denote the radial distance in the conical metric $g_c$,
and use $A_0 \dfn \operatorname{vol}_{g_{\Sigma}}(\Sigma)$ for the area (relative to the conical metric $g_c$) of the
cross-section of $E_0$ at distance one from the vertex.  

Also, for the next two sections, $\mathcal{Z} = T^{\kappa}_{\nu}(E_{\rad_0})$ will denote a generic tensor bundle over $E_{\rad_0}$.
Most of our constants will depend on some combination of the ``background'' parameters $n$, $A_0$, $K_0$, $\kappa$ and $\nu$;
for completeness, we add that we will say that a constant depends on $K_0$ only if it depends on $\max\{K_0, 1\}$ (and similarly for $A_0$, $\kappa$ and $\nu$).

\subsection{A divergence identity}
Both of the primary estimates \eqref{eq:pdecarlemanineq} and \eqref{eq:pdecarleman2}
arise from the following divergence identity, which generalizes Lemma 1 of \cite{EscauriazaSereginSverak} and 
Lemma 3.2 of \cite{Wang} to time-dependent backwards-heat operators acting on sections of tensor bundles.
Here the Laplacian on $\mathcal{Z}$ is defined by $\Delta Z \dfn g^{ij}\nabla_i\nabla_j Z$. We will use
  $F$ and $G$ to denote arbitrary smooth functions on $\Ec_{\rad_0}^{\tau_0}$ with $G > 0$,
and write $\phi \dfn \log G$.

By analogy with \cite{EscauriazaSereginSverak}, we then consider the operators
\begin{equation}\label{eq:asdef}
    \mcA \dfn \pdtau - \nabla_{\nabla\phi} + \frac{F}{2}\Id, 
  \quad\mbox{and}\quad \mcS\dfn \Delta + \nabla_{\nabla\phi} - \frac{F}{2}\Id,
\end{equation}
acting on $Z\in C^{\infty}(\mathcal{Z}\times[0, \tau_0])$.
Unlike their counterparts in \cite{EscauriazaSereginSverak}, $\mcA$ and $\mcS$ will not be quite 
antisymmetric and symmetric, respectively, in $L^2(Gd\mu\,d\tau)$, but will nevertheless be close enough to being so
that we may prove
a useful perturbation of the formula in that reference. The proof of the identity below is a straightforward
if somewhat lengthy verification.

\begin{lemma}\label{lem:dividentity1} The following identity holds on $\Ec_{\rad_0}^{\tau_0}$
for all $Z\in C^{\infty}(\mathcal{Z}\times[0, \tau_0])$ and all smooth $F$ and $G > 0$:
\begin{align}\label{eq:dividentity1}
\begin{split}
  &\nabla_i\bigg\{2\left\langle \pd{Z}{\tau}, \nabla_i Z\right\rangle G + |\nabla Z|^2\nabla_i G
  -2\left\langle\nabla_{\nabla G}Z, \nabla_i Z\right\rangle + \frac{FG}{2}\nabla_i|Z|^2\\
  &\phantom{=} \;+ \frac{1}{2}\left(F\nabla_i G - G\nabla_i F\right)|Z|^2\bigg\}\,d\mu
	 -\pdtau\bigg\{\left(|\nabla Z|^2 + \frac{F}{2}|Z|^2\right)G\,d\mu\bigg\}\\
   &= \bigg\{2\left\langle \mcA Z, \left(\pdtau + \Delta\right)Z\right\rangle G - 2\left|\mcA Z\right|^2G
    -\frac{1}{2}\left(\pd{F}{\tau} + \Delta F\right)|Z|^2G \\
  &\phantom{=} \;+\left(F - G^{-1}\left(\pd{G}{\tau} 
      - \Delta G + RG\right)\right)\left(|\nabla Z|^2 +\frac{F}{2}|Z|^2\right)G - \frac{FG}{2}\pd{g}{\tau}(Z, Z) \\
     &\phantom{=}\; 
    - 2\nabla_i\nabla_j\phi\langle \nabla_iZ, \nabla_jZ\rangle G -\pd{g}{\tau}(\nabla Z, \nabla Z)G 
+ 2E(Z, \nabla Z)G\bigg\}\,d\mu,
\end{split}
\end{align}
where $\pd{g}{\tau}$ represents the $\tau$-derivative of the metrics induced by $g$
on $\mathcal{Z}$ and $T^*(E_{\rad_0})\otimes \mathcal{Z}$, 
$E(Z, \nabla Z)$ denotes the sum of commutators 
\begin{align}
\begin{split}\label{eq:edef}
  E(Z, \nabla Z) &\dfn \left\langle \left[\nabla_i, \pdtau\right]Z, \nabla_i Z\right\rangle -
      \left\langle \left[\nabla_i, \nabla_j\right]Z, \nabla_i Z\right\rangle\nabla_j\phi\\
     &=g^{qm}
\left(\nabla_{i}R_{pm} + \nabla_pR_{im} - \nabla_mR_{ip}
      + R_{pmi}^j\nabla_j\phi\right)\left\langle \Theta^p_qZ, \nabla_iZ\right\rangle,
\end{split}
\end{align}
and $\Theta^p_q$ is the operator
\begin{align*}
    \Theta^p_q(Z_{\alpha}^{\beta}) &= \delta^p_{\alpha_1} Z_{q\alpha_2\cdots\alpha_\nu}^{\beta_1\beta_2\cdots\beta_\kappa}
	+ \delta^p_{\alpha_2}Z_{\alpha_1q\alpha_{3}\cdots\alpha_\nu}^{\beta_1\beta_2\cdots\beta_\kappa} + \cdots +
      \delta^p_{\alpha_\nu}Z_{\alpha_1\alpha_2\cdots q}^{\beta_1\beta_2\cdots\beta_\kappa}\\
      &\phantom{=}-\delta_q^{\beta_1} Z_{\alpha_1\alpha_2\cdots\alpha_\nu}^{p\beta_2\cdots\beta_\kappa}
	- \delta_q^{\beta_2} Z_{\alpha_1\alpha_2\cdots\alpha_\nu}^{\beta_1 p \cdots\beta_\kappa} - \cdots -
      \delta_q^{\beta_\kappa} Z_{\alpha_1\alpha_2\cdots\alpha_\nu}^{\beta_1\beta_2\cdots p},
\end{align*}
i.e., $\Theta^p_qZ_{i} = \delta^p_iZ_{q}$, $\Theta^p_qZ_{ij}^k = \delta^p_iZ_{qj}^k +\delta^p_jZ_{iq}^k 
      -\delta^k_qZ^{p}_{ij}$, etc.
\end{lemma}
\begin{remark}\label{rem:commutator}
We note for later an important observation regarding \eqref{eq:edef}.  In our applications below,
we will have $\nabla \phi = \Upsilon\nabla f$ for some function $\Upsilon$ and since
\begin{equation}\label{eq:nablarc}
    \nabla_{i}R_{jk} - \nabla_{j}R_{ik} = R_{ijk}^p\nabla_pf,
\end{equation}s
for $\tau\in (0, \tau_0]$,
owing to \eqref{eq:fid1}, we have
\begin{equation}\label{eq:eequation1}
 E(Z, \nabla Z) = \left(\nabla_{i}R_{pq} + (1+\Upsilon)(\nabla_pR_{iq} -\nabla_iR_{pq}) 
      \right)\langle \Theta^p_qZ, \nabla_iZ\rangle,
\end{equation}
so that, for some $C = C(n, \kappa, \nu)$, 
\begin{equation}\label{eq:eequation2}
  |E(Z, \nabla Z)| \leq C|\nabla \Rc|(|\nabla Z|^2 + (1+ \Upsilon^2)|Z|^2)
\end{equation}
on all of $\Ec_{\rad_0}^{\tau_0}$,
i.e., we may control $E(Z, \nabla Z)$ by $\Upsilon$, $\Rc$, $\nabla\Rc$, $Z$ and $\nabla Z$ alone,
and eliminate the dependency of the estimate on $\nabla \phi$.  
\end{remark}

\subsection{A weighted $L^2$-inequality for the operator $\partial_{\tau}+ \Delta$}

When $Z(\cdot, \tau)$ has compact support in $E_{\rad_0}$ for each $\tau$ and vanishes at $\tau = 0$, the above identity can be integrated
and used to control  $|Z|$ and $|\nabla Z|$  by $|(\partial_{\tau} + \Delta)Z|$ in a suitably weighted 
$L^2$-sense.  Choosing $F = G^{-1}(\partial_{\tau}G - \Delta G + RG)$ to obtain some cancellation
of terms on the right-hand side of \eqref{eq:dividentity1}, 
integrating over $\mc{E}_{\mc{R}_0}^{\tau_0}$ and using the Cauchy-Schwarz inequality, we obtain 
the following analog of Lemma 2 in \cite{EscauriazaSereginSverak}.
\begin{lemma}\label{lem:pdecarlemanineq}
  There exists a constant $N = N(n, \kappa, \nu, K_0)$ such that if 
$Z\in C^{\infty}(\mathcal{Z}\times[0, \tau_0])$ is compactly supported in $E_{\rad_0}$ for each $\tau$
and satisfies $Z(\cdot, 0)= 0$,
then, for any smooth $G > 0$, we have
\begin{align}\label{eq:pdecarlemanineq}
\begin{split} 
&\frac{1}{2}\iint_{\Ec_{\rad_0}^{\tau_0}} \left|\pd{Z}{\tau} + \Delta Z\right|^2 G\,d\mu\,d\tau
    +\int_{E_{\rad_0}\times\{\tau_0\}}\left(|\nabla Z|^2 + \frac{F}{2}|Z|^2\right)G\,d\mu\\
&\qquad\qquad \geq\iint_{\Ec_{\rad_0}^{\tau_0}} \big(Q_1(\nabla Z,\nabla Z) + Q_2(Z, Z) -2 E( Z, \nabla Z)\big) G\,d\mu\,d\tau
\end{split}
\end{align}
where  $E(Z, \nabla Z)$ is given by \eqref{eq:edef}, $F \dfn G^{-1}(\partial_{\tau}G- \Delta G) + R$, and
\begin{align}
\begin{split}\label{eq:q1q2def}
  Q_1(\nabla Z, \nabla Z) &= 2(\nabla_i\nabla_j\phi)\langle\nabla_i Z, \nabla_j Z\rangle -\frac{N}{r_c^2}|\nabla Z|^2\\
  Q_2(Z, Z) &= \frac{1}{2}\left(\pd{F}{\tau} + \Delta F\right)|Z|^2 - \frac{N|F|}{r_c^2}|Z|^2.
\end{split}
\end{align}
\end{lemma}

\subsection{A weighted $L^2$-inequality for the ODE component.}

Next, we establish a matching $L^2$-inequality for the ODE component of the system; its proof is essentially trivial.

\begin{lemma}
\label{lem:odecarlemanineq}
 There exists a constant $N = N(n, \kappa, \nu, K_0)$ such that if $Z\in C^{\infty}(\mathcal{Z}\times[0, \tau_0])$ is compactly supported in $E_{\rad_0}$ for each $\tau$
and satisfies $Z(\cdot, 0) = 0$, then, for all smooth $G >0$,
\begin{equation}
\label{eq:odecarlemanineq}
-\iint_{\Ec_{\rad_0}^{\tau_0}} \left(N+\pd{\phi}{\tau}\right)|Z|^2G\, d\mu\, d\tau\le
\iint_{\Ec_{\rad_0}^{\tau_0}} \abs{\frac{\partial Z}{\partial \tau}}^2 G\, d\mu\, d\tau.
\end{equation}
\end{lemma}

\begin{proof}
Note that
\begin{equation}
\label{eq:tweightnorm}
\frac{\partial}{\partial \tau}\left(\abs{Z}^2G\right)=2\left\la \frac{\partial Z}{\partial \tau}, Z\right\ra G+\abs{Z}^2\frac{\partial G}{\partial\tau}+\pd{g}{\tau}(Z, Z)G.
\end{equation}
The inequality (\ref{eq:odecarlemanineq}) then follows upon integrating \eqref{eq:tweightnorm} over $\Ec_{\rad_0}^{\tau_0}$ and applying the Cauchy-Schwarz inequality
together with \eqref{eq:curvdecay}.
\end{proof}

\subsection{An approximately radial function}
\label{ssec:hdef}
Our next task is to construct a suitable weight function $G_1$ to substitute for $G$
in inequalities \eqref{eq:pdecarlemanineq} and \eqref{eq:odecarlemanineq}.  
As a first step we introduce the function $h: \Ec_{\rad_0}^{\tau_0} \to \RR$ defined  by 
\begin{equation}\label{eq:hdef}
 h(x, \tau) \dfn 
\left\{\begin{array}{ll}
  2\sqrt{\tau f(x, \tau)} &\mbox{for}\quad \tau > 0,\\
  r_c(x) &\mbox{for}\quad \tau = 0.
\end{array}\right.
\end{equation}
which will prove to be a useful approximation of the (conical) radial distance on our evolving solution. 
Observe first that $h \in C^{\infty}(\Ec_{\rad_0}^{\tau_0})$; indeed, using the asymptotics we have established for $f$
in Proposition \ref{prop:reduction}, 
$\lim_{\tau\searrow 0} h(x, \tau) = r_c(x)$ in every $C^k$-norm and satisfies
\begin{equation}\label{eq:h0prop}
  |\nabla h|^2(x, 0) = 1, \quad\mbox{and}\quad \nabla\nabla (h^2)(x, 0) = 2g(x, 0) = 2g_c(x)
\end{equation}
on $E_{\rad_0}$. Also, from \eqref{eq:fid0}, we see that 
\begin{equation}\label{eq:hrbounds}
  \frac{1}{2}r_c(x)\leq h(x, \tau)\leq 2r_c(x),
\end{equation}
on $\Ec_{\rad_0}^{\tau_0}$; in view of  \eqref{eq:curvdecay}, we consequently have the inequalities
\begin{equation}\label{eq:hcurvbounds} 
h\geq \frac{1}{2} \quad \mbox{and}\quad  
h^2(|\Rm| + |\nabla \Rm|)\leq CK_0
\end{equation}
on $\Ec_{\rad_0}^{\tau_0}$ for some universal constant $C$.  The identities \eqref{eq:fid0} and \eqref{eq:fid1} also directly imply
the following expressions for the derivatives of $h$ for $\tau > 0$.
\begin{lemma}\label{lem:hder}
 On $E_{\rad_0}\times (0, \tau_0]$, the derivatives of $h$ satisfy
\begin{align}
\label{eq:hspatial}
\nabla h &= \frac{2\tau}{h}\nabla f,  \quad h\nabla\nabla h = g - 2\tau\Rc - \nabla h\otimes\nabla h
\end{align}
and 
\begin{align}\label{eq:hspatial2}
 |\nabla h|^2 = 1 - \frac{4\tau^2 R}{h^2},\quad h \Delta h = n - 2\tau R - |\nabla h|^2, 
\end{align}
and
\begin{align}
\label{eq:pdtauh}
  \pd{h}{\tau} &= \frac{h}{2\tau}\left(1-|\nabla h|^2\right) = \frac{2\tau R}{h}.
\end{align}
\end{lemma}

Equation \eqref{eq:pdtauh} can be used to obtain a useful refinement of \eqref{eq:hrbounds}.
\begin{lemma}\label{lem:rhcomparison}
There exists a universal constant $C$ such that
\begin{equation}\label{eq:rhcomparison}
    |h(x, \tau) -r_{c}(x)| \leq \frac{CK_0\tau^2}{r_c^{3}(x)}.
\end{equation}
for all  $(x, \tau) \in \Ec_{\rad_0}^{\tau_0}$.
\end{lemma}
\begin{proof}
  Fix an arbitrary $x\in E_{\rad_0}$ and integrate both sides of \eqref{eq:pdtauh} with respect to $\tau$. Using \eqref{eq:hcurvbounds} and
  that we have normalized to achieve $h(x, 0) = r_c(x)$, we obtain that
   $|h^2(x, \tau) - r^2_c(x)| \leq CK_0r_c^{-2}(x)\tau^2$, so 
\[
|h(x, \tau) - r_c(x)||h(x, \tau) + r_c(x)|\leq CK_0 r_c^{-2}(x)\tau^2.
\]
and the claim follows. 
\end{proof}

\subsection{A weight function of rapid growth}

With $h$ in hand, we now construct our weight function $G_1$.
We fix $\delta\in (0, 1)$ and define, for all $\alpha>0$, the function
\begin{equation}
\label{eq:G1def}
G_1 \dfn G_{1; \alpha, \tau_0}(x,\tau)=\exp\left[\alpha (\tau_0-\tau)h^{2-\delta}(x, \tau)+ h^2(x, \tau)\right]
\end{equation}
and, writing $\phi_1 \dfn \phi_{1, \alpha, \tau_0}\dfn\log G_1$, also define
\begin{equation}
\label{eq:Fdef}
F_1 \dfn F_{1; \alpha, \tau_0}\dfn \pd{\phi_1}{\tau} -\Delta \phi_1-|\nabla\phi_1|^2+ R.
\end{equation}
Using Lemma \ref{lem:hder} we may obtain expressions for the derivatives of $\phi_1$ up to second order. Eventually,
we will simply estimate away the terms involving curvature, but we must be reasonably precise about them
at this point, since we will later need to compute two additional derivatives of $\phi_1$ in order
to estimate the expression involving $F_1$ in \eqref{eq:pdecarlemanineq}.

\begin{lemma}\label{lem:phiest}
For any $\alpha$ and any $\tau_0$, $\delta \in (0, 1)$, $\phi_1 = \phi_{1; \alpha, \tau_0}$ satisfies 
\begin{align}
\begin{split}
\label{eq:phitau}
  \pd{\phi_1}{\tau} &= 4\tau R - \alpha h^{2-\delta}\left(1 - \frac{2(2-\delta)\tau(\tau_0 - \tau)R}{h^2}\right),
\end{split}\\
\begin{split}\label{eq:phigrad}
  \nabla \phi_1 &= \left(\alpha(2-\delta)(\tau_0 - \tau)h^{1-\delta} + 2 h\right)\nabla h, \quad\mbox{and}
\end{split}\\
\begin{split}\label{eq:phihess}
\nabla\nabla \phi_1 &= 2\left(g -2\tau\Rc\right)
    + \frac{\alpha(2-\delta)(\tau_0-\tau)}{h^{\delta}}\left(g-2\tau\Rc -\delta\nabla h\otimes \nabla h\right).
\end{split}
\end{align}
In particular, there exists  a constant $\rad_1 \geq \rad_0$ depending only on $n$, $\delta$, and $K_0$,
such that on $\Ec_{\rad_1}^{\tau_0}$
\begin{equation}\label{eq:phihessFest}
  \nabla\nabla \phi_{1} \geq g \quad\mbox{and} \quad 0\geq F_{1} \geq -N\left( 1+h^2 +\alpha h^{2-\delta} + 
    \alpha^2(\tau_0 -\tau)^2h^{2-2\delta}\right).
\end{equation}
for all $\alpha \geq 1$.  
\end{lemma}
\begin{proof}
 Equations \eqref{eq:phitau}, \eqref{eq:phigrad}, and \eqref{eq:phihess} follow easily from the identities for the corresponding
derivatives of $h$ in Lemma \ref{lem:hder}.  For the first inequality in \eqref{eq:phihessFest}, observe that, 
by \eqref{eq:curvdecay}, we can arrange that 
$|\Rc| \leq (1-\delta)/4$ on $\Ec_{\rad}^{\tau_0}$ by selecting $\rad\geq \rad_0$ sufficiently large. 
Since $|\tau_0| \leq 1$, the first term on the right in 
\eqref{eq:phihess} is then bounded below by $g$ on this set, and the tensor in the right factor of the second term
is bounded below by $((1+\delta)/2)g - \delta \nabla h\otimes\nabla h$.  For each $(x, \tau)\in \Ec_{\rad_0}^{\tau_0}$, 
the restriction of this latter tensor to 
the orthogonal complement of $\nabla h(x, \tau)$ is clearly positive definite, and since
\[
  \frac{1+\delta}{2}|\nabla h|^2 -\delta|\nabla h|^4 = |\nabla h|^2\left(\frac{1-\delta}{2} + \frac{4\delta\tau^2R}{h^2}\right), 
\]
by invoking \eqref{eq:curvdecay} and increasing $\rad$ if necessary, 
we may achieve that this same tensor is fully positive definite on $\Ec_{\rad}^{\tau_0}$.  
This implies the desired inequality on $\nabla\nabla\phi_1$.

For the second inequality in \eqref{eq:phihessFest}, we begin with \eqref{eq:phitau} and note that
\[
  \pd{\phi_1}{\tau} = - \alpha h^{2-\delta}\left(1 - \frac{2(2-\delta)\tau(\tau_0 - \tau)R}{h^2} - \frac{4\tau R}{\alpha h^{2-\delta}}\right)
  \leq -\frac{\alpha}{2}h^{2-\delta},
\]
if $\alpha \geq 1$. Then since our previous inequality for $\nabla\nabla \phi_1$ implies $\Delta\phi_1 \geq n$ on 
$\Ec_{\rad}^{\tau_0}$ for $\rad$ sufficiently large,
we also have
\[
  F_1 = \pd{\phi_1}{\tau} - \Delta \phi_1  -|\nabla \phi_1|^2 + R 
    \leq -\frac{\alpha}{2}h^{2-\delta} - n + CK_0r_c^{-2} \leq -\frac{1}{2}\left(\alpha h^{2-\delta} + n \right),
\]
after using \eqref{eq:curvdecay} and possibly increasing $\rad$ again by an amount determined by $n$ and $K_0$. This
gives the upper bound on $F_1$.  

For the lower bound, note that equations \eqref{eq:phigrad} and \eqref{eq:phihess} give that
\begin{align*}
\begin{split}
  |\nabla \phi_1|^2 &= \left(\alpha(2-\delta)(\tau_0-\tau)h^{1-\delta} + 2h\right)^2|\nabla h|^2, \quad\mbox{and}\\
  \Delta \phi_1 &= 2n -4\tau R  + \alpha(2-\delta)(\tau_0- \tau)(n - 2\tau R - \delta|\nabla h|^2)h^{-\delta},
\end{split}
\end{align*}
which, in combination with  \eqref{eq:hspatial2} and \eqref{eq:phitau}, yields
\begin{align}\label{eq:F_1detail}
\begin{split}
 F_1  &= (4\tau+1)R - \alpha\left(h^{2-\delta} - 2(2-\delta)(\tau_0-\tau)\tau h^{-\delta}R\right)\\
   &\phantom{=} -2n + 4\tau R - \alpha(2-\delta)(\tau_0- \tau)\left((n -\delta - 2\tau R)h^{-\delta} 
      + 4\delta\tau^2h^{-2-\delta}R\right)\\
  &\phantom{=} -\left(\alpha(2-\delta)(\tau_0-\tau)h^{1-\delta} + 2h\right)^2
		(1-4\tau^2h^{-2}R).
\end{split}
\end{align}
So, using \eqref{eq:hcurvbounds}, we have
\begin{equation*}
  |F_1| \leq N + 4h^2 + N\alpha h^{2-\delta} + N \alpha^2(\tau_0 -\tau)^2 h^{2-2\delta}
\end{equation*}
on $\Ec_{\rad}^{\tau_0}$.
\end{proof}

Next we seek a lower bound on $(\partial_{\tau} + \Delta) F_1$ in order 
to bound $Q_2$ in \eqref{eq:q1q2def} from below. We first return to the detailed expression \eqref{eq:F_1detail}
and group the terms with like powers of $\alpha$, writing 
$F_1 = B_{0} + \alpha B_{1} + \alpha^2B_{2}$.  Before differentiating, we note that
the derivatives of $h$ and $R$ are bounded on $\Ec_{\rad_0}^{\tau_0}$ by 
\eqref{eq:curvdecay} and \eqref{eq:hspatial}--\eqref{eq:pdtauh}, and 
since we also have $0\leq \tau \leq \tau_0 \leq 1$ we will really only need to consider carefully
the terms of highest order in $h$ in each $B_i$.  From \eqref{eq:F_1detail}, we see that
we in fact have
\begin{equation*}
 B_0 = -4h^2 + P_0, \quad   B_1 = -(1+ 4(2-\delta)(\tau_0-\tau))h^{2-\delta} + P_1,
\end{equation*}
and
\[
     B_2 = -(2-\delta)^2(\tau_0-\tau)^2h^{2-2\delta} + P_2,
\]
where $P_0$, $P_1$, and $P_2$ satisfy
\begin{equation*}
    \pd{P_0}{\tau}+\Delta P_0 \geq - C(K_0^2+1)h^{-2},\quad \pd{P_1}{\tau}+\Delta P_1\geq -C(K_0^2+1)h^{-\delta},
\end{equation*}
and
\begin{equation*}
	\pd{P_2}{\tau}+\Delta P_2 \geq -C(K_0^2+1)h^{-2\delta}.
\end{equation*}
for some constant $C = C(n)$. 

\begin{lemma}
\label{lem:bhFest} For all $\delta\in (0, 1)$,
there exists $\rad_2\geq \rad_0$ depending only on $n$, $\delta$, and $K_0$, such that the function 
$F_1 = F_{1; \alpha, \tau_0}$ satisfies
\begin{align}
\label{eq:Fbackheat}
\pd{F_1}{\tau} + \Delta F_{1} &\geq 3\alpha h^{2-\delta} + \alpha^2(\tau_0 - \tau)h^{2-2\delta}
\end{align} 
on $\Ec_{\rad_2}^{\tau_0}$ for all $\alpha \geq 1$ and $\tau_0 \in (0, 1]$. 
\end{lemma}

\begin{proof}
Using Lemma \ref{lem:hder}, we have
\begin{align*}
  \pd{h^{\beta}}{\tau} + \Delta h^{\beta} &= \beta(n+\beta -2)h^{\beta -2}
      -4\beta(\beta-2)\tau^2h^{\beta-4}R
\end{align*}
for any $\beta$. Consequently, using the definition of the $B_i$, we have
\begin{align*}
  \pd{B_0}{\tau} + \Delta B_0 &\geq -8n - C(K_0^2 + 1)h^{-2}, \\
  \pd{B_1}{\tau} + \Delta B_1 &\geq 4(2-\delta)h^{2-\delta} - C(K_0^2 + 1)h^{-\delta},\quad\mbox{and}\\
  \pd{B_2}{\tau} + \Delta B_2 &\geq 2(\tau_0 - \tau)(2-\delta)^2h^{2-2\delta} - C(K_0^2 + 1)h^{-2\delta}
\end{align*}
for some $C = C(n)$. Thus, since $\alpha \geq 1$ and $\delta \in (0, 1)$, we obtain that $F_1 = B_0 + \alpha B_1 + \alpha^2 B_2$ satisfies
\[
 \pd{F_1}{\tau} + \Delta F_1 \geq 3\alpha(2-\delta)h^{2-\delta} + \alpha^2(2-\delta)^2(\tau_0 - \tau)h^{2-2\delta}
\]
on $\Ec_{\rad}^{\tau_0}$ for  $\rad$ chosen sufficiently large
depending only on $n$, $\delta$, and $K_0$.
\end{proof}

\subsection{Carleman inequalities for the PDE-ODE system}
Substituting $G_{1; \alpha, \tau_0}$ for $G$ in Lemmas \ref{lem:pdecarlemanineq} and \ref{lem:odecarlemanineq} 
and using Lemmas \ref{lem:phiest} and \ref{lem:bhFest} to estimate the error terms, we now prove the first
set of our desired Carleman inequalities.
\begin{proposition}
\label{prop:systemcarlemanineq}
 For all $\delta \in (0, 1)$, There exists
 $\rad_3=\rad_3(n, \delta, K_0)\geq \rad_0$ such that, for  all $\alpha \geq 1$
 and all $Z\in C^\infty(\mathcal{Z}\times [0,\tau_0])$  satisfying $Z(\cdot, 0)=0$ and that
$Z(\cdot, \tau)$
 is compactly
supported in $E_{\rad_3}$ for each $\tau \in [0, \tau_0]$, we have the estimate
\begin{align}
\label{eq:systemcarlemanineq1}
\begin{split}
&\alpha\|ZG_1^{1/2}\|^2_{L^2(\Ec_{\rad_3}^{\tau_0})}+ \|\nabla Z G_1^{1/2}\|_{L^2(\Ec_{\rad_3}^{\tau_0})}^2 
\\
  &\qquad\le \frac{1}{2}\|(\partial_{\tau} +\Delta) ZG_1^{1/2}\|_{L^2(\Ec_{\rad_3}^{\tau_0})}^2
 + \|\nabla Z G_1^{1/2}\|^2_{L^2(E_{\rad_3}\times\{\tau_0\})}
\end{split}
\end{align}
and
\begin{align} 
\label{eq:systemcarlemanineq2}
&\alpha\|Z G_1^{1/2}\|_{L^2(\Ec_{\rad_3}^{\tau_0})}^2 \le 2\|\partial_{\tau}ZG_1^{1/2}\|_{L^2(\Ec_{\rad_3}^{\tau_0})}^2,
\end{align}
where $G_1 = G_{1; \alpha, \tau_0}$.
\end{proposition}
\begin{proof}
We apply Lemmas \ref{lem:phiest} and \ref{lem:bhFest} and let $\rad_3 \geq \max\{\rad_1, \rad_2\}$ initially.
Below, $N$ will denote a series of constants depending only on $n$ and $K_0$.
If $\alpha \geq 1$, then \eqref{eq:hrbounds} and \eqref{eq:phihessFest} imply that
\[
  r_c^{-2}|F_1| \leq N\alpha + N\alpha^2(\tau_0-\tau)^2h^{-2\delta}.
\]
Then, from \eqref{eq:phihessFest} and \eqref{eq:Fbackheat} (and since $0 \leq (\tau_0 -\tau)\leq 1$), we have
\[
 Q_1(\nabla Z, \nabla Z) \geq \left(2- \frac{N}{r_c^2}\right)|\nabla Z|^2,
\]
and
\begin{align*}
 Q_2(Z, Z) &\geq \left(\frac{3\alpha}{2} h^{2-\delta} - \frac{N|F_1|}{r_c^2}\right)|Z|^2
+ \frac{\alpha^2}{2}(\tau_0 - \tau)h^{2-2\delta}|Z|^2\\
      &\geq \alpha h^{2-\delta}\left(\frac{3}{2}- \frac{N}{h^{2-\delta}}\right)|Z|^2 
	    +\alpha^2(\tau_0 - \tau)h^{2-2\delta}\left(\frac{1}{2} - \frac{N}{h^2}\right)|Z|^2.
\end{align*}
So, enlarging $\rad_3$ if necessary, we can arrange that
\[
  Q_1(\nabla Z, \nabla Z) + Q_2(Z, Z) \geq \frac{3}{2}|\nabla Z|^2 + 
\left(\frac{4\alpha}{3}+ \frac{\alpha^2}{3}(\tau_0 - \tau)h^{2-2\delta}\right)|Z|^2
\]
on $\Ec_{\rad_3}^{\tau_0}$.  For \eqref{eq:systemcarlemanineq1}, then, 
it remains only to estimate the $E(Z, \nabla Z)$ term from Lemma \ref{lem:pdecarlemanineq}. 
Writing $\nabla \phi_1 = \Upsilon_1 \nabla f$,
where $\Upsilon_1 \dfn 2\tau(\alpha(2-\delta)(\tau_0-\tau)h^{-\delta} + 2)$, we may apply  \eqref{eq:eequation2} of Remark \ref{rem:commutator} to obtain
\[
-E(Z, \nabla Z) \geq -Nh^{-2}\left(|\nabla Z|^2 + \left(1+ \alpha^2\tau^2(\tau_0-\tau)^2h^{-2\delta}\right)|Z|^2\right)
\]
for some $N = N(n, \kappa, \nu, K_0)$.  Thus  after potentially increasing
$\rad_3$ again, we have
\[
    -E(Z, \nabla Z) \geq -\frac{1}{4}|\nabla Z|^2 - \frac{1}{6}|Z|^2 - \frac{\alpha{^2}(\tau_0-\tau)h^{2-2\delta}}{6}|Z|^2
\]
on $\Ec_{\rad_3}^{\tau_0}$, and \eqref{eq:systemcarlemanineq1} follows from Lemma \ref{lem:pdecarlemanineq}.

For inequality \eqref{eq:systemcarlemanineq2}, observe that
\[
  \pd{\phi_1}{\tau} = 4\tau R - \alpha h^{2-\delta}\left(1 - \frac{2(2-\delta)\tau(\tau_0 - \tau)R}{h^2}\right)
  \geq -\frac{2\alpha}{3}h^{2-\delta},
\]
from \eqref{eq:phitau}, and thus the desired inequality follows from \eqref{eq:odecarlemanineq},
by choosing $\rad_3$ sufficiently large. 
\end{proof}


\section{Carleman estimates to imply rapid decay}
\label{sec:carleman2}

Since the weight function $G_{1; \alpha, \tau_0}$ in the previous section 
has growth of order $\exp(Nr_c^2)$ at infinity, we cannot make use of estimates 
\eqref{eq:systemcarlemanineq1} -- \eqref{eq:systemcarlemanineq2}
until we guarantee that any solution to the PDE-ODE system \eqref{eq:pdeode} which vanishes on 
$E_{\rad_0}\times\{0\}$ decays at a correspondingly rapid rate.
We verify this decay with the help of another pair of Carleman estimates.  Our preliminary model is inequality (1.4) of \cite{EscauriazaSereginSverak},
which, writing $\sigma_a(\tau) = (\tau+a)e^{-(\tau+a)/3}$, asserts that, for some constant $C = C(n)$,
\begin{align*}
  &\sqrt{\alpha}\|\sigma_a^{-\alpha-\frac12} e^{-\frac{|x - y|^2}{8(\tau+a)}}u\|_{L^2(\RR^n\times(0, 1))}
+  \|\sigma_a^{-\alpha} e^{-\frac{|x - y|^2}{8(\tau+a)}}\nabla u\|_{L^2(\RR^n\times(0, 1))}\\
  &\quad \quad \leq C\|\sigma_a^{-\alpha} e^{-\frac{|x - y|^2}{8(\tau+a)}}(\partial_{\tau} + \Delta)u\|_{L^2(\RR^n\times(0, 1))}
\end{align*}
for all $\alpha \geq 0$, $y\in \RR^n$, $a\in (0, 1)$, and $u\in C^{\infty}_c(\RR^n\times[0, 1))$ satisfying $u(\cdot, 0) \equiv 0$.
We wish to find a generalization of this inequality to our geometric setting.

Replacing $\alpha$ with $\alpha + n/2$ in the above inequality, for example, one can see that the basic ingredient in the weight
is the time-shifted Euclidean heat kernel $(\tau + a)^{-n/2}e^{-|x-y|^2/(4(\tau+a))}$. The proof and subsequent
application of this estimate in \cite{EscauriazaSereginSverak} are considerably simplified by the fact that the weight is
an exact solution to the heat equation and possesses a translational invariance in $y$.  
Neither of these properties, however, are essential to verifying the decay we are after,
and with ``approximately radial, approximately caloric'' weight $G_2$, we are able to prove a
 weaker but still sufficiently powerful variant of their estimate applicable to the PDE component of our system.
Our prototype is the inequality 
\begin{align*}
 &\sqrt{\alpha}\|\sigma_a^{-\alpha-\frac12} e^{-\frac{(|x|- \rho)^2}{8(\tau+a)}}u\|_{L^2(\RR^n\times(0, 1))}
+  \|\sigma_a^{-\alpha} e^{-\frac{(|x|-\rho)^2}{8(\tau+a)}}\nabla u\|_{L^2(\RR^n\times(0, 1))}\\
  &\quad \quad \leq C(\gamma, n)\|\sigma_a^{-\alpha} e^{-\frac{(|x|-\rho)^2}{8(\tau+a)}}(\partial_{\tau} + \Delta)u\|_{L^2(\RR^n\times(0, 1))}
\end{align*}
with $\sigma_a$ as above and $\gamma \geq 1$ some fixed number, valid for all $\alpha \geq \alpha^{\prime}(\gamma, n) \geq 0$, $\rho \geq 1$, 
$a\in (0, 1)$, and $u\in C^{\infty}_c(\{\,|x|> \gamma\rho\}\times [0, 1))$ vanishing for $\tau = 0$.

It is worth remarking that, e.g., via a scaling argument applied relative to a finite fixed atlas,
the decay condition we seek can be reduced in principle to a local verification.  Escauriaza-Fernandez \cite{EscauriazaFernandez}
(cf. \cite{Nguyen}) have considered such problems for a very general class of parabolic equations with time-dependent coefficients,
and their estimates offer another potential model for the estimate on our PDE component.  However, since the elliptic operators 
in our problem are actually Laplacians relative to $g(\tau)$ (and so also perturbations of a conical
Laplacian), our situation is fundamentally simpler than that of \cite{EscauriazaFernandez}, 
and we find that the approach of \cite{EscauriazaSereginSverak} yields estimates with somewhat more transparent geometric interpretations.
In this approach it is possible to get by with far less complicated weights, the use of which also greatly simplifies
 the proof of the corresponding estimates for the ODE components. 

\subsection{Another divergence identity}
As in the previous section, our estimate will follow from 
integrating a general divergence identity against an appropriate weight. 
In this case, our choice of weight $G$ will be a perturbation
of a fundamental solution and so not itself be logarithmically convex.
In order to use an inequality of the form in Lemma \ref{lem:pdecarlemanineq}
to control $|\nabla Z|$ above by $|(\partial_{\tau} + \Delta)Z|$,
we must first tinker with the divergence identity
to increase the effective logarithmic convexity of $G$.
 
Thus, as in \cite{EscauriazaSereginSverak}, we introduce additional positive time dependent functions  $\sigma =\sigma(\tau)$ and $\theta = \theta(\tau)$,
stipulating only for the time-being that $\sigma$ be increasing.
Replacing $G$ in \eqref{eq:dividentity1} by $\sigma^{-\alpha}G$,
multiplying both sides by $\theta$, and using the product rule to bring the $\theta$ factor inside the time-derivative
of the last term on the left-hand side of that equation, we obtain the following perturbed identity.

\begin{lemma}\label{lem:dividentity2}  For any $F$, $G \in C^{\infty}(\Ec_{\rad_0}^{\tau_0})$ with $G > 0$, 
and positive functions $\sigma$, $\theta\in C^{\infty}([0, \tau_0])$ with $\sigma$ increasing, the following
identity holds for any $Z\in C^{\infty}(\mathcal{Z}\times[0, \tau_0])$: 
\begin{align}\label{eq:dividentity2}
\begin{split}
  &\theta\sigma^{-\alpha}\nabla_i\bigg\{2\left\langle \pd{Z}{\tau}, \nabla_i Z\right\rangle G + |\nabla Z|^2\nabla_i G
  -2\left\langle\nabla_{\nabla G}Z, \nabla_i Z\right\rangle + \frac{FG}{2}\nabla_i|Z|^2\\
  &\quad\phantom{=}\quad + \frac{1}{2}\left(F\nabla_i G - G\nabla_i F\right)|Z|^2\bigg\} d\mu
	 -\pdtau\bigg\{\left(|\nabla Z|^2 + \frac{F}{2}|Z|^2\right)\theta \sigma^{-\alpha}G d\mu\bigg\}\\
   &\quad= \bigg\{2\left\langle \mcA Z, \pd{Z}{\tau} + \Delta Z\right\rangle  - 2\left|\mcA Z\right|^2
    -\frac{1}{2}\left(\pd{F}{\tau} + \Delta F\right)|Z|^2 \\
  &\quad\phantom{=}\quad +\left(F - \left(G^{-1}\left(\pd{G}{\tau} 
      - \Delta G\right) + R-\alpha\frac{\dot{\sigma}}{\sigma}\right)\right)\left(|\nabla Z|^2 +\frac{F}{2}|Z|^2\right) \\
  &\quad\phantom{=}\quad -\frac{\dot{\theta}}{\theta}\left(|\nabla Z|^2 + \frac{F}{2}|Z|^2\right)
- 2\nabla_i\nabla_j\phi\langle \nabla_iZ, \nabla_jZ\rangle \\
  &\quad\phantom{=}\quad -\pd{g}{\tau}(\nabla Z, \nabla Z) -\frac{F}{2}\pd{g}{\tau}(Z, Z)  
      +2 E(Z, \nabla Z)\bigg\}\theta\sigma^{-\alpha} G d\mu.
\end{split}
\end{align}
Here, $\mcA$ and $E(Z, \nabla Z)$ are defined as in \eqref{eq:asdef} and \eqref{eq:edef}, respectively,
and the instances of 
$\pd{g}{\tau}$ are to be interpreted as in Lemma \ref{lem:dividentity1}.
\end{lemma}

\subsection{Two variations on the weighted $L^2$-inequality for the PDE component}
In our application of interest, we will first choose
\[
 F \dfn  \frac{1}{G}\left(\pd{G}{\tau} -\Delta G\right) + R -\alpha\frac{\dot{\sigma}}{\sigma}
    \dfn \tilde{F}  -\alpha\frac{\dot{\sigma}}{\sigma}
\]
to eliminate the fourth term on the right-hand side of \eqref{eq:dividentity2}. Then, choosing 
$\theta\dfn \sigma/\dot{\sigma}$ as in \cite{EscauriazaSereginSverak}, we have
\[
 \frac{\dot{\theta}}{\theta}\frac{\dot{\sigma}}{\sigma} = \frac{\dot{\sigma}^2}{\sigma^2}
    \left(1 - \frac{\ddot{\sigma}\sigma}{\dot{\sigma}^2}\right)
      = -\ddot{\widehat{\log\sigma}},
\]
which leads to a useful cancellation among the coefficients of $|Z|^2$ in \eqref{eq:dividentity2}:
\[
 \pd{F}{\tau} + \Delta F + \frac{\dot{\theta}}{\theta}F = 
  \pd{\tilde{F}}{\tau} + \Delta{\tilde{F}}  -\alpha\ddot{\widehat{\log\sigma}}
  + \frac{\dot{\theta}}{\theta}\left(\tilde{F}-\alpha\frac{\dot{\sigma}}{\sigma}\right)=  
    \pd{\tilde{F}}{\tau} + \Delta{\tilde{F}} +\frac{\dot{\theta}}{\theta}\tilde{F}.
\]
Finally, using the good $-2|\mcA Z|^2$ term in \eqref{eq:dividentity2} together with the
Cauchy-Schwarz
inequality, we obtain the following estimate upon integration over $\Ec_{\rad_0}^{\tau_0}$.
\begin{lemma}\label{lem:pdeintineq2}
There exists a constant $N = N(n, \kappa, \nu, K_0)$ such that, for any $\alpha$ and any 
$Z\in C^{\infty}(\mathcal{Z}\times[0, \tau_0])$
that is compactly supported in $E_{\rad_0} \times[0, \tau_0)$ and vanishes on $E_{\rad_0}\times\{0\}$, the inequality 
 \begin{align}\label{eq:pdeintident2}
 \begin{split}
   &\iint_{\Ec_{\rad_0}^{\tau_0}}\,\frac{\sigma^{1-\alpha}}{\dot{\sigma}}\big(Q_3(\nabla Z, \nabla Z) + Q_4(Z, Z) 
 - 2E(Z, \nabla Z)\big)G\,d\mu\,d\tau\\
 &\quad\quad
 \leq \iint_{\Ec_{\rad_0}^{\tau_0}}\frac{\sigma^{1-\alpha}}{\dot{\sigma}}
 	\left|\pd{Z}{\tau} + \Delta Z\right|^2\,G\,d\mu\,d\tau
 \end{split}
 \end{align}
holds, where
 \begin{align} 
 \begin{split}\label{eq:q3def}
   Q_3(\nabla Z, \nabla Z)  &\geq \left(2\nabla_i\nabla_j\phi  
     -\frac{\sigma}{\dot{\sigma}}\ddot{\widehat{\log\sigma}}g_{ij}\right)\langle \nabla_i Z, \nabla_j Z\rangle
    -\frac{N}{r_c^2}|\nabla Z|^2,
\end{split}
\end{align}
and
\begin{align} 
\begin{split}\label{eq:q4def}
  Q_4(Z, Z) &\geq \frac{1}{2}\left(\pd{\tilde{F}}{\tau} + \Delta \tilde{F} +\frac{\dot{\theta}}{\theta}\tilde{F}\right)|Z|^2
    -\frac{N}{r_c^2}|F||Z|^2.
\end{split}
\end{align}
\end{lemma}

In order to use the above inequality to control $|\nabla Z|$ above by $|(\partial_{\tau} + \Delta)Z|$,
we require an additional inequality to help us estimate $|Z|$ above by controllably small multiples of
$|\nabla Z|$ and $|(\partial_{\tau}+ \Delta)Z|$.
Its proof is very simple. Observe that, on one hand, we have the identity
\begin{align}\label{eq:dividentity3}
\begin{split}
  \left(\pdtau + \Delta\right)|Z|^2 = \pd{g}{\tau}(Z, Z) 
+ 2\left\langle \pd{Z}{\tau} + \Delta Z, Z\right\rangle + 2|\nabla Z|^2, 
\end{split}
\end{align}
while on the other (with $\tilde{F} = G^{-1}(\partial_{\tau}G -\Delta G) + R$ as before), we have
\begin{align}\label{eq:dividentity4}
\begin{split}
    &\sigma^{-2\alpha}\left(\pdtau + \Delta\right)|Z|^2 G\,d\mu 
 =  \sigma^{-2\alpha}\nabla_i\left(\nabla_i|Z|^2 G - |Z|^2\nabla_iG\right)\,d\mu\\
&\qquad\qquad+\pdtau\bigg\{\sigma^{-2\alpha}|Z|^2 G\,d\mu \bigg\} +\sigma^{-2\alpha}\left(
  2\alpha\frac{\dot{\sigma}}{\sigma}- \tilde{F}\right)|Z|^2G \,d\mu.
\end{split}
\end{align}
The inequality follows by integrating \eqref{eq:dividentity4} over $\Ec_{\rad_0}^{\tau_0}$ for appropriately supported sections $Z$, and using \eqref{eq:dividentity3} 
together with Cauchy-Schwarz.
\begin{lemma}\label{lem:pdeintineq3} There exists a constant $N = N(n, \kappa, \nu, K_0)$ such that,
for all $\alpha > 0$, all smooth positive $G = G(x, \tau)$, and all positive increasing $\sigma = \sigma(\tau)$, we have the inequality
\begin{align}
\begin{split}\label{eq:pdeintineq3}
  &\iint_{\Ec_{\rad_0}^{\tau_0}}\,\sigma^{-2\alpha}\left(\alpha\frac{\dot{\sigma}}{\sigma}-\tilde{F} - \frac{N}{r_c^2}\right)|Z|^2\,G\,d\mu\,d\tau\\
  &\qquad \leq\iint_{\Ec_{\rad_0}^{\tau_0}}
    \,\sigma^{-2\alpha}\left(2|\nabla Z|^2+\frac{\sigma}{\alpha\dot{\sigma}}\left|\pd{Z}{\tau} + \Delta Z\right|^2
	    \right)G\,d\mu\,d\tau
\end{split}
\end{align}
for all $Z\in C^{\infty}(\mathcal{Z}\times[0, \tau_0])$ with compact support in
$E_{\rad_0}\times[0, \tau_0)$ vanishing on $E_{\rad_0}\times\{0\}$.
\end{lemma}

\subsection{An approximate solution to the conjugate heat equation}

Let $h: \Ec_{\rad_0}^{\tau_0} \to \RR$ be as defined in Section \ref{ssec:hdef} and, for any $a\in (0, 1)$ and $\rho\in (\rad_0, \infty)$, define
\begin{equation}\label{eq:g2def}
    G_2(x, \tau) \dfn G_{2; a, \rho}(x, \tau) = (\tau + a)^{-n/2}\exp{\left(-\frac{(h(x, \tau) - \rho)^2}{4(\tau + a)}\right)}
\end{equation}
on $\Ec_{\rad_0}^{\tau_0}$. In view of the bounds \eqref{eq:hrbounds}, $G_{2}$ is localized around the set $\{\,r_c(x) = \rho\}$, and, 
in a manner we will make precise below, approximately solves the (forwards) conjugate heat equation in $\tau$.  Note that $G_{2; 0,0} = \tau^{-n/2}e^{-f}$ exactly
satisfies
\[
  \partial_{\tau}G_{2; 0, 0} - \Delta G_{2; 0, 0} + RG_{2; 0, 0} = 0.
\]

\subsubsection{Estimates on the derivatives of $G_2$.}
We first use Lemma \ref{lem:hder} to compute the derivatives of $G_2 = G_{2; a, \rho}$.  We have
\begin{align}
\label{eq:pdtauG}
\begin{split}
G_2^{-1}\pd{G_2}{\tau} = \frac{(h-\rho)^2}{4(\tau + a)^2} - \frac{\tau R(h-\rho)}{h(\tau + a)} - \frac{n}{2(\tau+a)},
\end{split}
\end{align}
and since $G_2^{-1}\nabla G_2 = -(h-\rho)/(2(\tau+a))\nabla h$, we compute that
\begin{align}
\begin{split}
\label{eq:hessG}
  G_2^{-1}\nabla\nabla G_2 
		 &= -\frac{1}{2(\tau+a)}g + \frac{\tau}{(\tau+a)}\Rc(g) 
      + \frac{\rho}{2(\tau+a)}\nabla\nabla h \\
  &\phantom{=}+\frac{(h-\rho)^2}{4(\tau+a)^2}\nabla h\otimes\nabla h
\end{split}
\end{align}
and
\begin{align}
\begin{split}\label{eq:deltaG}
 G^{-1}_2\Delta G_2
  &= -\frac{n}{2(\tau+a)} + \frac{\tau R}{\tau+a} + \frac{\rho}{2(\tau+a)}\Delta h + \frac{(h-\rho)^2}{4(\tau+a)^2}|\nabla h|^2.
\end{split}
\end{align}
Thus, combining the above equations, we obtain
\begin{align}
 \begin{split}\nonumber
  G^{-1}_2\left(\pd{G_2}{\tau} - \Delta G_2\right) &= \frac{(h-\rho)^2}{4(\tau + a)^2}\left(1-|\nabla h|^2\right)
  - \frac{\tau R}{\tau+a}\left(2-\frac{\rho}{h}\right)
 - \frac{\rho\Delta h}{2(\tau+a)}
 \end{split}\\
\begin{split}\label{eq:heatG}
  &= \left(\frac{a^2}{(\tau+a)^2} - 1\right)R- \frac{(n-1)\rho}{2h(\tau+a)} \\
  &\phantom{=}
    + \frac{2\tau\rho R}{h(\tau+a)}\left(1 -\frac{\tau}{(\tau+a)} + \frac{\rho\tau}{2h(\tau+a)} -\frac{\tau}{h^2}    \right)
\end{split}
\end{align}
so
\begin{align}
\begin{split}\label{eq:conjheatG}
&G_2^{-1}\left(\pd{G_2}{\tau} - \Delta G_2\right) + R =  - \frac{(n-1)\rho}{2h(\tau+a)} + \frac{a^2R}{(\tau+a)^2}\\
&\qquad\qquad\phantom{=}+ \frac{2\tau\rho R}{h(\tau+a)}\left(\frac{a}{(\tau+a)} + \frac{\rho\tau}{2h(\tau+a)}
    -\frac{\tau}{h^2}\right).
\end{split}
\end{align}
We now combine the above observations with \eqref{eq:hcurvbounds}, using the notation
\[
    F_2 \dfn F_{2; a, \rho} \dfn G_2^{-1}\left(\pd{G_2}{\tau} -\Delta G_2\right) + R -\alpha\frac{\dot{\sigma}}{\sigma},
    \quad\mbox{and}\quad
    \tilde{F}_2 \dfn \tilde{F}_{2; a, \rho}\dfn  F_2  +\alpha\frac{\dot{\sigma}}{\sigma}.
\]

\begin{lemma} For all $a\in (0, 1)$,  $\gamma > 0$, and $\rho > \rad_0\geq 1$,
 there exists a constant $C= C(n, \gamma) > 0$ and $\rad_4 \geq\rad_0$, depending only on $n$, $\gamma$ and $K_0$, such that
$G_2 = G_{2; a, \rho}$ and $\tilde{F}_2 = \tilde{F}_{2; a, \rho}$
 satisfy
\begin{equation}\label{eq:gghatcomp}
 \frac{1}{2}e^{-\frac{(r_c(x)-\rho)^2}{4(\tau+a)}} \leq (\tau+a)^{n/2}G_{2}(x, \tau) \leq 2e^{-\frac{(r_c(x)-\rho)^2}{4(\tau+a)}}, 
\end{equation}
and
\begin{equation}\label{eq:tildeFest}
 - \frac{(n-1)\rho}{2h(\tau+a)} -\frac{1}{8h}\leq\tilde{F}_2 \leq - \frac{(n-1)\rho}{2h(\tau+a)} +\frac{1}{8h} \leq 0,
\end{equation}
and $\phi_2\dfn \log G_2$ satisfies
\begin{equation}\label{eq:hesslogG}
  \nabla\nabla \phi_2 \geq -\frac{g}{2(\tau+a)} - \frac{g}{48}
\end{equation}
 on the set $(E_{\rad_4}\cap E_{\gamma\rho})\times[0, \tau_0]$.
\end{lemma}

\begin{proof}
First note that, using our curvature decay assumption \eqref{eq:hcurvbounds}, \eqref{eq:rhcomparison},
and our assumption that $a$, $\tau_0 \leq 1$, we have
that
\[
	\frac{(r_c - \rho)^2}{(\tau+a)} - \frac{CK_0\tau^2}{(\tau+a)r_c^3}\leq
	\frac{(h -\rho)^2}{(\tau+a)} \leq \frac{(r_c - \rho)^2}{(\tau+a)} + \frac{CK_0\tau^2}{(\tau+a)r_c^3}
\]		
Thus \eqref{eq:gghatcomp} is valid on $\Ec_{\rad}^{\tau_0}$ for $\rad\geq \rad_0$ sufficiently large. 

For \eqref{eq:tildeFest}, observe that we may
estimate the second term on the right side of \eqref{eq:conjheatG} by
\[
 \left|\frac{a^2R}{(\tau+a)^2}\right| \leq |R| \leq \frac{C^{\prime}K_0}{h^2}, 
\]
for $C^{\prime} = C^{\prime}(n)$
and, using  $\rho h^{-1} \leq 2\rho r_c^{-1} \leq 2\gamma^{-1}$, the third term on the right of the same equation by
\[
 \frac{2\tau\rho |R|}{h(\tau+a)}\left|\frac{a}{(\tau+a)} + \frac{\rho\tau}{2h(\tau+a)} -\frac{\tau}{h^2}\right|
\leq \frac{C^{\prime\prime}K_0}{h^2}
\]
for $C^{\prime\prime}=C^{\prime\prime}(\gamma, n)$
Summing these two inequalities, we see that by choosing $\rad$ large enough, 
we can ensure that $(C^{\prime}+C^{\prime\prime})K_0/h\leq 1/8$ (and, in particular, that $\tilde{F}_2\leq 0$)  
on $\left(E_{\rad}\cap E_{\gamma\rho}\right)\times[0, \tau_0]$. 

For \eqref{eq:hesslogG}, we may use Lemma \ref{lem:hder} to compute that
\begin{align*}  
\nabla\nabla \phi_2 &= -\frac{1}{2(\tau+a)}\left( (h-\rho)\nabla\nabla h + \nabla h\otimes \nabla h\right)\\
		     &=  -\frac{g}{2(\tau+a)} + \frac{\tau \Rc}{(\tau+a)} + \frac{\rho}{2h(\tau+a)}\left(g - 2\tau\Rc - \nabla h\otimes\nabla h\right)\\
		     &= -\frac{g}{2(\tau+a)} + \frac{\tau(h-\rho) \Rc}{h(\tau+a)} + \frac{\rho g}{2h(\tau+a)} -\frac{\rho\nabla h\otimes\nabla h}{2h(\tau+a)},
\end{align*}
and from the bounds $|\Rc| \leq C(n)K_0h^{-2}$ and $|\nabla h|^2 \leq 1 + C(n)K_0\tau^2h^{-4}$ available to us on $\Ec_{\rad_0}^{\tau_0}$,
we can obtain a constant $C^{\prime\prime\prime} = C^{\prime\prime\prime}(n, \gamma)$ such that
\[
  \frac{\tau(h-\rho)\Rc}{h(\tau+a)} \geq  -\frac{C^{\prime\prime\prime}K_0g}{h^{2}}\,\quad\mbox{and}\quad
 \frac{\rho\nabla h\otimes\nabla h}{2h(\tau+a)} \leq \frac{\rho g}{2h(\tau+a)} + \frac{C^{\prime\prime\prime}K_0g}{h^4}.
\]
Thus, for large $\rad$ we may achieve $C^{\prime\prime\prime}K_0/h \leq 1/96$ and hence the inequality
 \eqref{eq:hesslogG} on $\left(E_{\rad}\cap E_{\gamma\rho}\right) \times[0, \tau_0]$.
\end{proof}

\subsubsection{Estimates on the derivatives of $\tilde{F_2}$}
In order to estimate the $Q_4$ term from \eqref{eq:pdeintident2}, we still need to compute 
$(\partial_{\tau} + \Delta)\tilde{F}_2$.  Returning to 
\eqref{eq:conjheatG}, we group terms with like powers of $h^{-1}$ and write $\tilde{F}_2$ in the form
\[
    \tilde{F}_2 = -\frac{(n-1)\rho}{2(\tau + a)h} + R H\left(\frac{\rho}{h}, \tau\right)
\]
where
\[
  H(s, \tau) \dfn \frac{a^2}{(\tau+a)^2} + \frac{2a\tau }{(\tau+a)^2}s
    + \frac{\tau^2}{(\tau+a)^2}s^2 - \frac{2\tau^2}{\rho^2(\tau+a)}s^3.
\]
Then, differentiating, we obtain the equations
\begin{align*}
 \pd{\tilde{F}_2}{\tau} &= \frac{(n-1)\rho}{2(\tau+a)^2h} + \frac{(n-1)\rho}{2(\tau+a)h^2}\pd{h}{\tau}
      +H\pd{R}{\tau} + H_{\tau}R - \frac{\rho H_s R}{h^2}\pd{h}{\tau},\\
\nabla \tilde{F}_2 &= \frac{(n-1)\rho}{2(\tau+a)}\frac{\nabla h}{h^2} + H\nabla R - \frac{\rho H_s R}{h^2}\nabla h,\\
\begin{split}
    \Delta \tilde{F}_2 &= \frac{(n-1)\rho}{2(\tau+a) h^2}\Delta h -\frac{(n-1)\rho}{(\tau+a)h^3}|\nabla h|^2
		+H\Delta R - \frac{2\rho H_s }{h^2}\langle \nabla R, \nabla h\rangle\\
	    &\phantom{=} + \frac{\rho^2H_{ss}R}{h^4}|\nabla h|^2
	    -\frac{\rho H_s R}{h^2}\Delta h + \frac{2\rho H_sR}{h^3}|\nabla h|^2.
\end{split}
\end{align*}
Put together, we have
\begin{align*}
\begin{split}
\pd{\tilde{F}_2}{\tau} + \Delta \tilde{F}_2 &=
  \frac{(n-1)\rho}{2(\tau+a)^2h} + \frac{(n-1)\rho}{2(\tau+a)h^2}\left(\pd{h}{\tau}+ \Delta h 
      - \frac{2|\nabla h|^2}{h}\right)\\
 &\phantom{=}	+ H\left(\pd{R}{\tau} + \Delta R\right)
  + H_{\tau}R - \frac{\rho H_sR}{h^2}\left(\pd{h}{\tau} + \Delta h\right)\\
     &\phantom{=}   -\frac{2\rho H_s}{h^2}\langle \nabla R, \nabla h\rangle
+\frac{\rho R}{h^3}\left(2H_s + \frac{\rho H_{ss}}{h}\right)|\nabla h|^2,
\end{split}
\end{align*}
which, after applying Lemma \ref{lem:hder} and rearranging terms, becomes
\begin{align*}
\begin{split}
\pd{\tilde{F}_2}{\tau} + \Delta \tilde{F}_2&=
  \frac{(n-1)\rho}{2(\tau+a)^2h} + \frac{(n-1)\rho}{2(\tau+a)h^3}\left((n-3) + \frac{12\tau^2R}{h^2}\right)\\
 &\phantom{=}	- 2H|\Rc|^2
  + H_{\tau}R - \frac{\rho H_sR}{h^3}\left((n-1) +\frac{4\tau^2 R}{h^2}\right)\\
     &\phantom{=}   -\frac{2\rho H_s}{h^2}\langle \nabla R, \nabla h\rangle
+\frac{\rho R}{h^3}\left(2H_s + \frac{\rho H_{ss}}{h}\right)\left(1 -\frac{4\tau^2R}{h^2}\right).
\end{split}
\end{align*}

Fortunately, we will not need to analyze the complicated right-hand side of this equation
too carefully.  For our purposes, the dominant term is the first -- the others, as we see next, are either of lower order in $(\tau+a)^{-1}$, or higher order
in $h^{-1}$ (e.g., through factors of $|\Rc|$ or $R$) and can be made to be as small as we like
after further shrinking our end by a fixed amount.

\begin{lemma}\label{lem:bheatF}
  For all $a\in (0, 1)$, $\gamma > 0$, and $\rho > \rad_0$,
 there exist constants $N$ and $\rad_5 \geq\rad_0$ both depending only on $n$, $\gamma$ and $K_0$ such that 
$\tilde{F}_2 = \tilde{F}_{2; a,\rho}$ satisfies
\begin{equation}\label{eq:bheatF}
  \pd{\tilde{F}_2}{\tau} + \Delta \tilde{F}_2 \geq \frac{(n-1)\rho}{2h(\tau+a)^2} - \frac{N}{(\tau+a)}
\end{equation}
 on the set $\left(E_{\rad_5}\cap E_{\gamma\rho}\right)\times[0, \tau_0]$.
\end{lemma}
\begin{proof}
Observe that, for $\rad\geq \rad_0$ sufficiently large (depending only on $\gamma$), we can ensure that
\[
    |H(\rho h^{-1}, \tau)| + |H_{s}(\rho h^{-1}, \tau)| + |H_{ss}(\rho h^{-1}, \tau)| \leq C(n)(\gamma^{-1}+\gamma^{-3})
\]
on $E_{\rad}\cap E_{\gamma\rho}\times[0, \tau]$. Indeed, estimating $\tau/(\tau+ a)$
and $a/(\tau+a)$ above by $1$, and using that $h/\rho \leq \gamma^{-1}$, we see that each term in $H$ is bounded above
by a constant depending only on $\gamma$, and the statements for $H_{s}(\rho h^{-1}, \tau)$ 
and $H_{ss}(\rho h^{-1}, \tau)$ follow from the fact that $H$ is polynomial in $s$.  By similar reasoning,
we obtain a bound of the form $|H_{\tau}|\leq C(n)(\gamma^{-1} + \gamma^{-3})(\tau+a)^{-1}$ for analogously
restricted $(x, \tau)$.  So $|H_{\tau}R| \leq C(n, \gamma)/(\tau+a)$ on $\Ec_{\rad}^{\tau_0}$ for $\rad$ 
taken sufficiently large (depending on $n$, $\gamma$, and $K_0$). Using Lemma \ref{lem:hder} and \eqref{eq:hcurvbounds},
we can bound all the remaining terms similarly.
\end{proof}

\subsection{A Carleman inequality for the PDE component}
Now we return to the integral inequality \eqref{eq:pdeintident2} and substitute $G_2 = G_{2; a, \rho}$
for the weight $G$. For the time-dependent weight $\sigma$, following \cite{EscauriazaSereginSverak}, we define
$\sigma(\tau) \dfn \tau e^{-\tau/3}$ and its translates $\sigma_a(\tau) \dfn \sigma(\tau +a)$ for $a\in (0, 1)$.
Note that $\sigma_a$ is approximately linear in that
\begin{align}
\begin{split}
\label{eq:sigmabounds}
 &\frac{1}{3e}(\tau + a) \leq \sigma_a(\tau) \leq (\tau + a) \quad\mbox{and}\quad \frac{1}{3e} \leq \dot{\sigma}_a(\tau) \leq 1\\
\end{split}
\end{align}
for $\tau\in [0, 1]$. Additionally, $\sigma_{a}$ satisfies 
\begin{equation}\label{eq:logthetadot}
   -\frac{\sigma_a}{\dot{\sigma}_a}\ddot{\widehat{\log \sigma_a}}
	=\frac{1}{(\tau + a)(1-(1/3)(\tau+a))}.
\end{equation}

We now verify lower bounds for the forms $Q_3$ and $Q_4$ from Lemma \ref{lem:pdeintineq2}.
\begin{lemma}\label{lem:quadlowerbounds}
For all $\alpha > 0$, $a\in (0, 1)$, and $\gamma > 0$, there exist $N$ and $\rad_6\geq \rad_0$ depending on $n$, $\gamma$, $\kappa$, $\nu$, and $K_0$, 
such that the quadratic forms $Q_3$ and $Q_4$ from \eqref{eq:q3def} and \eqref{eq:q4def} and the commutator term
$E(Z, \nabla Z)$ from \eqref{eq:edef} (depending on $\phi_{2; a, \rho} = \log G_{2; a, \rho}$)
satisfy
\begin{equation}\label{eq:q3q4ineq}
 Q_3(\nabla Z, \nabla Z) + Q_4(Z, Z) -2 E(Z, \nabla Z) \geq 
\frac{1}{4}|\nabla Z|^2 -\sigma_a^{-1}\left(N + \frac{\alpha}{10000}\right)|Z|^2
\end{equation}
on $\left(E_{\rad_6}\cap E_{\gamma\rho}\right)\times[0, \tau_0]$.
\end{lemma}
\begin{proof}
We begin by assuming that $\rad > \max\{\rad_3, \rad_4, \rad_5\}$.  In the argument that follows,
$N$ will denote a sequence of positive constants depending on $n$, $\gamma$, $\kappa$, $\nu$, and $K_0$
that may vary from line to line.

First, from \eqref{eq:q3def}, \eqref{eq:hesslogG}, and \eqref{eq:logthetadot} we have
\begin{align}
\nonumber
  Q_3(\nabla Z, \nabla Z) &\geq \left(2\nabla_i\nabla_j\phi -\frac{\sigma_a}{\dot{\sigma}_a}\ddot{\widehat{\log \sigma_a}}g_{ij}
    -\frac{N}{r_c^2}g_{ij}\right)\langle \nabla_i Z, \nabla_j Z\rangle\\
\nonumber
			  &\geq \left(\frac{1}{(\tau + a)(1-(1/3)(\tau+a))} - \frac{1}{(\tau+a)} -\frac{1}{24} 
			  - \frac{N}{r_c^2}\right)|\nabla Z|^2\\
\label{eq:q3est}
			  &\geq \left(\frac{7}{24}- \frac{N}{r_c^2}\right)|\nabla Z|^2
\end{align}
on $\left(E_{\rad}\cap E_{\gamma\rho}\right)\times[0, \tau_0]$.
Similarly, from \eqref{eq:q4def}, on the same set we have
\begin{align*}
 Q_4(Z, Z) &\geq \frac{1}{2}\left(\pd{\tilde{F}_2}{\tau} + \Delta \tilde{F}_2 +\frac{\dot{\theta}}{\theta}\tilde{F}_2\right)|Z|^2
-\frac{N|F_2|}{r_c^2}|Z|^2,
\end{align*}
while from \eqref{eq:tildeFest},
\eqref{eq:bheatF}, and \eqref{eq:logthetadot}, we have
\begin{align*}
 \frac{1}{2}\left(\pd{\tilde{F}_2}{\tau} + \Delta \tilde{F}_2 +\frac{\dot{\theta}}{\theta}\tilde{F}_2\right)
&\geq \frac{(n-1)\rho}{4h(\tau+a)}\left(\frac{1}{\tau+a} - \frac{1}{(\tau + a)(1-(1/3)(\tau+a))}\right)\\
  &\phantom{\geq}
      - \frac{N}{(\tau+a)} - \frac{1}{(\tau+a)(1-(1/3)(\tau+a))}\frac{1}{16h}\\
  &\geq -N\sigma_a^{-1}.
\end{align*}
Since $|\tilde{F}_2|\leq C(n, \gamma)\sigma_a^{-1}$ on $\left(E_{\rad_0}\cap E_{\gamma\rho}\right)\times [0, \tau_0]$ by \eqref{eq:tildeFest},
we have
\[
 |F_2| = |\tilde{F}_2 - \alpha\dot{\widehat{\log \sigma_a}}|\leq (C(n,\gamma) + \alpha)\sigma_a^{-1}
\]
and therefore, from the previous two inequalities and \eqref{eq:sigmabounds}, that
\begin{equation}\label{eq:q4est}
  Q_4(Z, Z) \geq -\sigma_a^{-1}\left(N +  \frac{\alpha}{10000}\right)|Z|^2
\end{equation}
on $E_{\rad}\cap E_{\gamma\rho}\times[0, \tau_0]$ for $\rad$ sufficiently large. 

Now observe that we may write
\[
  \nabla \phi_2 = -\frac{(h-\rho)\nabla h}{2(\tau+a)} =  -\frac{\tau(h-\rho)}{h(\tau+a)}\nabla f \dfn \Upsilon_2\nabla f,
\]
where $\Upsilon_2^2 \leq C(n, \gamma)$ on $\Ec_{\gamma\rho}^{\tau_0}$. By \eqref{eq:eequation2} we then have
\begin{equation*}
  |E(Z, \nabla Z)| \leq \frac{N}{r_c^2}(|\nabla Z|^2 + (1+ \Upsilon_2^2)|Z|^2) \leq \frac{N}{r_c^2}(|\nabla Z|^2 + |Z|^2), 
\end{equation*}
which, with \eqref{eq:q3est} and \eqref{eq:q4est}, implies \eqref{eq:q3q4ineq} after increasing $\rad$ still further.
\end{proof}

We are now ready to prove our second Carleman estimate.
\begin{proposition}\label{prop:pdecarleman2}
For any $a\in (0, 1)$, $\gamma > 0$, and $\rho \geq \rad_0$, there exists
$C > 0$ depending only on $n$, and $\alpha_0 > 0$, $\rad_7 > \rad_0$ depending
only on $n$, $\gamma$, $\kappa$, $\nu$, and $K_0$ such that for any smooth section
$Z$ of $\mathcal{Z}\times [0, \tau_0]$ which is compactly supported in $\left(E_{\rad_7}\cap E_{\gamma\rho}\right)\times[0, \tau_0)$,
and satisfies $Z(\cdot, 0) \equiv 0$, we have
\begin{align}
\begin{split}
\label{eq:pdecarleman2}
 &\sqrt{\alpha}\|\sigma_a^{-\alpha-1/2}Z\hat{G}_{2}^{1/2}\|_{L^2(\Ec_{\rad_0}^{\tau_0})}
  +\|\sigma_a^{-\alpha}\nabla Z\hat{G}_{2}^{1/2}\|_{L^2(\Ec_{\rad_0}^{\tau_0})}\\
  &\qquad\qquad\leq C\|\sigma_a^{-\alpha}(\partial_{\tau}Z+\Delta Z)\hat{G}_{2}^{1/2}\|_{L^2(\Ec_{\rad_0}^{\tau_0})}
\end{split}
\end{align}
for all $\alpha \geq \alpha_0$, where 
\[
    \hat{G}_2(x, \tau) \dfn \hat{G}_{2; a, \rho}(x, \tau) \dfn \exp{\left(-\frac{(r_c(x) - \rho)^2}{4(\tau+a)}\right)}.
\]
\end{proposition}
\begin{remark}
An essentially identical inequality holds with $G_2 = G_{2; a, \rho}$ in place of $\hat{G}_{2}$ -- in fact, we will prove 
it first for $G_{2}$ and appeal to \eqref{eq:gghatcomp} to obtain \eqref{eq:pdecarleman2}. 
We find the inequality easier to apply with the weight $\hat{G}_{2}$, but easier to prove with $G_{2}$.
\end{remark}

\begin{proof} We begin by choosing $\rad_7$ to be greater than the constant $\rad_6$ from Lemma \ref{lem:quadlowerbounds}, 
and we will continue to increase it as necessary as the argument progresses.
To further reduce clutter, we will use the temporary shorthand 
$\mathcal{E} = \left(E_{\rad_7}\cap E_{\gamma\rho}\right)\times[0, \tau_0]$, and use $C$ and $N$ to denote sequences
of positive constants depending, respectively, only on $n$, and on $n$, $\gamma$, $\kappa$, $\nu$, and $K_0$.

First, combining \eqref{eq:pdeintident2} with \eqref{eq:q3q4ineq} yields the estimate
\begin{align}
\begin{split}
\label{eq:precarleman1}
 &\frac{1}{12e}\|\sigma_a^{-\alpha}\nabla Z G_2^{1/2}\|_{L^2(\mathcal{E})}^2 \leq 
  \left(N + \frac{\alpha}{5000}\right)\|\sigma_a^{-\alpha-1/2}ZG_2^{1/2}\|_{L^2(\mathcal{E})}^2\\
  &\qquad\qquad\qquad\phantom{\leq}  + \|\sigma_a^{-\alpha}(\partial_{\tau}Z+\Delta Z)G_2^{1/2}\|_{L^2(\mathcal{E})}^2,
\end{split}
\end{align}
valid for any $\alpha > 0$. (Here we have renamed $\alpha$ to write the factors of $\sigma_a^{1-\alpha}$
as $\sigma_a^{-2\alpha}$, and have used \eqref{eq:sigmabounds} to estimate the extra factors of $\dot{\sigma}_a$
in \eqref{eq:pdeintident2}.)  

Next, observe that, by \eqref{eq:tildeFest} and our choice of $\rad_7$,
we have $\tilde{F}_2\leq 0$ on $\operatorname{supp}{Z}$.  Using Lemma \ref{lem:pdeintineq3}, we can therefore choose
an
$\alpha_0 = \alpha_0(n, \gamma, \kappa, \nu,  K_0)\geq 1$ such that
\begin{align}
 \begin{split}\label{eq:precarleman2}
\frac{\alpha}{10}\|\sigma_a^{-\alpha-1/2}ZG^{1/2}_{2}\|_{L^2(\mathcal{E})}^2
&\leq 2\|\sigma_a^{-\alpha}\nabla Z G_{2}^{1/2}\|_{L^2(\mathcal{E})}^2\\
 &\phantom{\leq}
  + \frac{20}{\alpha}\|\sigma_a^{-\alpha}(\partial_{\tau}Z+\Delta Z)G_{2}^{1/2}\|_{L^2(\mathcal{E})}^2
 \end{split}
\end{align}
for all $\alpha \geq \alpha_0$.  Increasing $\alpha_0$, if necessary, to ensure
that $12e(N + \alpha/5000) \leq \alpha/40$ for all $\alpha \geq \alpha_0$, we may combine \eqref{eq:precarleman1} 
and \eqref{eq:precarleman2} to obtain that
\begin{equation}\label{eq:precarleman3}
\|\sigma_a^{-\alpha}\nabla Z G_{2}^{1/2}\|_{L^2(\mathcal{E})}^2\leq
 C\|\sigma_a^{-\alpha}(\partial_{\tau}Z+\Delta Z)G_{2}^{1/2}\|_{L^2(\mathcal{E})}^2
\end{equation}
for all $\alpha \geq \alpha_0$.  An appropriate further combination of \eqref{eq:precarleman2} and \eqref{eq:precarleman3} implies
\eqref{eq:pdecarleman2} with the substitute weight $G_{2}$. Then, using \eqref{eq:gghatcomp}, we can replace
$G_{2}$ with $\hat{G}_{2}$ at the expense of increasing the constant $C$ by a factor of $4$.
Finally, relabeling $\alpha$ once more to be $\alpha - n/4$, and using \eqref{eq:sigmabounds} to adjust the constant
by another universal factor, we obtain \eqref{eq:pdecarleman2}. 
\end{proof}

\subsection{A Carleman-type inequality for the ODE component}
Now we derive a matching $L^2$- estimate for the ODE portion of our system.  Since we will not perform any spatial integrations-by-parts
and the metrics $g(\tau)$ are uniformly equivalent, it will suffice 
to first prove the estimate relative to the fixed metric $g_c = g(0)$
and measure $d\mu_{g_c} = d\mu_{g(0)}$, and doing so will eliminate some extra terms in our computations. We will also work with
the function $\hat{G}_2 = \hat{G}_{2; a, \rho}$ from the outset.

Thus far the parameter $a$ has only been restricted to lie in $(0, 1)$; we will assume further now that $0<a\leq a_0$ for some $0 < a_0 \leq 1/8$.
We will also assume that $0 < \tau_0 \leq 1/4$ and that 
$Z\in C^{\infty}(\mathcal{Z}\times[0, \tau_0])$
is both compactly supported on $E_{\rad_0}\times[0, \tau_0)$ and vanishes identically on $E_{\rad_0}\times\{0\}$.
For convenience, we extend $Z$ to a piecewise smooth family of smooth sections of $\mathcal{Z}$ by declaring $Z(x, \tau) = 0$ 
for $\tau \ne [0, \tau_0]$. 
The basis for our estimate is the simple identity
\begin{align}
\begin{split}
\label{eq:odedivident2}
 &\sigma_a^{-2\alpha}\pdtau|Z|^2_{g_c}\hat{G}_2
   - \pdtau\left(\sigma_a^{-2\alpha}|Z|_{g_c}^2\hat{G}_2\right)\\
&\qquad\qquad= \sigma_a^{-2\alpha}\left(\frac{2\alpha}{(\tau+a)} - \frac{2\alpha}{3} - \frac{(r_c-\rho)^2}{4(\tau+a)^2}\right)|Z|^2_{g_c}\hat{G}_2,
\end{split}
\end{align}
valid for any $\alpha$ and $\rho$. From it we derive the inequalities
\begin{align}
\begin{split}
\label{eq:odedivineq2}
&\frac{3}{\alpha}\sigma_a^{-2\alpha}\left|\pd{Z}{\tau}\right|^2_{g_c}\hat{G}_2
   - \pdtau\left(\sigma_a^{-2\alpha}|Z|_{g_c}^2\hat{G}_2\right)\\
&\qquad\qquad\geq \sigma_a^{-2\alpha}\left(\frac{\alpha}{(\tau+a)} - \frac{(r_c-\rho)^2}{4(\tau+a)^2}\right)|Z|^2_{g_c}\hat{G}_2
\end{split}
\end{align}
and
\begin{align}
\begin{split}
\label{eq:odedivineq3}
&\frac{3}{2\alpha}\sigma_a^{-2\alpha}\left|\pd{Z}{\tau}\right|^2_{g_c}\hat{G}_2
   + \pdtau\left(\sigma_a^{-2\alpha}|Z|_{g_c}^2\hat{G}_2\right)\\
&\qquad\qquad\geq \sigma_a^{-2\alpha}\left(\frac{(r_c-\rho)^2}{4(\tau+a)^2} - \frac{2\alpha}{\tau+a}\right)|Z|^2_{g_c}\hat{G}_2
\end{split}
\end{align}
on $\Ec_{\rad_0}^{\tau_0}$ using Cauchy-Schwarz.
Consider the sets
\begin{align}
\begin{split} 
\label{eq:omegaprimedef}
  \Omega_{a, \alpha, \rho}^{\prime} &\dfn \left\{\,(x, \tau)\,|\, (r_c(x)-\rho)^2 \leq 2\alpha(\tau+a)\,\right\},\quad\mbox{and}\\
  \Omega_{a, \alpha, \rho}^{\prime\prime} &\dfn \left\{\,(x, \tau)\,|\, (r_c(x)-\rho)^2 \geq 10\alpha(\tau+a)\,\right\}.
\end{split}
\end{align}
For fixed $x\in E_{\rad_0}$, the intervals
\begin{align*}
  J^{\prime}(x) \dfn J_{a, \alpha, \rho}^{\prime}(x) &\dfn \{\,\tau \,|\,(x, \tau) \in \Omega^{\prime}_{a, \alpha, \rho}\,\}\,\quad\mbox{and}\\
  J^{\prime\prime}(x) \dfn J^{\prime\prime}_{a, \alpha, \rho}(x) &\dfn \{\,\tau\,|\,(x, \tau) \in \Omega^{\prime\prime}_{a, \alpha, \rho}\,\}
\end{align*}
are of the form $[b^{\prime}(x), \infty)$ and $(-\infty, b^{\prime\prime}(x)]$, respectively, for 
\[
  b^{\prime}(x) \dfn b^{\prime}_{a, \alpha, \rho}(x) \dfn \frac{(r_c(x) -\rho)^2}{2\alpha} - a, \quad\mbox{and}\quad  
b^{\prime\prime}(x) \dfn b^{\prime\prime}_{a, \alpha, \rho}(x) \dfn \frac{(r_c(x) -\rho)^2}{10\alpha} - a.
\]
Then, upon integration, we obtain from \eqref{eq:odedivineq2} that 
\begin{align*}
\begin{split}
 &\frac{3}{\alpha}\,\int_{J^{\prime}(x)}\sigma_a^{-2\alpha}\left|\pd{Z}{\tau}\right|_{g_c}^2\hat{G}_2\,d\tau
   + \left.\sigma_a^{-2\alpha}|Z|_{g_c}^2\hat{G}_2\right|_{(x, b^{\prime}(x))}\\ 
&\qquad\qquad\geq \int_{J^{\prime}(x)}\,\sigma_a^{-2\alpha}
\left(\frac{\alpha}{\tau+a} - \frac{(r_c-\rho)^2}{4(\tau+a)^2}\right)|Z|^2_{g_c}\hat{G}_2\,d\tau\\
&\qquad \qquad\geq \frac{\alpha}{6e}\int_{J^{\prime}(x)}\,\sigma_a^{-2\alpha-1}|Z|^2_{g_c}\hat{G}_2\,d\tau,
\end{split}
\end{align*}
and, similarly, from \eqref{eq:odedivineq3}, that
\begin{align*}
  &\frac{3}{2\alpha}\int_{(J^{\prime})^c(x)}\,\sigma_a^{-2\alpha}\left|\pd{Z}{\tau}\right|_{g_c}^2\hat{G}_2\,d\tau
   + \left.\sigma_a^{-2\alpha}|Z|_{g_c}^2\hat{G}_2\right|_{(x, b^{\prime}(x))}\\ 
&\geq \int_{(J^{\prime})^c(x)}\,\sigma_a^{-2\alpha}
\left(\frac{(r_c-\rho)^2}{4(\tau+a)^2}-\frac{2\alpha}{\tau+a}\right)|Z|^2_{g_c}\hat{G}_2\,d\tau\\
&\geq \frac{\alpha}{6e}\int_{J^{\prime\prime}(x)}\,\sigma_a^{-2\alpha-1}|Z|^2_{g_c}\hat{G}_2\,d\tau
  + \int_{b^{\prime\prime}(x)}^{b^{\prime}(x)}\,\sigma_a^{-2\alpha}
\left(\frac{(r_c-\rho)^2}{4(\tau+a)^2}-\frac{2\alpha}{\tau+a}\right)|Z|^2_{g_c}\hat{G}_2\,d\tau\\
&\geq \frac{\alpha}{6e}\int_{J^{\prime\prime}(x)}\,\sigma_a^{-2\alpha-1}|Z|^2_{g_c}\hat{G}_2\,d\tau
  - \frac{3\alpha}{2}\int_{b^{\prime\prime}(x)}^{b^{\prime}(x)}\,\sigma_a^{-2\alpha-1}|Z|^2_{g_c}\hat{G}_2\,d\tau.
\end{align*}
When combined, the above inequalities yield
\begin{align}
\begin{split}\label{eq:odedivineq4}
&\int_{0}^{\tau_0}\,\sigma_a^{-2\alpha -1}|Z|_{g_c}^2\hat{G}_2\,d\tau
  \leq \frac{C}{\alpha^2}\int_{0}^{\tau_0}\,\sigma_a^{-2\alpha }\left|\pd{Z}{\tau}\right|_{g_c}^2\hat{G}_2\,d\tau\\
&\quad\phantom{\leq} + C\int_{b^{\prime\prime}(x)}^{b^{\prime}(x)}\,\sigma_a^{-2\alpha -1}|Z|_{g_c}^2\hat{G}_2\,d\tau 
+ \frac{C}{\alpha}\left.\sigma_a^{-2\alpha}|Z|_{g_c}^2\hat{G}_2\right|_{(x, b^{\prime}(x))},
\end{split}
\end{align}
for some $C = C(n)$, for all $a\in (0, a_0)$, $\alpha > 0 $, $\rho > 0$, and $x\in E_{\rad_0}$.

We pause to estimate the last term in \eqref{eq:odedivineq4}.  Note that, if $b^{\prime}(x)\in [0, \tau_0]$, then 
$x\in A(r^{\prime}, r^{\prime\prime}) = \{\,r^{\prime} \leq |r_c(x)-\rho| \leq r^{\prime\prime}\,\}$
with 
\begin{equation}\label{eq:rprimedef}
   r^{\prime} \dfn r^{\prime}_{a, \alpha, \rho} \dfn \sqrt{2a\alpha},  \quad\mbox{and}\quad 
r^{\prime\prime} \dfn r^{\prime\prime}_{a, \alpha, \rho} \dfn \sqrt{2\alpha(\tau_0+ a)}.
\end{equation}
Then, using the definition of $b^{\prime}(x)$, we see that 
 \begin{align*}
  \left.\sigma_a^{-2\alpha}|Z|_{g_c}^2\hat{G}_2\right|_{(x, b^{\prime}(x))} 
 &=\frac{e^{\frac{2\alpha}{3}(b^{\prime}(x) + a)}}{(b^{\prime}(x) + a)^{2\alpha}}|Z|^2_{g_c}(x, b^{\prime}(x))e^{-\frac{\alpha}{2}}\\
 &\leq \left(\frac{1}{ae^{\frac18}}\right)^{2\alpha}|Z|^2_{g_c}(x, b^{\prime}(x)),
 \end{align*}
since, if $x\in \operatorname{supp}(Z)$, then $b^{\prime}(x) + a \leq \tau_0 + a_0 < 3/8$ by assumption. Also, 
\[
\operatorname{vol}_{g_c}(A(r_{a, \alpha, \rho}^{\prime}, r^{\prime\prime}_{a, \alpha, \rho})) \leq CA_0(\rho + \sqrt{\alpha})^{n}.
\]
for some $C = C(n)$, by \eqref{eq:rprimedef}.

Continuing on now from \eqref{eq:odedivineq4},  upon integration over $E_{\rad_0}$ and an application of Fubini's theorem, we obtain that
\begin{align}
\begin{split}\label{eq:odedivineq5}
&\iint_{\Ec_{\rad_0}^{\tau_0}}\sigma_a^{-2\alpha -1}|Z|_{g_c}^2\hat{G}_2\, d\mu_{g_{c}}\,d\tau
\leq \frac{C}{\alpha^2}\iint_{\Ec_{\rad_0}^{\tau_0}}\,\sigma_a^{-2\alpha }\left|\pd{Z}{\tau}\right|_{g_c}^2\hat{G}_2\,d\mu_{g_c}\,d\tau\\
&\qquad\qquad\phantom{\leq} + C\iint\limits_{\Omega_{a, \alpha, \rho} \cap \Ec_{\rad_0}^{\tau_0}}\,\sigma_a^{-2\alpha -1}|Z|_{g_c}^2\hat{G}_2\,d\mu_{g_c}\,d\tau\\
&\qquad\qquad\phantom{\leq}
+ \frac{C}{\alpha}A_0(\rho + \sqrt{\alpha})^{n}\left(\frac{1}{ae^{\frac18}}\right)^{2\alpha}\|Z\|^2_{\infty, g_c},
\end{split}
\end{align}
where
\begin{equation}\label{eq:omegadef}
  \Omega_{a, \alpha, \rho} \dfn \left(\Omega_{a, \alpha, \rho}^{\prime}\cup \Omega_{a, \alpha, \rho}^{\prime\prime}\right)^c
	= \{\,(x, \tau)\,|\, 2\alpha\leq (r_c(x)-\rho)^2(\tau+a)^{-1}\leq 10\alpha\,\}.
\end{equation}
By an argument similar to that in the preceding paragraph, on $\Omega_{a, \alpha, \rho}\cap\Ec_{\rad_0}^{\tau_0}$, the integrand
in the penultimate integral in \eqref{eq:odedivineq5} can be bounded above by $Ca^{-1}(ae^{1/8})^{-2\alpha}\|Z\|_{\infty, g_c}$, 
and the (space-time) measure of the set 
$\Omega_{a, \alpha, \rho}$ is again bounded above by $CA_0(\rho + \sqrt{\alpha})^{n}$. 
Thus, combining \eqref{eq:odedivineq5} with the equivalence of the norms $|\cdot|_{g_c}$ and $|\cdot| = |\cdot|_{g(\tau)}$ and
 the measures 
$d\mu_{g_c}$ and $d\mu = d\mu_{g(\tau)}$,
we arrive at our desired Carleman-type estimate.
\begin{proposition}\label{prop:odecarleman2}  Suppose $a_0 \in (0, 1/8)$, $\tau_0 \in (0, 1/4)$, and $\rad_0 \geq 1$.
Then there exist constants 
$N$ and $\rad_8 \geq \rad_0$, depending
only on $n$, $\kappa$, $\nu$, $A_0$, and  $K_0$, such that, for any smooth family $Z = Z(\tau)$ 
of sections of $\mathcal{Z}$ with compact support in $E_{\rad_8}\times[0, \tau_0)$
and which satisfies $Z(\cdot, 0) \equiv 0$, we have
\begin{align}
\begin{split}\label{eq:odecarleman2}
\|\sigma_{a}^{-\alpha-\frac12}Z\hat{G}_{2}^{\frac12}\|_{L^2(\Ec_{\rad_0}^{\tau_0})}
&\leq N\alpha^{-1}\|\sigma_{a}^{-\alpha}\partial_{\tau}Z\hat{G}_{2}^{\frac12}\|_{L^2(\Ec_{\rad_0}^{\tau_0})}\\
&\phantom{\leq}+ Na^{-\frac12}(\rho + \sqrt{\alpha})^{\frac{n}{2}}\left(\frac{1}{ae^{\frac18}}\right)^{\alpha}\|Z\|_{\infty, g_c}
\end{split}
\end{align}
for all $\alpha > 0$.
\end{proposition}


\section{Proof of backwards uniqueness}
\label{sec:backwardsuniqueness}

We now have the components we need to assemble our proof of Theorem \ref{thm:bu}. 
Below, we will continue to use one of the metrics, $g = g(\tau)$, from the statement
of that theorem as a reference metric in our estimates, and so will continue to assume that $g$ and its potential 
$f$ satisfy equations \eqref{eq:brf} -- \eqref{eq:fid1} of Proposition \ref{prop:reduction}.
By the arguments in Sections \ref{sec:reduction} and \ref{sec:pdeode},
it is enough to show that the sections $\Xb = S\oplus T$ and $\Yb = U\oplus V\oplus W$  defined in Section \ref{ssec:xydef}
 vanish identically on $\Ec_{\rad}^{\tau^{\prime}}$ for $\rad$ sufficiently large and $\tau^{\prime} \in (0, 1)$ sufficiently small. 

To begin, we observe that, from Proposition \ref{prop:pdeode}, there exists a constant $N = N(n, K_0)$ and, for any $\epsilon > 0$, another
constant $\rad_9 = \rad_9(\epsilon, n, K_0, \rad_0) \geq \rad_0$
such that 
\begin{align}\label{eq:xybound}
|\Xb| + |\nabla \Xb| + |\Yb|  \leq N
\end{align}
and
\begin{align}\label{eq:modpdeode}
\begin{split}
  \left|\pd{\Xb}{\tau} + \Delta \Xb\right| \leq \epsilon\left(|\Xb| +  |\Yb|\right), \quad \left|\pd{\Yb}{\tau}\right|\leq N\left(|\Xb| + |\nabla \Xb|\right) + \epsilon|\Yb| 
\end{split}
\end{align}
on $\Ec_{\rad_9}^{1}$.

Next, we describe the basic spatial cutoff function we will use in our argument.
\begin{lemma}\label{lem:cutoff}
 Given $\rho > 12\rad_0$ and $\xi > 4\rho$, there exists 
a smooth function $\psi_{\rho, \xi} \in C^{\infty}(E_{\rad_0}, [0, 1])$ satisfying
 $\psi_{\rho, \xi} \equiv 1$ on $E_{\frac{\rho}{3}}\setminus E_{2\xi}$ and $\psi_{\rho, \xi}\equiv 0$ on $(E_{\rad_0}\setminus E_{\rho/6})\cup E_{3\xi}$
 whose derivatives satisfy
 \begin{equation}\label{eq:cutoffderbounds}
    |\nabla \psi_{\rho, \xi}| + |\Delta \psi_{\rho, \xi}| \leq N\rho^{-1}
 \end{equation}
for some $N = N(n, K_0)$.  Here $|\cdot| = |\cdot|_{g(\tau)}$ and $\nabla = \nabla_{g(\tau)}$.
\end{lemma}
\begin{proof} 
It is a routine matter to construct such a function in the form $\psi_{\rho, \xi}(x) = \eta_1(r_c(x)/\rho) - \eta_2(r_c(x)/\xi)$ 
for some $\eta_i \in C^{\infty}_c(\RR, [0, 1])$.
The only potentially nonstandard detail to verify is the \emph{two-sided} bound on the Laplacian, which may be derived
from the identity $\nabla\nabla_{g_c} r_c = r_c g_{\Sigma}$, the uniform equivalence of the metrics $g$ and $g_c$,
 and the pointwise estimate on $|\Gamma - \Gamma_{g_c}|$ one obtains from the bounds on $\nabla\Rc$.
\end{proof}

From this point onwards, the proof consists of two general steps.
First, we apply the PDE and ODE Carleman inequalities of Section \ref{sec:carleman2} to (suitably cut-off versions of) $\Xb$ and $\Yb$,
and use them to verify that $\Xb$ and $\Yb$ have quadratic exponential decay in space if they vanish at $\tau = 0$.  This ensures the validity
of our second step, in which we apply the Carleman estimates in Section \ref{sec:carleman1} to deduce that 
$\Xb$ and $\Yb$ vanish identically.

\subsection{Exponential Decay} We now proceed with the first of these steps, verifying the following ancillary claim.
\begin{claim}\label{cl:expdecay} There exist constants $s_0 = s_0(n)$, $C = C(n)$, $N = N(n, K_0)$, and $\rad_{10} = \rad_{10}(n, K_0)$
with $s_0 \in (0, 1]$ and $\rad_{10}\geq \rad_0$,
such that, for all $\rad \geq \rad_{10}$,
\begin{equation}\label{eq:xydecay}
  \||\Xb| + |\nabla \Xb| + |\Yb|\|_{L^2(A((1-\sqrt{s})\rho, (1+\sqrt{s})\rho)\times[0, s])}
\leq N e^{-\frac{\rho^2}{Cs}},
\end{equation}
for any $s\in (0, s_0]$ and $\rho > 12\rad$.  Here $A(r_1, r_2)$ denotes the annular region $E_{r_1}\setminus E_{r_2}$. 
\end{claim}

\begin{proof}
We follow the broad outline of the proof of Lemma 4 of \cite{EscauriazaSereginSverak}, making adjustments where necessary
to handle the additional error term contributed by the inequality \eqref{eq:odecarleman2},
the lack of consistent scaling among the components of $\Xb$ and $\Yb$, and the somewhat different
form of our Carleman estimates.
 
Let $\rad_7$ and $\rad_8$ be the constants guaranteed by Propositions \ref{prop:pdecarleman2} and \ref{prop:odecarleman2}
with the choice $\gamma = 1/12$.  We take $s_0 = 1/4$ and $\rad_{10} = \max\{\rad_7, \rad_8\}$ initially, and adjust them as the argument progresses, always assuming
$\rad \geq \rad_{10}$. 
 We then let $\rho$ be a positive parameter satisfying $\rho \geq 12\rad$ and choose a further large number  $\xi \geq 4\rho$. 
 Below, $C$  will denote a series of constants depending only on the parameter $n$ and $N$ 
a series depending only on $n$, $A_0$ and $K_0$.

Take $\psi_{\rho, \xi}$ to be the cutoff function guaranteed by Lemma \ref{lem:cutoff}, and choose a temporal cutoff function 
$\varphi\in C^{\infty}(\RR, [0, 1])$ with 
 $\varphi \equiv 1$ for $\tau\leq 1/6$ and $\varphi \equiv 0$ for $\tau\geq 1/5$.
Then $\Xb_{\rho, \xi}\dfn \varphi\psi_{\rho, \xi}\Xb$ and $\Yb_{\rho, \xi}\dfn \varphi\psi_{\rho, \xi}\Yb$ are compactly supported in
$A(\rho/6, 3\xi)\times [0, 1/4)$.  Applying Propositions \ref{prop:pdecarleman2} and \ref{prop:odecarleman2}, respectively, to
the components of $\Xb_{\rho, \xi}$ and $\Yb_{\rho, \xi}$, summing the result, and using \eqref{eq:xybound}, we obtain constants $k_0$ and $N$
such that
\begin{align*}
\begin{split}
 &k^{\frac12}\|\sigma_a^{-k-\frac{1}{2}}\Xb_{\rho, \xi}\hat{G}_2^{\frac{1}{2}}|\|_{L^2(A(\frac{\rho}{6},3\xi)\times [0, \frac15])}
  +\|\sigma_a^{-k}\nabla \Xb_{\rho, \xi}\hat{G}^{\frac12}_2\|_{L^2(A(\frac{\rho}{6},3\xi)\times [0, \frac15])}\\
  &\qquad\phantom{\leq}+\|\sigma_{a}^{-k-\frac{1}{2}}\Yb_{\rho, \xi}\hat{G}_2^{\frac12}\|_{L^2(A(\frac{\rho}{6},3\xi)\times [0, \frac15])}\\
  &\leq N\|\sigma_a^{-k}(\partial_{\tau}+\Delta) \Xb_{\rho, \xi}\hat{G}_2^{\frac12}\|_{L^2(A(\frac{\rho}{6},3\xi)\times [0, \frac15])} \\
  &\qquad\phantom{\leq}
      + Nk^{-1}\|\sigma_{a}^{-k}\partial_{\tau}\Yb_{\rho, \xi}\hat{G}_2^{\frac12}\|_{L^2(A(\frac{\rho}{6},3\xi)\times [0, \frac15])}\\
  &\qquad\phantom{\leq}
+ N(\rho + k^{\frac12})^{\frac{n}{2}}a^{-\frac12}(ae^{\frac18})^{-k}
\end{split}
\end{align*}
for any $a\in (0, 1/8)$ and any $k \geq k_0$. Here $\hat{G}_2= \hat{G}_{2; a, \rho}$. 

Now, by \eqref{eq:modpdeode}, for all $\epsilon > 0$, there is $\rad_9 = \rad_9(\epsilon)$ such that the pair of inequalities
\begin{align*}
  \left|\pd{\Xb_{\rho, \xi}}{\tau} + \Delta \Xb_{\rho, \xi}\right|&\leq \epsilon\left(|\Xb_{\rho, \xi}| + |\Yb_{\rho, \xi}|\right)
	      +\psi_{\rho, \xi}|\varphi^{\prime}||\Xb|\\
  &\phantom{\leq}
\qquad+ \varphi\left(|\Delta\psi_{\rho, \xi}||\Xb| + 2|\nabla \psi_{\rho, \xi}||\nabla \Xb|\right),\\
 \left|\pd{\Yb_{\rho, \xi}}{\tau}\right| &\leq 
  N\left(|\Xb_{\rho, \xi}| + |\nabla \Xb_{\rho, \xi}|\right) + \epsilon|\Yb_{\rho, \xi}| + N\varphi|\nabla\psi_{\rho, \xi}||\Xb|
      +\psi_{\rho, \xi}|\varphi^{\prime}||\Yb|,
\end{align*}
hold on $\Ec_{\rad_9}^{1}$.  Thus if $\epsilon$ is taken sufficiently small,
and $k_1 \geq k_0$ sufficiently large, we may increase $\rad_{10}$ to ensure  $\rad_{10}\geq \rad_9$, assume $k\geq k_1$,  
and return to the preceding inequality
to absorb the terms proportional to $|\Xb_{\rho, \xi}|$, $|\nabla \Xb_{\rho, \xi}|$,
and  $|\Yb_{\rho, \xi}|$ on the right into the left-hand side. (Here we also use that $\sigma_a \leq 1$.) We obtain the inequality
\begin{align}\label{eq:decayest2}
\begin{split}
 &\|\sigma_a^{-k-\frac{1}{2}}(|\Xb_{\rho, \xi}|+|\Yb_{\rho, \xi}|)\hat{G}_2^{\frac{1}{2}}|\|_{L^2(A(\frac{\rho}{6},3\xi)\times [0, \frac15])}\\
 &\qquad\phantom{\leq}
  +\|\sigma_a^{-k}\nabla \Xb_{\rho, \xi}\hat{G}_2^{\frac12}\|_{L^2(A(\frac{\rho}{6},3\xi)\times [0, \frac15])}\\
  &\qquad\leq N \|\sigma_a^{-k-\frac{1}{2}}(|\Xb|+|\Yb|)\hat{G}_2^{\frac{1}{2}}|\|_{L^2(A(\frac{\rho}{6},3\xi)\times [\frac16, \frac15])}\\
  &\qquad\phantom{\leq} + N\|\sigma_a^{-k}(|\Xb| +|\nabla \Xb| + |\Yb|)\hat{G}_2^{\frac12}\|_{L^2(A(\frac{\rho}{6},\frac{\rho}{3})\times [0, \frac15])}\\
  &\qquad \phantom{\leq} + N\|\sigma_a^{-k}(|\Xb| +|\nabla \Xb| + |\Yb|)\hat{G}_2^{\frac12}\|_{L^2(A(2\xi, 3\xi)\times [0, \frac15])} \\
  &\qquad\phantom{\leq}+ N(\rho + k^{\frac12})^{\frac{n}{2}}a^{-\frac12}(ae^{\frac18})^{-k},
\end{split}
\end{align}
valid for all $k\geq k_1$ and all $0 < a < 1/8$.

Consider the penultimate term in \eqref{eq:decayest2}.  Since $\xi\geq 4\rho$, we have 
$r_c -\rho \geq 7\xi/4$ on $A(2\xi, 3\xi)\times [0, 1/5]$ and thus, for some universal $\beta$, we have
\begin{align*}
& N\|\sigma_a^{-k}(|\Xb| + |\nabla \Xb|+|\Yb|)\hat{G}_2^{\frac12}\|_{L^2(A(2\xi, 3\xi)\times [0, \frac15])}
\leq N\left(e/a\right)^{k}e^{-\beta\xi^2}\xi^{\frac{n}{2}}
\end{align*}
where we have used the uniform equivalence of the metrics $g(\tau)$ and $g_c$ to estimate the volume. It follows that 
this term tends to $0$ as $\xi\to\infty$. In fact, from \eqref{eq:xybound} and the quadratic exponential decay of $\hat{G}_2$ we see also that
the integrals in the other terms in \eqref{eq:decayest2} will be finite as $\xi\to\infty$.  
Therefore, upon sending $\xi\to\infty$ in \eqref{eq:decayest2}, using the monotone convergence theorem,
and shrinking the domain of integration on the left side,
we obtain
\begin{align}\label{eq:decayest3}
\begin{split}
 &\|(\tau+a)^{-k}(|\Xb|+|\nabla \Xb| + |\Yb|)\hat{G}_2^{\frac12}\|_{L^2(E_{\frac{\rho}{3}}\times [0, \frac{1}{6}])}\\
  &\leq C^k N\|(\tau+a)^{-k}(|\Xb|+|\Yb|)\hat{G}_2^{\frac12}\|_{L^2(E_{\frac{\rho}{6}}\times [\frac16, \frac15])}\\
&\qquad+ C^kN\|(\tau+a)^{-k}(|\Xb| + |\nabla \Xb|+|\Yb|)\hat{G}_2^{\frac12}\|_{L^2(A(\frac{\rho}{6}, \frac{\rho}{3})\times [0, \frac15])}\\
  &\qquad+ N(\rho + k^{\frac12})^{\frac{n}{2}}a^{-\frac12}(ae^{\frac18})^{-k}
\end{split}
\end{align}
for some universal constant $C$. The inequality is valid 
for all $k\geq k_1$ and $a\in (0, 1/8)$.

Now, by \eqref{eq:xybound}, the first
term on the right side of \eqref{eq:decayest3} can be estimated as
\begin{align*}
 \|(\tau+a)^{-k}(|\Xb|+|\Yb|)\hat{G}_2^{\frac12}\|_{L^2(E_{\frac{\rho}{6}}\times [\frac16, \frac15])}
    &\leq 6^kN\|e^{-\frac{(r_c-\rho)^2}{3}}\|_{L^2(E_{\frac{\rho}{6}}\times[\frac16, \frac15])}\leq C^kN\rho^{\frac{n}{2}}
\end{align*}
for all $k\geq k_1$ and $a\in (0, 1/8)$.
On the domain of the second term on the right side of \eqref{eq:decayest3},
we have $e^{-(r_c-\rho)^2/(8(\tau+a))} \leq e^{-\rho^2/(18(\tau+a))}$, and, by Stirling's formula,
\[
 \max_{s>0} s^{-k}e^{-\rho^2/(18s)}  = \rho^{-2k}(18k)^ke^{-k} \leq \rho^{-2k}C^k k!,
\]
so, invoking \eqref{eq:xybound}, we can estimate this term as
\[
  \|(\tau+a)^{-k}(|\Xb| + |\nabla \Xb|+ |\Yb|)\hat{G}_2^{\frac12}\|_{L^2(A(\frac{\rho}{6}, \frac{\rho}{3})\times [0, \frac15])} 
    \leq N\rho^{-2k}C^kk!\rho^{\frac{n}{2}}
\]
for a universal constant $C$ and all $a\in (0, 1/8)$ and $k \geq k_1$.

Returning to \eqref{eq:decayest3} with these two estimates in hand, we obtain
\begin{align}\label{eq:decayest4}
\begin{split}
 &\|(\tau+a)^{-(k_1+l)}(|\Xb|+|\nabla \Xb| + |\Yb|)\hat{G}_2^{\frac12}\|_{L^2(E_{\frac{\rho}{3}}\times [0, \frac{1}{6}])}\\ 
  &\leq C^{k_1+l}N\rho^{\frac{n}{2}}(1+ \rho^{-2(k_1+l)}(k_1+l)!)
    + N(\rho + (k_1+l)^{\frac12})^{\frac{n}{2}}a^{-\frac12}(ae^{\frac18})^{-(k_1+l)}
\end{split}
\end{align}
for all $l\geq 0$ and all $a\in (0, 1/8)$. Noting that $(l+k_1)! \leq C^{k_1+l}(l!)(k_1!)$ 
and $l^{n/4}e^{-l/16} \leq C$ for some universal $C$,
we can
multiply both sides of \eqref{eq:decayest4} by $\rho^{2l}/((2C)^ll!)$ and sum over all $l\geq 0$
to obtain
\begin{align}\label{eq:decayest5}
 \begin{split}
  &\|(|\Xb|+|\nabla \Xb| + |\Yb|)e^{\frac{\rho^2}{C(\tau+a)}
      - \frac{(r_c-\rho)^2}{8(\tau+a)}}\|_{L^2(E_{\frac{\rho}{3}}\times [0, \frac16])}\\
&\qquad\leq N\rho^{\frac n 2}\left(1+ e^{\frac{\rho^2}{2}}  + a^{-(k_1+\frac12)}e^{\frac{\rho^2}{Cae^{1/16}}}\right).
 \end{split}
\end{align}
for some possibly increased universal $C$.

The $e^{1/16}$ factor in the denominator of the exponent in the last term on the right is crucial
here, as it enables us to achieve a slightly smaller relative value in the denominator of
the exponent of the corresponding factor on the left by suitably restricting $\tau$. Specifically, if we write $e^{1/16} = 1 +2\delta$, then
\begin{align*}
 \begin{split}
  &\|(|\Xb|+|\nabla \Xb| + |\Yb|)e^{-\frac{(r_c-\rho)^2}{8(\tau+a)}}\|_{L^2(E_{\frac{\rho}{3}}\times [0, \delta a])}\\
    &\qquad \leq N\rho^{\frac n 2}\left(\left(1+ e^{\frac{\rho^2}{2}}\right)e^{-\frac{\rho^2}{Ca(1+\delta)}}
    + a^{-(k_1+\frac12)}e^{\frac{-\delta\rho^2}{Ca(1+2\delta)(1+\delta)}}\right).
 \end{split}
\end{align*}
Since $a^{-(k_1+1/2)}e^{-\delta\rho^2/(2Ca(1+2\delta)(1+\delta))}$ is bounded above by a constant
depending only on $k_1$, provided we increase $N$ by an appropriate factor 
(recall that $k_1$ depends on the same parameters as $N$), it follows that there is $C = C(n)$ sufficiently
large such that
\[
 \|(|\Xb|+|\nabla \Xb| + |\Yb|)e^{-\frac{(r_c-\rho)^2}{8(\tau+a)}}\|_{L^2(E_{\frac{\rho}{3}}\times [0, \delta a])}
    \leq Ne^{\frac{-\rho^2}{Ca}}.
\]
for all $0 < a\leq 1/C$. On the other hand, $e^{-|r_c -\rho|^2/(8(\tau+a))} \geq N^{-1}$
on the set $A((1-\sqrt{\delta a})\rho, (1+\sqrt{\delta a})\rho)\times [0, \delta a]$, so
we obtain that
\begin{align}
\begin{split}
    \||\Xb|+|\nabla \Xb| +|\Yb|\|_{L^2(A((1-\sqrt{s})\rho,(1+\sqrt{s})\rho)\times[0, s])}
\leq Ne^{-\frac{\rho^2}{Cs}}
\end{split}  
\end{align}
for $s\in [0, 1/C]$ and arbitrary $\rho \geq 12\rad_6$. 
\end{proof}

\subsection{The vanishing of $\Xb$ and $\Yb$}

The proof of Theorem \ref{thm:bu} is now reduced to that of the following claim.
\begin{claim} There exist $\tau^{\prime}\in (0, 1)$ and $\rad_{11} = \rad_{11}(n, K_0)$ such that
$\Xb \equiv 0$ and $\Yb\equiv 0$ on $\Ec_{\rad_{11}}^{\tau^{\prime}}$. 
\end{claim}
\begin{proof}  Below, we will continue to use $N$ to denote a series of constants which depend at most
on the parameters $n$, $A_0$, $K_0$, and $\rad_0$.  
We first show that the (space-time) $L^2$-norms of $\Xb$ and $\nabla \Xb$, weighted by 
$e^{8r_c^2}$, are finite on $\Ec_{\rad_0}^{s}$ for $s$ sufficiently small.
According to \eqref{eq:xydecay}, there is  $s_1 = s_1(n)$
such that
\[ 
\||\Xb| + |\nabla\Xb|\|_{L^2(A_{\rad, \sqrt{s}} \times[0, s])} 
  \leq e^{-16\rad^2},
\]
provided $\rad \geq 12\rad_{10}$ and $s\in [0, s_1]$. Here $A_{\rad, \epsilon} = E_{\rad - \epsilon}\setminus E_{\rad +\epsilon}$.
Thus, for any $\rad^{\prime\prime}\geq \rad^{\prime}\dfn 12\rad_{10}$, we have that
\begin{align*}
&\|(|\Xb| + |\nabla\Xb|) e^{4r_c^2}\|_{L^2(E_{\rad_0}\setminus E_{\rad^{\prime\prime}}\times[0, s_1])}\\
&\quad\leq \|(|\Xb|+|\nabla \Xb|) e^{4r_c^2}\|_{L^2(E_{\rad^{\prime}}\setminus E_{\rad^{\prime\prime}}\times[0, s_1])}
+ \|(|\Xb|+|\nabla\Xb|) e^{4r_c^2}\|_{L^2(E_{\rad_0}\setminus E_{\rad^{\prime}}\times[0, s_1])}\\
&\quad \leq \sum_{i=0}^{k^{\prime}}\|(|\Xb|+|\nabla\Xb|) e^{4r_c^2}\|_{L^2(A_{\rad^{\prime}+(2i+1)\sqrt{s_1}, \sqrt{s_1}}\times[0, s_1])}
    + N \|e^{4r_c^2}\|_{L^2(E_{\rad_0}\setminus E_{\rad^{\prime}}\times[0, s_1])}\\
&\quad \leq \sum_{i=0}^{k^{\prime}} N(\rad^{\prime} +2(i+1)\sqrt{s_1})^{\frac{n}{2}}e^{-12((\rad^{\prime})^2+i^2s_1)}
    + N
\end{align*}
where $k^{\prime} = \lceil(\rad^{\prime\prime}-\rad^{\prime})/(2\sqrt{s_1})\rceil$.
Sending $\rad^{\prime\prime}\to \infty$ it follows that
\begin{equation}\label{eq:xnablaxbounds}
   \|(|\Xb| + |\nabla \Xb|) e^{4r_c^2}\|_{L^2(\Ec_{\rad_0}^{s_1})} \leq N
\end{equation}
for some constant $N = N(n, A_0, K_0)$. 
In particular, by the mean value theorem, there is at least one $\tau^* \in (0, s_1)$ such that
\begin{equation}\label{eq:xnablaxspatialbounds}
  \int_{E_{\rad_0}\times\{\tau^*\}}\left(|\Xb| + |\nabla\Xb|\right)^2e^{8r_c^2}\,d\mu 
= \frac{1}{s_1}\|(|\Xb|+|\nabla \Xb|) e^{4r_c^2}\|_{L^2(\Ec_{\rad_0}^{s_1})}^2 \leq N.
\end{equation}

Now we are ready to apply our first Carleman estimate.  By \eqref{eq:modpdeode},
we can choose $\rad_{12} \geq \rad_0$ to ensure that
\begin{align}\label{eq:modpdeode2}
 \begin{split}
  &\left|\pd{\Xb}{\tau} + \Delta \Xb\right| \leq \frac{1}{100}\left(|\Xb| + |\Yb|\right),\\
  &\left|\pd{\Yb}{\tau}\right| \leq N(|\Xb| + |\nabla\Xb|) + \frac{1}{100}|\Yb|. 
 \end{split}
\end{align}
Next, as in Section \ref{sec:carleman1}, we take
\[
    G_1(x, \tau) \dfn G_{1; \alpha, \tau^*} = e^{\alpha(\tau^*-\tau)h^{2-\delta}(x, \tau) + h^2(x, \tau)} 
\]
for $\alpha \geq 0$ and $(x, \tau)\in \Ec_{\rad_0}^{\tau^*}$.  Observe that, by \eqref{eq:hrbounds}, we have $G_1(x, \tau^*) \leq e^{4r_c^2(x)}$
on $E_{\rad_0}$ and, generally, $G_1(x, \tau) \leq e^{8r_c^2(x)}$ on $\Ec_{\rad(\alpha)}^{\tau^*}$ for $\rad(\alpha)$ sufficiently large.  
Choose $\rad \geq \max\{12\rad_0, \rad_3, \rad_{12}\}$, and, for all $\xi > 4\rad$, let $\psi_{\rad, \xi}: E_{\rad_0}\to [0, 1]$ be a cutoff function, constructed as in Lemma \ref{lem:cutoff},
satisfying $\psi_{\rad, \xi} \equiv 1$ for $ 2\rad \leq r_c(x)\leq \xi$, $\psi_{\rad, \xi} \equiv 0$ for $r_c(x) <\rad$ and $r_c(x) > 2\xi$, and
$|\nabla \psi_{\rad, \xi}| + |\Delta \psi_{\rad, \xi}| \leq L(n, \rad, K_0)$.  Then $\Xb_{\rad, \xi} \dfn \psi_{\rad, \xi}\Xb$
and $\Yb_{\rad, \xi} \dfn \psi_{\rad, \xi}\Yb$ have compact support in $E_{\rad}$ for each $\tau\in[0, \tau^*]$, and so, by Proposition \ref{prop:systemcarlemanineq},
we have
\begin{align*}
\begin{split}
&\alpha^{\frac12}\|\Xb_{\rad, \xi}G_1^{\frac12}\|_{L^2(\Ec_{\rad}^{\tau^*})}
  + \|\nabla \Xb_{\rad, \xi} G_1^{\frac12}\|_{L^2(\Ec_{\rad}^{\tau^*})} 
+\|\Yb_{\rad, \xi} G_1^{\frac12}\|_{L^2(\Ec_{\rad}^{\tau^*})}
\\
  &\qquad\le 2\|(\partial_{\tau} +\Delta) \Xb_{\rad, \xi} G_1^{\frac12}\|_{L^2(\Ec_{\rad}^{\tau^*})}
   + 4\alpha^{-\frac12}\|\partial_{\tau}\Yb_{\rad, \xi}G_1^{\frac12}\|_{L^2(\Ec_\rad^{\tau^*})} \\
  &\qquad\phantom{\le}  + 2\|\nabla \Xb_{\rad, \xi} G_1^{\frac12}\|_{L^2(E_{\rad}\times\{\tau^*\})},
\end{split}
\end{align*}
for all $\alpha \geq 1$ where $G_1 = G_{1; \alpha, \tau^*}$.

On the other hand, by \eqref{eq:modpdeode2},
\begin{align*}
  \left|\left(\pdtau + \Delta\right)\Xb_{\rad, \xi}\right| 
    &\leq \frac{1}{100} \left(|\Xb_{\rad, \xi}| + |\Yb_{\rad, \xi}|\right) + 2|\nabla\psi_{\rad, \xi}||\nabla\Xb|
     + |\Delta \psi_{\rad, \xi}||\Xb|\\
  \left|\pdtau\Yb_{\rad, \xi}\right| &\leq N(|\Xb_{\rad, \xi}| + |\nabla\Xb_{\rad, \xi}|) + \frac{1}{100}|\Yb_{\rad, \xi}|
	  +N|\nabla\psi_{\rad, \xi}| |\Xb|
\end{align*}
on $\Ec_{\rad}^{\tau^*}$, so there exists $\alpha_4 = \alpha_4(n, A_0, K_0)$ such that 
\begin{align}\label{eq:xyineq1}
\begin{split}
&\|(|\Xb| + |\Yb|) G_1^{\frac12}\|_{L^2(A(2\rad, \xi)\times[0, \tau^*])}\\
  &\qquad\le N\|(|\Xb| + |\nabla \Xb|)G_1^{\frac12}\|_{L^2(A(\rad, 2\rad)\times[0, \tau^*])}\\
  &\qquad\phantom{\le} + N\|(|\Xb| + |\nabla \Xb|)G_1^{\frac12}\|_{L^2(A(\xi, 2\xi)\times[0, \tau^*])}\\
   &\qquad\phantom{\le} + N\|\Xb G_1^{\frac12}\|_{L^2(A(\rad, 2\rad)\times\{\tau^*\})}
     + N\|\Xb G_1^{\frac12}\|_{L^2(A(\xi, 2\xi)\times\{\tau^*\})}\\
   &\qquad\phantom{\le} +N\|\nabla \Xb G_1^{\frac12}\|_{L^2(A(\rad,2\xi)\times\{\tau^*\})}\\
\end{split}
\end{align}
for all $\alpha \geq \alpha_4$.  Now,
\[
  \|(|\Xb| + |\nabla \Xb|)G_1^{\frac12}\|_{L^2(A(\xi, 2\xi)\times[0, \tau^*])} 
\leq L^{\prime} \|(|\Xb| + |\nabla \Xb|)e^{4r_c^2}\|_{L^2(A(\xi, 2\xi)\times[0, \tau^*])}
\]
for some constant $L^{\prime} = L^{\prime}(\alpha, \delta)$, and since $G_{1; \alpha, \tau^*}(x, \tau^*) = e^{h^2(x, \tau^*)}$, we have
\begin{align*}
  &\|\Xb G^{\frac12}\|_{L^2(A(\xi, 2\xi)\times\{\tau^*\})} \leq \|\Xb e^{4r_c^2}\|_{L^2(A(\xi, 2\xi)\times\{\tau^*\})},
\quad\mbox{and}\\
&\|\nabla \Xb G_{1}^{\frac12}\|_{L^2(A(\rad,2\xi)\times\{\tau^*\})}
      \leq\|\nabla \Xb e^{4r_c^2}\|_{L^2(A(\rad, 2\xi)\times\{\tau^*\})},
\end{align*}
so, in view of \eqref{eq:xnablaxbounds} and \eqref{eq:xnablaxspatialbounds}, sending $\xi\to\infty$
in \eqref{eq:xyineq1}, we obtain
\begin{align*}
\begin{split}
&\|(|\Xb| + |\Yb|) G_1^{\frac12}\|_{L^2(E_{2\rad}\times[0, \tau^*])}
\le N\|(|\Xb| + |\nabla \Xb|)G_1^{\frac12}\|_{L^2(A(\rad, 2\rad)\times[0, \tau^*])}+ N.
\end{split}
\end{align*}
But then, for any $\theta \geq 1$, we have 
$G_1 \geq \exp{(16\alpha\tau^*(\theta\rad)^{2-\delta})}$ on $\Ec_{64\theta\rad}^{\tau^*/2}$, while we will have 
$G_1 \leq \exp{(16\rad^2(\alpha\tau^* + 1))}$ on $A(\rad, 2\rad)\times[0, \tau^*]$.
So, for all $\alpha \geq \alpha_4$, we have
\begin{align*}
\begin{split}
&\|(|\Xb| + |\Yb|)\|_{L^2(E_{64\theta\rad}\times[0, \frac{\tau^*}{2}])}\\
&\;\;\le Ne^{8\rad^{2}(1 + \alpha\tau^*(1 - \theta^{2-\delta}\rad^{-\delta}))}\|(|\Xb| + |\nabla \Xb|)\|_{L^2(A(\rad, 2\rad)\times[0, \tau^*])}
+ Ne^{-8\alpha\tau^*(\theta\rad)^{2-\delta}}.
\end{split}
\end{align*}
Choosing $\theta$ such that $\theta^{2-\delta} > \rad^{\delta}$, we can send $\alpha \to \infty$ to conclude at last that 
$\Xb$ and $\Yb$ must vanish identically on $\Ec_{64\theta\rad}^{\tau^*/2}$.
\end{proof}

\appendix

\section{Asymptotically conical metrics} \label{app:asymcone}

In this appendix, $(\Sigma, g_{\Sigma})$ and $(\hat{\Sigma}, g_{\hat{\Sigma}})$ will denote closed $(n-1)$-dimensional Riemannian 
manifolds and $g_c$ and $\hat{g}_c$  regular cones on $E_0 \dfn (0, \infty)\times \Sigma$ and $\hat{E}_0 \dfn 
(0, \infty)\times \hat{\Sigma}$, respectively. We will denote the associated dilation maps by 
$\rho_{\lambda}$ and $\hat{\rho}_{\lambda}$ and use $\mathcal{C} \dfn E_0 \cup \{O\}$ and 
$\hat{\mathcal{C}} \dfn \hat{E}_0 \cup \{\hat{O}\}$ to denote the (completed) metric cones
with vertices $O$ and $\hat{O}$ and metrics $d_{\mathcal{C}}$ and $d_{\hat{\mathcal{C}}}$.

\subsection{Some elementary consequences of Definition \ref{def:asymcone}.}

\begin{lemma}\label{lem:asymcone1}  
Let $(M, g)$ be a Riemannian manifold, $V$ an end of $M$, and $\Phi: E_a \to V$ a diffeomorphism
for some $a > 0$. For all $k = 0, 1, 2, \ldots$, define the proposition
\begin{equation}
\label{eq:ack}
 \tag{$AC_k$} 
 \lim_{\lambda\to\infty}
\lambda^{-2}\rho^*_{\lambda}\Phi^*g = g_c \quad\mbox{in}\quad C^k_{\mathrm{loc}}(E_0, g_c).
\end{equation}
Then 
\begin{enumerate}
 \item[(a)]
\eqref{eq:ack} holds if and only if
\[
	    \lim_{b\to\infty}b^l\|\nabla^{(l)}_{g_c}(\Phi^*g - g_c)\|_{C^0(E_b, g_c)} = 0
\]
for each $l = 0, 1, 2, \ldots, k$.
\item[(b)] If ($AC_0$) holds, then the metrics $\Phi^*g$ and $g_c$ are uniformly equivalent on $\overline{E_b}$ for any $b > a$,
and, for all $\epsilon > 0$, there exists $b > a$ such that, for $(r,\sigma)\in E_b$,
\begin{equation}
\label{eq:epsest}
(1-\epsilon)|r-b|\leq\bar{r}_b(r,\sigma)\le (1+\epsilon)|r-b|,
\end{equation}  
where $\bar{r}_b(x)\dfn \dist_{\Phi^\ast g}(x,\partial E_b)$.
\item[(c)] If ($AC_2$) holds, then for any $b > a$, there exists a constant $K = K(b, g_{\Sigma}) > 0$ such that
\begin{equation}
\label{eq:acquaddecay1}
     \sup_{x\in E_b}(\bar{r}_b^2(x) + 1)|\Rm(\Phi^\ast g)|_{\Phi^\ast g}(x) \leq K.
\end{equation}
\end{enumerate}
\end{lemma}
\begin{proof}
 The proof of (a) is a direct application of the identity
  \[
	\sup_{E_{b}\setminus E_{2b}}\left|\nabla_{g_c}^{(k)}\left(\lambda^{-2}\rho_{\lambda}^*\Phi^*g\right)\right |_{g_c} = 
	    \sup_{E_{\lambda b}\setminus E_{2\lambda b}}\lambda^{k}\left|\nabla_{g_c}^{(k)}\left(\Phi^*g\right)\right|_{g_c},
  \]
 valid for any $k$, $\lambda \geq 1$,  and $b > a$.  The uniform equivalence assertion in (b) follows immediately from (a).

To prove the estimate \eqref{eq:epsest} in (b), first we invoke (a) to obtain $b > a+1$ such that
\begin{equation}\label{eq:unifequiv}
	 (1-\epsilon)^2 g_c \leq \Phi^\ast g \leq (1+\epsilon)^{2}g_c
\end{equation}
on $\overline{E_{b-1}}$. Suppose $x=(r,\sigma)\in E_b$. Any curve $\gamma$ in $E_b$ joining $x$ to a point $y=(b,\hat{\sigma})\in\partial E_b$ 
will satisfy that
\[
   (1-\epsilon)\length_{g_c}[\gamma]\leq \length_{\Phi^\ast g}[\gamma] \leq (1+\epsilon)\length_{g_c}[\gamma], 
\]
and so it follows from $\dist_{g_c}(x,\partial E_b)=r-b$ that 
\[
(1-\epsilon)|r-b|\leq\bar{r}_b(x)\leq (1+\epsilon)|r-b|.
\]

Finally, for the curvature estimate in \eqref{eq:acquaddecay1}, fix any $b > a$ and 
note that, according to (a) and the uniform equivalence of $\Phi^\ast g$ and $g_c$ in $E_b$, we have
\[
 \sup_{(r, \sigma)\in E_b}(r^2 + 1) |\Rm(\Phi^\ast g)|_{\Phi^\ast g}(r, \sigma) \leq K, 
\]
for some $K$ depending on $b$ and the curvature of $g_{\Sigma}$.
On the other hand, by (\ref{eq:epsest}), there exists $b'>0$ (independent of $b$) such that for any $x = (r, \sigma) \in E_{b'}$, after possibly enlarging $K$,
\[
 \bar{r}_{b}(x) \leq \bar{r}_{b'}(x) + \operatorname{diam}_{\Phi^\ast g}(E_b\setminus E_{b'})
 \leq K(|r-b'|+ 1). 
\]
Thus, for a still larger $K$,
\[
 \sup_{x \in E_b}(\bar{r}_b^2(x)+1)|\Rm(\Phi^\ast g)|_{\Phi^\ast g}(x) \leq K,
\]
completing the proof.
\end{proof}

\subsection{Reparametrizing an asymptotically conical soliton}
In the next lemma, we will show that a shrinking soliton asymptotic to a cone along some end
admits a reparametrization on that end in which the level sets of the potential function
coincide with those of the radial coordinate.
We include the details since the ends we are working on are incomplete (complete with boundary), 
but we note that there are very precise estimates (see e.g., \cite{CaoZhou}) on the growth of $f$ on arbitrary complete shrinking 
gradient solitons. Given the quadratic decay of the curvature tensor, our situation is actually far simpler,
and an elementary argument 
in the spirit of the first portion of Lemma 1.2 of \cite{Perelman2} suffices. 
 \begin{lemma}\label{lem:fnormalization}
Suppose $(E_{\rad}, g, f)$ is a shrinking soliton satisfying
\begin{equation}\label{eq:fnormqc}
    (\bar{r}^2(x) + 1)|\Rm(g)| \leq K.
\end{equation}
for some $K$, where $\bar{r}(x) \dfn \dist_{g}(x, \partial E_{2\rad})$,
and $\lim_{r_i\to\infty}\bar{r}(r_i, \sigma_i)\to \infty$ for all sequences $(r_i, \sigma_i)\in E_{\rad}$
with $r_i\to\infty$ as $i\to\infty$.
Then there exists $\srad > 0$, a closed $(n-1)$-dimensional manifold $\bar{\Sigma}$, and a map
$\bar{\Phi}: \bar{E}_{\srad} \to E_{\rad}$,
where $\bar{E}_{\srad} \dfn (\srad, \infty)\times \bar{\Sigma}$, with the following
properties:
\begin{enumerate}
 \item[(1)] $\bar{\Phi}$ is a diffeomorphism onto its image, and $\bar{\Phi}(\bar{E}_{\srad})$ is an end of the closure of ${E}_{2\rad}$.
 \item[(2)] For all $(s, \bar{\sigma})\in \bar{E}_{\srad}$,
\[
   \bar{f}(s, \bar{\sigma}) = \frac{s^2}{4},\quad \mbox{and}\quad   \pd{\bar{\Phi}}{s} = 
  \bar{f}^{\frac12}\frac{\bar{\nabla}\bar{f}}{|\bar{\nabla}\bar{f}|^2_{\bar{g}}}.
\]
\item[(3)] There exists a constant $N > 0$ such that, for all $(s, \bar{\sigma}) \in \bar{E}_{2\srad}$,
\[
      N^{-1}(s - 1) \leq \bar{s}(s, \bar{\sigma}) \leq N(s+1)\quad\mbox{and}\quad
(s^2 + 1) |\Rm(\bar{g})|_{\bar{g}}(s, \bar{\sigma})\leq N.
\]
\end{enumerate}
Here, $\bar{f} \dfn f\circ\bar{\Phi}$, $\bar{g} \dfn \bar\Phi^*g$, and $\bar{s}(x) \dfn 
\dist_{\bar{g}}(x, \partial \bar{E}_{2\srad})$.
\end{lemma}

\begin{proof}
 For any  $x\in E_{2\rad}$, if
$\gamma: [0, l] \to E_{2\rad}$ is a unit speed geodesic 
with $\gamma(0) = x_0 \in \partial E_{2\rad}$, $\gamma(l) = x$, and
 $l = \bar{r}(x)$ , then $\gamma([0, l])\subset \overline{E_{2\rad}}$
and,  by \eqref{eq:shrinker} and the assumed quadratic curvature decay, we have
\[
  \frac{1}{2} - \frac{K^{\prime}}{t^2+1} \leq \frac{d^2}{dt^2} (f\circ \gamma)(t) \leq \frac{1}{2} 
  + \frac{K^{\prime}}{t^2+1}
\]
for all $t \in [0, l]$ for some $K^{\prime} = K^{\prime}(n, K)$. So 
\begin{equation}\label{eq:frbarcomp}
       \frac{\bar{r}^2(x)}{4} - N^{\prime}(\bar{r}(x) + 1) \leq f(x) \leq \frac{\bar{r}^2(x)}{4} + N^{\prime}(\bar{r}(x) + 1)
\end{equation}
for some constant $N^{\prime}$ depending on $K^{\prime}$ and $\sup_{\partial E_{2\rad}}(|f| + |\nabla f|)$.  In particular, by 
the second equation in \eqref{eq:shrinker} (and the boundedness of $R$),
it follows that $f$ is proper and $\nabla f \neq 0$ 
on $\overline{E_{\rad^{\prime}}}$ for $\rad^{\prime} > 2\rad$ sufficiently large.

Let $U_{a} \dfn \{\,x\in E_{2\rad}\,|\, f > a \,\}$.  
Then there is $b$ such that $U_{b^{\prime}} \subset E_{\rad^{\prime}}$ for all $b^{\prime} \geq b$,
and a diffeomorphism $\varphi : (b, \infty) \times \bar\Sigma \to U_{b}$ 
for some smooth compact $(n-1)$-dimensional manifold $\bar\Sigma$
diffeomorphic to the level sets $\{\,x\in E_{2\rad}\,|\,f = b^{\prime}\,\}$ for $b^{\prime} \geq b$.  This 
diffeomorphism may be taken to satisfy
\[
 f(\varphi_b(u, \bar\sigma )) = u \quad\mbox{and}\quad \pd{\varphi_b}{ u} 
= \left(\frac{\nabla f}{|\nabla f|^2}\right)
  \circ\varphi_b.
\]
Observe that $\bar{\Sigma}$ must be connected since we assume $\Sigma$ to be.  For suppose $\bar\Sigma = \cup_{i=1}^m \bar\Sigma_i$
for disjoint closed $\bar\Sigma_i$.  Equation \eqref{eq:frbarcomp} implies that $E_{\rad^{\prime\prime}} \subset U_b$
for some $\rad^{\prime\prime} > 0$, and so $\varphi^{-1}_b(E_{\rad^{\prime\prime}})\subset (b, \infty) \times \bar\Sigma_{i_0}$
for some $i_0$, since $E_{\rad^{\prime\prime}}$ is connected.  But, again in view of \eqref{eq:frbarcomp}, 
$ E_{\rad^{\prime\prime}}\cap\varphi_b((b, \infty)\times \bar\Sigma_{i}) \neq \varnothing$ for all $i$.
Thus $\bar\Sigma = \bar\Sigma_{i_0}$ and is connected. 
Note also that $\overline{E_{2\rad}} \setminus U_b$ is a closed subset of 
$\overline{E_{3\rad/2}}\setminus E_{\rad^{\prime\prime}}$, hence compact.   So, taking $\srad \dfn 2\sqrt{b}$,
and defining $\bar{\Phi}:(\srad, \infty) \times \bar\Sigma \to E_{\rad}$ by 
$\bar{\Phi}(s, \bar{\sigma}) \dfn \varphi_b(s^2/4, \bar{\sigma})$, we obtain $\bar\Phi$ satisfying (1) and (2).

For the first inequality in (3), note that, for any $x = (s, \bar\sigma) \in \bar{E}_{\srad}$, it follows from
\eqref{eq:frbarcomp} that
\[
 N^{-1}(s - 1) \leq
  \bar{s}(x) = \dist_{g}(\bar\Phi(x), \partial U_{4b})\leq N(s + 1)
\]
for some $N$ depending on $N^{\prime}$ and $\dist_{g}(\partial E_{2\rad}, \partial U_{4b})$.  The second inequality
in (3) follows directly from \eqref{eq:fnormqc} and \eqref{eq:frbarcomp} after suitably enlarging $N$.
\end{proof}
\subsection{Uniqueness of asymptotically conical models}
Next we wish to determine conditions under which the two cones $(E_0, g_c)$ and $(\hat{E}_0, \hat{g}_c)$
must be isometric if $(M, g)$ is asymptotic to them both along some common end $V\subset M$. 
We will argue broadly as follows: if 
$g$ is asymptotic to $(E_0, g_c)$ and $(\hat{E}_0, \hat{g}_c)$ along $V$, then
 $g_c$ will be asymptotic to $\hat{g}_c$ along some end of $E_0$ in the sense of 
Definition \ref{def:asymcone}.  But then, the asymptotic cones of $(\mathcal{C}, d_{\mathcal{C}})$
and $(\hat{\mathcal{C}}, d_{\hat{\mathcal{C}}})$ (defined in the pointed
Gromov-Hausdorff sense) must be isometric, and these are separately isometric to the original cones.
The following lemma gives the precise (and somewhat more general) statement.

\begin{lemma}\label{lem:uniquenessofasymcones}
 Suppose that $\Phi:\hat{E}_{a_0}\to V$ is a diffeomorphism 
onto some end $V\subset (\overline{E_{b_0}},g)$, and
\[
\lambda^{-2}\rho_{\lambda}^\ast g \rightarrow g_c\quad \mbox{in}\quad C^0_{\mathrm{loc}}(E_{0}, g_c)
\quad\mbox{and}\quad \lambda^{-2}\hat{\rho}_{\lambda}^\ast \Phi^\ast g  \rightarrow \hat{g}_{c}  
\quad\mbox{in}\quad C^0_{\mathrm{loc}}(\hat{E}_{0},\hat{g}_{c})
\]
as $\lambda\to\infty$. Then  $(E_0, g_c)$ and  $(\hat{E}_0, \hat{g}_c)$ are isometric. 
\end{lemma}

\begin{proof} By part (a) of Lemma \ref{lem:asymcone1}, we have
\begin{equation}\label{eq:cone1}
 \lim_{b\to\infty}\|g - g_c\|_{C^0(E_b, g_c)} = \lim_{a\to\infty}\|\Phi^*g - \hat{g}_c\|_{C^0(\hat{E}_a, \hat{g}_c)} = 0,
\end{equation}
and we claim that $\lim_{a\to\infty}\|\Phi^*g_c - \hat{g}_c\|_{C^0(\hat{E}_a, \hat{g}_c)} = 0$ also.  
By the first portion of Lemma \ref{lem:asymcone1} (b), there is a constant $N$ such that
\begin{align*}
  \|\Phi^*g_c - \hat{g}_c\|_{C^0(\hat{E}_a, \hat{g}_c)} 
&\leq  N\|\Phi^*(g_c - g)\|_{C^0(\hat{E}_a, \Phi^*g)}
+ \|\Phi^*g - \hat{g}_c\|_{C^0(\hat{E}_a, \hat{g}_c)}\\
&\leq N\|g_c - g\|_{C^0(\Phi(\hat{E}_a), g)}
+ \|\Phi^*g - \hat{g}_c\|_{C^0(\hat{E}_a, \hat{g}_c)},
\end{align*}
so in view of \eqref{eq:cone1}, we only need to verify that, for all $b > 0$, there exists $a$ sufficiently large
such that $\Phi(\hat{E}_{a}) \subset E_{b}$.  If not, then there exists $b_1 > b_0$ and a sequence of points 
$\hat{x}_j=(\hat{r}_j,\hat{\sigma}_j)\in \hat{E}_{2a_0}$ such that $\hat{r}_j\to\infty$ 
while $\Phi(\hat{x}_j)\in V\setminus E_{b_1}$ for all $j$. 
Then, by the second portion of Lemma \ref{lem:asymcone1} (b),
\[
\dist_g(\Phi(\hat{x}_j),\Phi(\partial\hat{E}_{2a_0}))=\dist_{\Phi^\ast g}(\hat{x}_j,\partial\hat{E}_{2a_0})\to\infty,
\]
as $j\to\infty$ which gives a contradiction.

In fact, by Lemma \ref{lem:asymcone1} (a), we know that $\lambda^{-2}\hat{\rho}_{\lambda}^*\Phi^*g_c \to \hat{g}_c$ 
in $C^0_{\mathrm{loc}}(\hat{E_{0}}, \hat{g}_c)$
as $\lambda\to\infty$, and this is equivalent to the assertion that $\Phi_{\lambda}^*g_c \to \hat{g}_c$ 
in $C^0_{\mathrm{loc}}(\hat{E_{0}}, \hat{g}_c)$
where $\Phi_{\lambda} \dfn \rho_{\lambda^{-1}} \circ\Phi \circ \hat{\rho}_{\lambda}$. 
We write $\Phi_{\lambda}(\hat{x}) = (r_{\lambda}(\hat{x}), \sigma_{\lambda}(\hat{x}))$
for $\hat{x}\in \hat{E}_{\lambda^{-1}a_0}$.

Now, applying the second assertion of Lemma \ref{lem:asymcone1} (b) to $\Phi^*g_c$ and $\hat{g}_c$ for some sufficiently 
large $b_2$, we claim that we have $E_{4b_2} \subset V = \Phi(\hat{E}_{a_0})$.
To see this, observe that, since $V$ is an end of $\overline{E_{b_0}}$,
$V$ is the unique unbounded connected component of  $\overline{E_{b_0}}\setminus \Omega$ for some compact $\Omega$.
Thus, there exists $b_2$ such that $E_{4b_2} \cap \Omega = \varnothing$, and since $E_{4b_2}$ is connected and unbounded,
we must have $E_{4b_2}\subset V$.

Next, observe that, by Lemma \ref{lem:asymcone1} (b), for all $\epsilon > 0$, there exists $a = a(\epsilon)>a_0$ such that whenever 
$\hat{x} = (\hat{r}, \hat{\sigma})\in\hat{E}_{a}$, the inequality
\[
(1-\epsilon)|\hat{r}-a| \leq\dist_{g_c}(\Phi(\hat{r},\hat{\sigma}), \Phi(\partial \hat{E}_a))\leq 
(1+\epsilon)\left|\hat{r}-a\right|
\]
holds.
Using that $\lambda^{-1} d_{g_c}(x, y) = d_{g_c}(\rho_{\lambda^{-1}}(x), \rho_{\lambda^{-1}}(y))$ for $\lambda > 0$, we then have
\[
 (1-\epsilon)\left|\hat{r}-\frac{a}{\lambda}\right|
\leq\dist_{g_c}(\Phi_{\lambda}(\hat{r},\hat{\sigma}), (\rho_{\lambda^{-1}}\circ\Phi)(\partial\hat{E}_{a}))\leq (1+\epsilon)\left|\hat{r}-\frac{a}{\lambda}\right|,
\]
for all $\lambda \geq 1$ and $(\hat{r}, \hat{\sigma})\in \hat{E}_{\lambda^{-1}a}$.
By the compactness of $\Phi(\partial\hat{E}_a)$, we have 
$d_{\mathcal{C}}((\rho_{\lambda^{-1}}\circ\Phi)(\partial\hat{E}_a), \mathcal{O}) \leq C\lambda^{-1}$
for some constant $C$ independent of $\lambda$,
and it follows that $r_{\lambda}(\hat{r}, \hat{\sigma})/ \hat{r} 
= d_{\mathcal{C}}(\Phi_{\lambda}(\hat{r}, \hat{\sigma}), \mathcal{O})/ \hat{r}$
converges uniformly to $1$ as $\lambda \to \infty$ on $\hat{E}_{a^{\prime}}$ for any fixed $a^{\prime} > 0$.
In particular, there is $a_1 > a_0$ and $\lambda_0 > 0$ such that $\Phi_{\lambda}(\hat{r}, \cdot) \in E_{\hat{r}/2}\setminus E_{2\hat{r}}$
whenever $\lambda\geq \lambda_0$ and $\hat{r}\geq a_1 /\lambda$.  

With this and the local uniform convergence of $\Phi_{\lambda}^*g_c$ to $\hat{g}_c$, we can then find a sequence
$\{\lambda_i\}_{i=1}^{\infty}$ such that $\Phi_{\lambda_i}(\hat{r}, \cdot) \in E_{\hat{r}/2}\setminus E_{2\hat{r}}$
and
\[
|d_{\hat{g}_c}(\hat{x}_1, \hat{x}_2) - d_{g_c}(\Phi_{\lambda_i}(\hat{x}_1), \Phi_{\lambda_i}(\hat{x}_2))| \leq \frac{N_0}{i} 
\]
on $\hat{E}_{1/4i}\setminus\hat{E}_{4i}$ for some $N_0$ depending only on the diameters
of $(\Sigma, g_{\Sigma})$ and $(\hat{\Sigma}, g_{\hat{\Sigma}})$.  We define a sequence of maps 
$F_i: (B_{\hat{\mathcal{C}}}(\hat{\mathcal{O}}, i), d_{\hat{\mathcal{C}}}) \to
(\mathcal{C}, d_{\mathcal{C}})$
by
\[
      F_i \dfn \left\{\begin{array}{ll}
			\Phi_{\lambda_i} &\mbox{on} \quad \hat{E}_{1/i}\setminus \hat{E}_{i}\\
			\mathcal{O} & \mbox{on} \quad \hat{\mathcal{C}} \setminus \hat{E}_{1/i}.
                          \end{array}\right.
\]
Using the $F_i$ in conjunction with the convergence of $r_{\lambda_i}(\hat{r}, \hat{\sigma})$
to $\hat{r}$ and the distance comparison above, the constant sequence 
$\{(\hat{\mathcal{C}}, d_{\hat{\mathcal{C}}}, \hat{O})\}_{i=1}^{\infty}$
can be seen to converge to $(\mathcal{C}, d_{\mathcal{C}}, O)$ in the pointed Gromov-Hausdorff sense, and it follows (see, e.g,
Theorem 8.1.7 in \cite{BuragoBuragoIvanov}) that there exists a pointed isometry 
$\varphi: (\hat{\mathcal{C}}, d_{\hat{\mathcal{C}}}, \hat{O})
\to (\mathcal{C}, d_{\mathcal{C}}, O)$.  The classical theorem of Calabi-Hartman \cite{CalabiHartman} then gives that the restriction
of $\varphi$ to $\hat{E}_0$ must in fact be a smooth isometry between $(\hat{E}_0, \hat{g}_c)$ and $(E_0, g_c)$.
\end{proof}

\section{Existence of rotationally symmetric shrinking ends}
\label{app:rotsymend}
In this appendix, we construct (incomplete) rotationally symmetric gradient Ricci solitons on topological
half-cylinders asymptotic to prescribed rotationally symmetric cones.  Our construction is based on the analysis
of a system of ODE which has been carefully treated, particularly in the steady and expanding cases, in the unpublished notes \cite{BryantSoliton} of Bryant;  
the argument we present below is heavily indebted to that reference.

Let $g_{S^{n-1}}$ be the standard round metric on the sphere $S^{n-1}$ of constant sectional curvature $1$ and, for $0<\rad<\widetilde{\rad}$, 
consider the warped product metric $g=dr^2+a(r)^2g_{S^{n-1}}$ 
on the annulus $A(\rad,\widetilde{\rad})\dfn (\rad,\widetilde{\rad})\times S^{n-1}$. The Ricci curvature tensor of $g$ is given by 
\begin{equation*}
\Rc(g)=-(n-1)\frac{a''}{a} dr^2+\left[(n-2)-aa''-(n-2)(a')^2\right]g_{S^{n-1}},
\end{equation*}
and the hessian of an arbitrary radial function $f = f(r)$ relative to $g$
has the form
\begin{equation*}
\nabla\nabla f=f'' dr^2+a a' f' g_{S^{n-1}} 
\end{equation*}
where the prime denotes differentiation with respect to $r$,
Thus, $(A(\rad,\widetilde{\rad}),g,f)$ satisfies \eqref{eq:shrinker} if and only if $a$ and $f$ satisfy the system
\begin{align}
\label{eq:SolitonA}
\left\{
\begin{array}{rl}
2f'' & = 1+2(n-1)\frac{a''}{a} \\
2a a' f' & = a^2-2\left[(n-2)-aa''-(n-2)(a')^2\right]
\end{array}
\right.
\end{align}
with $a(r) > 0$ for $r\in (\rad, \widetilde{\rad})$.

Given $\alpha\in (0,1) \cup (1,\infty)$, we seek to find solutions of the system (\ref{eq:SolitonA}) with $\rad > 0$ 
and $\widetilde{\rad} = \infty$ that satisfy $a(r) > 0$ and the asymptotic conditions
\begin{equation}\label{eq:asymcond} 
\frac{a(r)}{r} \to  \sqrt{\alpha},\quad\mbox{and}\quad \frac{4f(r)}{r^2} \to 1 \quad\mbox{as}\quad r\to\infty. 
\end{equation}
We will be working exclusively in the region where $a' > 0$, and there (following \cite{BryantSoliton})
it is convenient to change the radial 
coordinate from $r$ to $a(r)$.  In terms of $a$, $g$ assumes the form
\begin{equation*}
g=\frac{da^2}{w(a^2)}+a^2 g_{S^{n-1}},
\end{equation*}
where $a' (r)=\sqrt{w(a^2(r))}$, and (\ref{eq:SolitonA}) becomes
\begin{align}
\label{eq:Soliton1A}
\left\{
\begin{array}{rl}
1+2(n-1)w' & = 8s w f'' +4w f' +4sw' f' \\
4s w f'  &= s-2\left[(n-2)-sw'-(n-2)w\right], 
\end{array}
\right.
\end{align}
where $s=a^2$ and the prime now represents differentiation with respect to $s$.
We can now substitute the second equation in \eqref{eq:Soliton1A} into the first to eliminate $f$ and obtain
a single second-order equation for $w$.
\begin{equation}
\label{eq:MetricSolitonA}
4s^2ww''-\left[2sw'+s-2(n-2)\right]sw'+2(n-2)(1-w)w=0.
\end{equation}

\begin{proposition}
\label{prop:ExistenceA}
Given $\alpha\in (0,1)\cup (1,\infty)$ and $n\ge 2$, there exists $\srad=\srad(n,\alpha)>0$ and a positive solution $w$ of the equation (\ref{eq:MetricSolitonA}) 
on the interval $s>\srad$ such that $\lim_{s\to\infty}w(s)=\alpha$. In fact, $w$ has the asymptotic expansion 
\begin{equation*}
w(s)=\alpha-\frac{2(n-2)\alpha(1-\alpha)}{s}+\varphi(s)
\end{equation*}
where $\varphi(s)=O(s^{-2})$, $\varphi'(s)=O(s^{-3})$, and $\varphi''(s)=O(s^{-3})$. 
Furthermore, up to an additive constant, the function $f$ satisfies the expansion
\begin{equation*}
f(s)=\frac{s}{4\alpha}+\psi(s)
\end{equation*}
where $\psi(s)=O(s^{-1})$, $\psi'(s)=O(s^{-2})$ and $\psi''(s)=O(s^{-3})$.
\end{proposition}
\begin{remark}
 The case $n=2$ was proven in Section 5 of \cite{Ramos}; see also \cite{BernsteinMettler}. 
\end{remark}
\begin{proof}
Through the rest of the proof, we fix $\alpha\in (0,1) \cup (1,\infty)$ and suppose $n\ge 3$. Our strategy is to seek to obtain solutions of (\ref{eq:MetricSolitonA}) 
with the desired asymptotic 
behavior as limits of sequences of solutions to \eqref{eq:MetricSolitonA} on finite intervals satisfying appropriate initial conditions.

Given $\srad_0>1$, the local theory of ODE implies that there is $\srad_1 \in [0, \srad_0)$ and a unique solution, $w_{\srad_0}(s)$,
of (\ref{eq:MetricSolitonA}) on $(\srad_1, \srad_0]$ with initial conditions $w_{\srad_0}(\srad_0)=\alpha$ and 
$w_{\srad_0}'(\srad_0)=0$, such that $(S_1, S_0]$ is the maximal subinterval of $(0, S_0]$ on which $w_{\srad_0}$ exists
and is positive.
Note that, if $0<\alpha<1$, $w_{\srad_0}''(\srad_0)<0$ by (\ref{eq:MetricSolitonA}) so $w_{\srad_0}$ 
is increasing in some interval $(\srad_0-\delta,\srad_0)$. Moreover, by the strong maximum principle, 
there are no local minimum points in the strip $\{0<w_{\srad_0}(s)<1,\, s>0\}$. Thus $w_{\srad_0}$ is increasing in the interval $(\srad_1,\srad_0)$ for $\alpha\in (0,1)$.  A similar argument gives that $w_{\srad_0}$ is decreasing in the interval $(\srad_1,\srad_0)$ for $\alpha\in (1,\infty)$.

Next we define $\srad_2\dfn\inf\{s\in(\srad_1,\srad_0): \alpha/2\le w_{\srad_0}(s)\le 2\alpha\}$.
It follows from the monotonicity of $w_{\srad_0}$ that $\alpha/2\le w_{\srad_0}(s)\le 2\alpha$ for $\srad_2<s\le\srad_0$,
and so the equation (\ref{eq:MetricSolitonA}) implies that
\begin{equation}
\label{eq:InteFactorA}
\begin{split}
& \frac{d}{ds}\left\{\exp\left(-\int^s \frac{w_{\srad_0}'(\rho)}{2w_{\srad_0}(\rho)}+
\frac{1}{4w_{\srad_0}(\rho)}-\frac{n-2}{2\rho w_{\srad_0}(\rho)}d\rho\right) w_{\srad_0}'(s)\right\}\\
= & \frac{n-2}{2s^2} (w_{\srad_0}(s)-1)\exp\left(-\int^s \frac{w_{\srad_0}'(\rho)}{2w_{\srad_0}(\rho)}+
\frac{1}{4w_{\srad_0}(\rho)}-\frac{n-2}{2\rho w_{\srad_0}(\rho)}d\rho\right)
\end{split}
\end{equation}
on $(\srad_2, \srad_0)$
Assume that $\srad_0>4(n-2)$. Integrating (\ref{eq:InteFactorA}) with respect to $s$, we have that if $\max\{4(n-2),\srad_2\}<s<\srad_0$, 
then
\begin{align*}
 \abs{w_{\srad_0}'(s)}
\le  & \int_{s}^{\srad_0} \!\frac{n-2}{2\sigma^2}\sqrt{\frac{w_{\srad_0}(s)}{w_{\srad_0}(\sigma)}} \abs{1-w_{\srad_0}(\sigma)}
\exp\left(-\int^\sigma_s \left(\frac{1}{4}-\frac{n-2}{2\rho}\right)\frac{d\rho}{w_{\srad_0}(\rho)}\right)d\sigma\\
\le & \frac{(n-2)(1+2\alpha)}{s^2}\int_s^{\srad_0} \exp\left(-\frac{\sigma-s}{16\alpha}\right)d\sigma.
\end{align*}
Hence there exists $N\dfn N(n,\alpha)>0$ such that
\begin{equation}
\label{eq:1stderivativeA}
\abs{w_{\srad_0}'(s)}\le Ns^{-2} \  \text{for} \  \max\{4(n-2),\srad_2\}<s\le\srad_0.
\end{equation} 

Furthermore, we may obtain a uniform upper bound for $\srad_2$ independent of $\srad_0$: if $\srad_2\ge 4(n-2)$, 
integrating (\ref{eq:1stderivativeA}) from $\srad_2$ to $\srad_0$ implies that
\begin{equation*}
\abs{w_{\srad_0}(\srad_0)-w_{\srad_0}(\srad_2)} \le N\left(\frac{1}{\srad_2}-\frac{1}{\srad_0}\right).
\end{equation*}
Note that either $w_{\srad_0}(\srad_2)=\alpha/2$ if $\alpha\in (0,1)$ or $w_{\srad_0}(\srad_2)=2\alpha$ if $\alpha\in (1,\infty)$. 
Otherwise, we have $\srad_1 = \srad_2 \ge 4(n-2) > 0$, which, in view of \eqref{eq:1stderivativeA}, violates the maximality of the interval $(\srad_1, \srad_0)$. 
Since $w_{\srad_0}(\srad_0)=\alpha$, we see that either way we have $\srad_2\le 2N/\alpha$, and hence that $\srad_2\le\min\{4(n-2), 2N/\alpha\}$.
Thus, letting $\srad_0\to\infty$, we obtain subsequential convergence of $w_{\srad_0}$ to a positive solution $w$ of (\ref{eq:MetricSolitonA}) 
in $(\srad,\infty)$ satisfying 
\[
\lim_{s\to\infty}w(s)=\alpha,\quad w(s)- \alpha =O(s^{-1}),\quad w'(s)=O(s^{-2}),\quad \mbox{and}\quad w''(s)=O(s^{-2}). 
\]
Here $\srad$ is defined to be $\min\{4(n-2),2N/\alpha\}$.  

Next, define a function $\varphi(s)$ by
\begin{equation*}
w(s)=\alpha-\frac{2(n-2)\alpha(1-\alpha)}{s}+\varphi(s).
\end{equation*}
The second term in this equation is chosen after formally expanding $w(s)$ in a power series in terms of $s^{-k}$ and solving for the coefficient of the $s^{-1}$ term. 
On one hand, the asymptotics of $w$ imply that $\varphi(s)=O(s^{-1})$, 
$\varphi'(s)=O(s^{-2})$, and $\varphi''(s)=O(s^{-2})$. On the other hand, using (\ref{eq:MetricSolitonA}), we see that the function $\varphi$ satisfies the equation
\begin{equation*}
\varphi''(s)-\frac{\varphi'(s)}{4\alpha}=Q(s),
\end{equation*}
where $Q(s)=O(s^{-3})$. 
Thus we have
 \begin{equation*}
 \varphi'(s)=e^{\frac{s}{4\alpha}}\int_s^\infty Q(\sigma)e^{-\frac{\sigma}{4\alpha}}d\sigma,
 \end{equation*}
 so $\varphi'(s)=O(s^{-3})$, and hence also  $\varphi(s)=O(s^{-2})$ and $\varphi''(s)=O(s^{-3})$.

It remains to derive the asymptotic expansion of the function $f$.  From the second equation in (\ref{eq:Soliton1A}), we have
\begin{align*}
\frac{d}{ds}\left(f(s)-\frac{s}{4\alpha}\right) w(s) = & \frac{\alpha-w(s)}{4\alpha}-\frac{1}{2s}\left[(n-2)(1-w(s))-sw'(s)\right]\\
= & \frac{(n-2)(n-1)\alpha(1-\alpha)}{s^2}+\frac{(n-2)\varphi(s)}{2s}-\frac{\varphi(s)}{4\alpha}+\frac{\varphi'(s)}{2}\\
\dfn & \psi'(s)w(s)
\end{align*}
Since $\alpha/2\le w(s)\le 2\alpha$ for $s>\srad$, it follows that $\psi'(s)=O(s^{-2})$ and $\psi''(s)=O(s^{-3})$. Moreover, we may assume that $\lim_{s\to\infty}\psi(s)=0$ and 
so achieve that $\psi(s)=O(s^{-1})$. 
\end{proof}
Again invoking the results of \cite{BernsteinMettler} and \cite{Ramos}  for the case $n=2$, Proposition \ref{prop:ExistenceA} can be restated 
to yield the following existence theorem.
  
  \begin{proposition}
  \label{prop:Existence1A}
  For each $\alpha\in (0,1) \cup (1,\infty)$ and $n\ge 2$, there exists an rotationally symmetric shrinking gradient Ricci soliton asymptotic to the rotationally symmetric cone 
 $((0,\infty)\times S^{n-1}, dr^2+\alpha r^2 g_{S^{n-1}})$ in the sense of Definition \ref{def:asymcone}.
  \end{proposition}
 By Theorem \ref{thm:unique}, the maximal extensions of the metrics constructed above are the unique rotationally symmetric examples asymptotic
to the given  cone, however,
  according to the classification in \cite{KotschwarRotationalSoliton}, 
   none of these extensions yield  complete metrics on $\RR^n$ or $\RR\times S^{n-1}$. 
\begin{acknowledgement*} The authors wish to thank Ben Chow, Toby Colding, Bill Minicozzi, and Lei Ni 
for their support and for sharing some of their intuition, and also Huai-Dong Cao for his comments on a preliminary draft
of the paper. The second author would also like to thank FIM/ETH for their hospitality during her visit in December 2012  while this project was underway. 
\end{acknowledgement*}


\begin{thebibliography}{99}
 \bibitem{Alexakis} Alexakis, Spyros.  
   ``Unique continuation for the vacuum Einstein equations.'' 
    {\tt arXiv:0902.1131 [gr-qc]}.

\bibitem{BohmWilking} B\"ohm, Christoph; Wilking, Burkhard.
``Manifolds with positive curvature operators are space forms.''
\textit{Ann. of Math.} (2) 167 (2008), no. 3, 1079--1097. 

\bibitem{Bando}  Bando, Shigetoshi.  
 ``Real analyticity of solutions of Hamilton's equation.''
  \textit{Math. Z.} 195 (1987), no. 1, 93--97. 

\bibitem{BernsteinMettler} Bernstein, Jacob; Mettler, Thomas.
 ``Two-dimensional gradient Ricci solitons revisited."
   {\tt arXiv: 1303.6854 [math.DG].}

\bibitem{BrendleBryant3D} Brendle, Simon.
``Rotational symmetry of self-similar solutions to the Ricci flow.''
 \textit{Inv. Math.} Online first: {\tt http://dx.doi.org/10.1007/s00222-013-0457-0}.

\bibitem{BrendleBryantND} Brendle, Simon.
``Rotational symmetry of Ricci solitons in higher dimensions.''
 {\tt arXiv: 1203.0270 [math.DG].}

\bibitem{BryantSoliton} Bryant, Robert L.
``Ricci flow solitons in dimension three with $\operatorname{SO}(3)$-symmetries.''
Available at {\tt http://www.math.duke.edu/$\sim$bryant/3DRotSymRicciSolitons.pdf}.

\bibitem{BuragoBuragoIvanov} Burago, Dmitri; Burago, Yuri; Ivanov, Sergei.
\emph{A course in metric geometry.}
Graduate Studies in Mathematics, 33. American Mathematical Society, Providence, RI, 2001. xiv+415 pp. 

\bibitem{CalabiHartman} Calabi, Eugenio; Hartman, Philip.
 ``On the smoothness of isometries.''
\textit{Duke Math. J.} 37 (1970) 741--750.

\bibitem{CaoExpander} Cao, Huai-Dong. 
  ``Limits of solutions to the K\"ahler-Ricci flow.''
  \textit{J. Differential Geom.} 45 (1997), no. 2, 257--272.


\bibitem{CaoSurvey1} Cao, Huai-Dong.
``Recent progress on Ricci solitons.''
 Recent advances in geometric analysis, 1--38,
\textit{Adv. Lect. Math.}, 11, Int. Press, Somerville, MA, 2010.  

\bibitem{CaoSurvey2} Cao, Huai-Dong
``Geometry of complete gradient shrinking Ricci solitons.''
Geometry and analysis. No. 1, 227--246,
\textit{Adv. Lect. Math.}, 17, Int. Press, Somerville, MA, 2011. 

\bibitem{CaoChen} Cao, Huai-Dong; Chen, Qiang.
``On Bach-flat gradient shrinking Ricci solitons.''
 \textit{ Duke Math. J.} 162 (2013), no. 6, 1149--1169. 

 \bibitem{CaoChenZhu} Cao, Huai-Dong; Chen, Bing-Long; Zhu, Xi-Ping.
 ``Recent developments on Hamilton’s Ricci flow.''
 \textit{Surveys in differential geometry}, Vol. XII, 47–112, Surv. Differ. Geom., XII, Int. Press,
 Somerville, MA, 2008.

\bibitem{CaoZhou} Cao, Huai-Dong; Zhou, Detang.
``On complete gradient shrinking Ricci solitons.''
\textit{J. Differential Geom.} 85 (2010), no. 2, 175--185. 

\bibitem{CaoWangZhang} Cao, Xiaodong; Wang, Biao; Zhang, Zhou.
``On locally conformally flat gradient shrinking Ricci solitons.''
\emph{Commun. Contemp. Math.} 13 (2011), no. 2, 269--282. 

\bibitem{CarrilloNi} Carrillo, Jos\'e A.; Ni, Lei.
``Sharp logarithmic Sobolev inequalities on gradient solitons and applications.''
\textit{Comm. Anal. Geom.} 17 (2009), no. 4, 721--753. 

\bibitem{CheegerColding} Cheeger, Jeff; Colding, Tobias H. 
  ``Lower Bounds on Ricci Curvature and the Almost Rigidity of Warped Products.''
   \textit{Ann. of Math.} (2) 144 (1996), no. 1, 189--237.

\bibitem{ChenStrongUniqueness} Chen, Bing-Long.
  ``Strong uniqueness of the Ricci flow.'' 
  \textit{J. Differential Geom. 82} (2009), no. 2, 363–382. 


\bibitem{ChodoshExpanders} Chodosh, Otis.
  ``Expanding Ricci solitons asymptotic to cones.''
  {\tt arXiv:1303.2983 [math.DG]}.

\bibitem{ChodoshFong} Chodosh, Otis; Fong, Frederick Tsz-Ho.
   ``Rotational symmetry of conical K\"ahler-Ricci solitons.''
  {\tt arXiv:1304.0277 [math.DG]}.

   
\bibitem{ChowKnopf} Chow, Bennett; Knopf, Dan.
  \emph{The Ricci flow: an introduction.}
   Mathematical Surveys and Monographs, 110. American Mathematical Society, Providence, RI, 2004. xii+325 pp.   
   
\bibitem{ChowLuNi} Chow, Bennett; Lu, Peng; Ni, Lei.
  \emph{Hamilton's Ricci flow.}
   Graduate Studies in Mathematics, 77. American Mathematical Society, Providence, RI; Science Press, New York, 2006. xxxvi+608 pp.  

\bibitem{DancerWang} Dancer, Andrew S; Wang, McKenzie Y.
``On Ricci solitons of cohomogeneity one.''
\textit{Ann. Global Anal. Geom.} 39 (2011), no. 3, 259--292. 

\bibitem{DeTurck} DeTurck, Dennis M. 
  ``Deforming metrics in the direction of their Ricci tensors, improved version.'' 
\textit{Collected Papers on Ricci Flow}, ed. Cao, H.-D.;  Chow, B.; Chu, S.-C.; Yau, S.-T. 
  Internat. Press, Somerville, MA (2003).

\bibitem{EminentiLaNaveMantegazza} Eminenti, Manolo; La Nave, Gabriele; Mantegazza, Carlo.
``Ricci solitons: the equation point of view.''
\emph{Manuscripta Math.} 127 (2008), no. 3, 345--367. 

\bibitem{EscauriazaFernandez} Escauriaza, Luis; Fern\'{a}ndez, Francisco Javier.
  ``Unique continuation for parabolic operators.''
   \textit{Ark. Mat.} 41 (2003), no. 1, 35--60.

\bibitem{EscauriazaSereginSverak} Escauriaza, Luis; Seregin, Gregory; \v{S}ver\'{a}k, Vladimir.
  ``Backward uniqueness for parabolic equations.''
   \textit{Arch. Ration. Mech. Anal.} 169 (2003), no. 2, 147--157.   

\bibitem{FeldmanIlmanenKnopf} Feldman, Mikhail; Ilmanen, Tom; Knopf, Dan.
``Rotationally symmetric shrinking and expanding gradient K\"ahler-Ricci solitons.''
\textit{J. Differential Geom.} 65 (2003), no. 2, 169--209. 

\bibitem{FernandezLopezGarciaRio}  Fern\'andez-L\'opez, Manuel; Garc\'ia-R\'io, Eduardo. 
``Some gap theorems for gradient Ricci solitons.'' \textit{Internat. J. Math.} 23 (2012), no. 7, 
1250072, 9 pp.

\bibitem{HamiltonSurfaces} Hamilton, Richard S.
``The Ricci flow on surfaces.'' Mathematics and general relativity (Santa Cruz, CA, 1986), 237--262,
\textit{Contemp. Math.}, 71, Amer. Math. Soc., Providence, RI, 1988.   

\bibitem{HamiltonSingularities} Hamilton, Richard S.  
  ``The formation of singularities in Ricci flow,''
   \textit{Surveys in differential geometry, Vol. II} (Cambridge, MA, 1993), 7--136, Int. Press, Cambridge, MA, 1995. 

\bibitem{Ivey3DSolitons} Ivey, Thomas.
``Ricci solitons on compact three-manifolds.''
 \textit{Differential Geom. Appl.} 3 (1993), no. 4, 301--307.

\bibitem{IveyLocalSoliton} Ivey, Thomas A.
``Local existence of Ricci solitons.''
\textit{Manuscripta Math.} 91 (1996), no. 2, 151--162.  
   
\bibitem{KobayashiNomizu} Kobayashi, Shoshichi; Nomizu, Katsumi.
   \textit{Foundations of differential geometry. Vol. I.}
   Interscience Publishers, a division of John Wiley \& Sons, New York-London, 1963 xi+329 pp.   
   
\bibitem{KotschwarBackwardsUniqueness} Kotschwar, Brett. 
  ``Backwards uniqueness for the Ricci flow.'' 
  \textit{Int. Math. Res. Not.} (2010) no. 21, 4064--4097.
  
\bibitem{KotschwarRotationalSoliton} Kotschwar, Brett.
  ``On rotationally invariant shrinking Ricci solitons.''
   \textit{Pacific J. Math.} 236 (2008), no. 1, 73--88. 

\bibitem{MunteanuSesum} Munteanu, Ovidiu; Sesum, Natasa.
``On gradient Ricci solitons.'' \emph{J. Geom. Anal.} 23 (2013), no. 2, 539--561. 

\bibitem{MunteanuWang} Munteanu, Ovidiu; Wang, Jiaping.
``Analysis of weighted Laplacian and applications to Ricci solitons.''
\textit{Comm. Anal. Geom.} 20 (2012), no. 1, 55--94. 

\bibitem{Myers} Myers, Sumner Byron.
 ``Riemannian manifolds in the large."
   \textit{Duke Math. J.} 1 (1935), no. 1, 39--49. 

\bibitem{Naber4D} Naber, Aaron.
``Noncompact shrinking four solitons with nonnegative curvature.''
\textit{J. Reine Angew. Math.} 645 (2010), 125--153. 


\bibitem{Nguyen} Nguyen, Tu A.
``On a question of Landis and Oleinik.''
\textit{Trans. Amer. Math. Soc.} 362 (2010), no. 6, 2875--2899. 


\bibitem{NiWallach} Ni, Lei; Wallach, Nolan.
``On a classification of gradient shrinking solitons.''
\textit{Math. Res. Lett.}, 15(5), (2010), 941--955.



\bibitem{Perelman1} Perelman, Grigory.  
``The entropy formula for the Ricci flow and its geometric applications.'' 
{\tt arXiv:math/0211159 [math.DG]}.


\bibitem{Perelman2} Perelman, Grigory.  
``Ricci flow with surgery on three-manifolds.'' 
{\tt arXiv:math/0303109 [math.DG]}


\bibitem{PetersenWylie} Petersen, Peter; Wylie, William.
``On the classification of gradient Ricci solitons.''
{\textit Geom. Topol.} 14 (2010), no. 4, 2277--2300. 


\bibitem{Ramos} Ramos, Daniel.
 ``Gradient Ricci solitons on surfaces."
{\tt arXiv:1304.6391 [math.DG].}   


\bibitem{Shi} Shi, Wan-Xiong. 
  ``Deforming the metric on complete Riemannian manifolds.''  
  \textit{J. Differential Geom.}  30  (1989),  no. 1, 223--301.


\bibitem{Wang} Wang, Lu.
  ``Uniqueness of Self-similar Shrinkers with Asymptotically Conical Ends.''
   {\tt arXiv:1110.0450 [math.DG]},


\bibitem{WongYu}  Wong, Willie Wai-Yeung; Yu, Pin.
  ``On strong unique continuation of coupled Einstein metrics.''
  \textit{Int. Math. Res. Not.} (2012), no. 3, 544--560. 


 \bibitem{Yang} Yang, Bo.
 ``A characterization of noncompact Koiso-type solitons.''
 \textit{Internat. J. Math.} 23 (2012), no. 5, 1250054, 13 pp. 


\bibitem{ZhangCompleteness} Zhang, Zhu-Hong.
``On the completeness of gradient Ricci solitons.''
\textit{Proc. Amer. Math. Soc.} 137 (2009), no. 8, 2755--2759. 


\bibitem{ZhangWeyl} Zhang, Zhu-Hong.
``Gradient shrinking solitons with vanishing Weyl tensor.''
\textit{Pacific J. Math.} 242 (2009), no. 1, 189--200. 

\end{thebibliography}
\end{document}